\documentclass[final]{siamltex}
\usepackage{epsfig,amssymb,amsmath,version}
\usepackage{amssymb,version,graphicx,fancybox,mathrsfs,multirow}

\usepackage{url,hyperref}
\usepackage[notcite,notref]{showkeys}
\usepackage{algorithm,algorithmic}
 \usepackage{caption}
  \usepackage{subcaption}
\usepackage{color}
\usepackage{cases}
\usepackage{mathtools}
\allowdisplaybreaks

\newtheorem{remark}{Remark}[section]
\newcommand{\al}{\alpha}
\newcommand{\fy}{\varphi}
\renewcommand{\d}{{\rm d}}
\def\Dal{{\partial_t^\al}}

\def\Om{\Omega}
\def\II{(\Om)}

\def\dH#1{\dot H^{#1}(\Omega)}

\def\bDal{\bar\partial_\tau^\alpha}

\def \contour{{\Gamma_{\theta,\sigma}}}
\def\L2Om{{L^2(\Omega)}}
\def\FF{\mathcal{F}}

\def\t#1{\tilde{#1}}
\def \E{E}
\def\GG{\mathcal{G}}
\def\tu{\tilde{u}}
\def\tud{\tilde{u}^\delta}
\def\I{\mathcal{I}}
\def \vecfunc#1#2{\begin{bmatrix}
#1(#2)\\
\partial_t #1(#2)
\end{bmatrix}}

\input{siam.sty}

\title{Backward diffusion-wave problem: stability, regularization and approximation}

\author{Zhengqi Zhang\thanks{Department of Applied Mathematics, The Hong Kong Polytechnic University, Hung Hom, Hong Kong. Email address: zhengqi.zhang@connect.polyu.hk}
\and Zhi Zhou\thanks{Department of Applied Mathematics, The Hong Kong Polytechnic University, Hung Hom, Hong Kong. Email address: zhizhou@polyu.edu.hk}}
\date{\today}

\begin{document}

\maketitle

\setlength\abovedisplayskip{5pt}
\setlength\belowdisplayskip{5pt}

\begin{abstract}
We aim at the development and analysis of the numerical schemes for approximately solving the backward diffusion-wave problem,
which involves a fractional derivative in time with order $\alpha\in(1,2)$.
From terminal observations at two time levels, i.e., $u(T_1)$ and $u(T_2)$,
we simultaneously recover two initial data $u(0)$ and $u_t(0)$ and hence the solution $u(t)$ for all $t > 0$.
First of all, existence, uniqueness and Lipschitz stability of the backward diffusion-wave problem were established
under some conditions about $T_1$ and $T_2$.
Moreover, for noisy data, we propose a quasi-boundary value scheme to regularize the "mildly" ill-posed problem,
and show the convergence of the regularized solution.
Next, to numerically solve the regularized problem, a fully discrete scheme is proposed
by applying finite element method in space and convolution quadrature in time.
We establish error bounds of the discrete solution in both cases of smooth and nonsmooth data.
The error estimate is very useful in practice since it indicates the way
to choose discretization parameters and regularization parameter, according to the noise level.
The theoretical results are supported by numerical experiments.
\end{abstract}



\pagestyle{myheadings}
\thispagestyle{plain}
\section{Introduction}
Let $\Omega\subset\mathbb{R}^d $ ($d=1,2,3$) be a
convex polyhedral domain
with boundary $\partial\Omega$. We consider
the following initial-boundary value problem of diffusion-wave equation with $\alpha\in(1,2)$
\begin{align}\label{eqn:fde-0}
\begin{aligned}
\Dal u -\Delta u & = f,\ &&\text{in}\ \Omega\times(0,T],\\
u &= 0,\ &&\text{on}\ \partial\Omega, \\
u(0) = a,~\partial_t u(0) &= b,\ &&\text{in}\ \Omega, \\
\end{aligned}
\end{align}
where $T>0$ is a fixed final time, $f \in L^\infty(0,T;L^2(\Omega))$ and $a,b\in L^2(\Omega)$ are given
source term and initial data, respectively, and $\Delta$ denotes the Laplace operator in space. Here
$\Dal u(t)$ denotes the Caputo fractional derivative in time $t$ of order $\alpha\in(1,2)$ \cite[p. 70]{KilbasSrivastavaTrujillo:2006}
\begin{align*}
   \Dal u(t)= \frac{1}{\Gamma(2-\alpha)}\int_0^t(t-s)^{1-\alpha} \tfrac{\partial^2}{\partial s^2} u(s)\d s.
\end{align*}

In recent years, there has been a growing interest in fractional / nonlocal models
due to their diverse applications in physics, engineering, biology
and finance. Specifically, the time-fractional diffusion equations ($\alpha\in(0,1)$) are often used to model subdiffusion phenomena
in media with highly
heterogeneous aquifers \cite{Adams:1992, Hatano:1998} and fractal geometry \cite{Nigmatullin:1986},
while the time-fractional diffusion-wave equations \eqref{eqn:fde-0}  ($\alpha\in(1,2)$)  are frequently used to describe
the propagation of mechanical waves in viscoelastic media \cite{Mainardi:1996, Mainardi:book}.
We refer interested readers to \cite{Metzler:2014,MetzlerKlafter:2000}
for a long list of applications of fractional models arising from biology and physics.

Inverse problems for fractional evolution models have attracted much interest, and there has already been a vast literature; see e.g.,
review papers \cite{JinRundell:2015, LiLiuYamamoto:2019b, LiYamamoto:2019a, LiuLiYamamoto:2019c} and references therein.
The aim of this paper is to investigate the backward problem for the diffusion-wave model \eqref{eqn:fde-0}: \textbf{(IP)}
we simultaneously determine the initial data $u(x,0)$ and $u_t(x,0)$ with $x\in \Omega$ (and hence the function
$u(x,t)$ for all $(x,t)\in \Omega\times(0,T)$)
from two terminal
observations
\begin{equation*}
u(x,T_1)=g_1(x), \quad u(x,T_2) = g_2(x)\quad \text{for all}  ~~x\in\Omega,
\end{equation*}
where $T_1, T_2\in(0,T]$ and $T_1<T_2$.

The study on the backward problem
for the diffusion-wave model remains fairly scarce.
In \cite{WeiZhang:2018} Wei and Zhang studied the backward problem to recover a single initial condition $u(0)$
or $u_t(0)$ (with the other one known) from the single terminal data $u(T)$.
Floridia and Yamamoto analyzed the simutaneous recovery of two initial data
from two terminal observations $u(T)$ and $u_t(T)$, and established a Lipschitz stability
in \cite{FloridiaYamamoto:2020}. In the setting of current paper,
we consider two observations $u(T_1)$ and $u(T_2)$,
which are practical in many empirical experiments.
As far as we know, there is no rigorous analysis of the discretized (numerical) scheme for
solving the backward problem (IP)
where some regularization error and discretization error(s) will be introduced into the system.
Then there arises a natural question: is it possible to derive an a priori error estimate, showing the way to
to balance discretization error, regularization parameter and the noise?
However, such an analysis remains unavailable, and it is precisely this gap that the project aims to fill in.

The backward subdiffusion problem ($\alpha\in(0,1)$) has been intensively studied in recent years, where
the single initial condition $u(0)$ is determined from the single observation $u(T)$.
See e.g. \cite{SakamotoYamamoto:2011,Tuan:2020} for the uniqueness and some stability estimate,
\cite{LiuYamamoto:2010,WangLiu:TV,YangLiu:2013,WeiWang:2014}
for some regularization methods, and \cite{ZhangZhou:2020} for error analysis of fully discrete schemes.
Compared with the subdiffusion problem, the initial layer near $t=0$ is more singular for the diffusion wave model, in sense that
$$  \|  u(t)  \|_{\dH 2} + \|  \partial_t^\alpha u(t)  \|_{L^2\II} \le c t^{-\alpha} \Big(\|  u(0) \|_{L^2\II} + t \|  u_t(0) \|_{L^2\II}\Big), $$
that clearly indicates a stronger singularity for $\alpha\rightarrow2$.
This brings more challenges in both numerical approximation and analysis.
Besides, the mathematical and numerical analysis of the backward subdiffusion problem
heavily rely on the completely monotonicity of the Mittag--Leffler function $E_{\alpha,1}(-z)$ or its discrete analogue,
which is not valid for $\alpha\in (1,2)$. Therefore, the well-posedness of the backward problem requires additional conditions
on the terminal time levels $T_1$ and $T_2$. This contrasts sharply with the subdiffusin counterpart  ($\alpha\in(0,1)$),
where the existence, uniqueness and two-side Lipschitz stability hold valid for any observation $u(T)$ with $T>0$.

In the first part of this paper, we show the well-posedness of the backward diffusion-wave problem.
In particular, using the asymptotic behavior of Mittag-Leffler functions, we show that,
under some conditions on $T_1$ and $T_2$ (depending on the spectrum of $-\Delta$),
 for any
$g_1, g_2\in \dH2$, there exists $a,b\in \L2Om$ such that the solution $u$ to \eqref{eqn:fde} satisfies
$u(T_1) = g_1$ and $u(T_2) = g_2$, and there holds two-sided Lipschitz stability (Theorem \ref{thm:stab})
\begin{equation*}
c_1\Big(\|g_1\|_{\dH 2}+\|g_2\|_{\dH 2}\Big) \le \|a\|_\L2Om+\|b\|_\L2Om\le c_2\Big(\|g_1\|_{\dH 2}+\|g_2\|_{\dH 2}\Big).
\end{equation*}
where the constants $c_1$ and $c_2$ only depend on $T_1$, $T_2$ and the fractional order $\alpha$.
In practice, we assume that the observation data $g_1^\delta$ and $g_2^\delta$ are noisy in sense that
\begin{align}\label{eqn:noise}
\|g_1-g_1^\delta\|_{L^2\II} = \|g_2-g_2^\delta\|_{L^2\II} = \delta.
\end{align}
Note that the empirical observations $g_1^\delta$ and $g_2^\delta$ only belong to $L^2\II$.
In order to regularize the mildly ill-posed problem, we apply the quasi-boundary value method \cite{Liu:2019, YangLiu:2013}: find $\tud(t)$ satisfies
\begin{align}\label{eqn:back-2}
\begin{aligned}
\Dal \tud -\Delta \tud & = 0,&& \text{in}\ \Omega\times(0,T],\\
\tud &= 0,&& \text{on}\ \partial\Omega,\\
-\gamma \tud(0)+\tud(T_1) &= g_1^\delta,&& \text{in}\ \Omega, \\
\gamma \partial_t \tud(0)+\tud(T_2) &= g_2^\delta, &&\text{in}\ \Omega,
\end{aligned}
\end{align}
where the constant $\gamma>0$ denotes the regularization parameter.
In Theorem \ref{thm:est-tud-u},
we show that if $T_1$ and $T_2$ sufficiently large and $a,b\in L^2\II$ then there holds
\begin{equation*}
 \|(\tud-u)(0)\|_{\L2Om}  + \|\partial_t(\tud-u)(0)\|_{\dH {-s}}\rightarrow
 0\quad \text{for}~~ \delta,\gamma\rightarrow0, \frac{\delta}{\gamma}\rightarrow0.
 \end{equation*}
 with any $s\in(0,1]$. Moreover, if $a,b\in \dH q$ with $q\in[0,2]$, we have the following \textsl{a priori} estimate
\begin{equation*}\begin{aligned}
\|(\tud-u)(0)\|_{L^2(\Omega)} +
\|\partial_t (\tud-u)(0)\|_{L^2(\Omega)}
&\le c\Big(\gamma^\frac q2+\delta\gamma^{-1}\Big)  .
\end{aligned}\end{equation*}
and for all $t\in(0,T]$
\begin{equation*}\begin{aligned}
\|(\tud-u)(t)\|_{L^2(\Omega)}  &\le c\big(\gamma  \min(\gamma^{-(1-\frac{q}{2})},t^{-\alpha(1-\frac{q}{2})}) + \delta
\min(\gamma^{-1},t^{-\alpha}) \big).
\end{aligned}\end{equation*}
To approximate $u(t)$ with $t>0$, the above estimates indicate the optimal choice of regularized parameter $\gamma\sim \delta$,
then the corresponding error is of order $O(\delta)$,
which is independent of the smoothness of initial data. Meanwhile, for $t=0$, the choice $\gamma\sim \delta^{\frac{2}{q+2}}$ leads to the optimal approximation $O(\delta^\frac{q}{q+2})$
if $a,b\in \dH q$ with $q\in(0,2]$.
These results will be intensively used in the  error estimation of (discretized) numerical schemes.
The proof mainly relies on the asymptotic behaviors of Mittag--Leffler functions. 

The second contribution of this paper is to develop a discrete numerical schemes for
solving the backward diffusion-wave problem with provable error bound.
The literature on the numerical approximation for the direct problems of time-fractional
models is vast. The most popular methods include
convolution quadrature \cite{CuestaLubichPalencia:2006,JinLiZhou:2017sisc,BanjaiLopez:2019,Fischer:2019},
collocation-type method \cite{ZhuXu:2019, Stynes:2017, Liao:2018, Kopteva:2019, Kopteva:2020, Kopteva:2021},
discontinuous Galerkin method
\cite{MustaphaAbdallahFurati:2014, McLeanMustapha:2015},
and spectral method \cite{Chen:2020, HouXu:2017, ZayernouriKarniadakis:2013}.
See also \cite{Baffet:2017,
JiangZhang:2017, Lubich:2008, XuHesthavenChen:2015}
for some fast algorithms.
Specifically, in this work,
we discretize the regularized
problem \eqref{eqn:back-2} by applying piecewise linear finite element method (FEM) in space with mesh size $h$,
and convolution quadrature generated by backward Euler scheme
(CQ-BE) in time with time step size $\tau$. Then some discretization error will be introduced into the system.
We carefully establish some
error bounds for the proposed scheme and
specify the way to balance the discrization error, regularization parameter and noise level.
For example, we show the following error estimates.
Suppose that $u(t)$ is the exact solution of the backward diffusion-wave problem
and $\t U_n^\delta$ is the fully discrete solution (approximating $u(t_n)$ at time level $t_n = n\tau$),
$\t a_{h,\tau}^\delta$ and $\t b_{h,\tau}^\delta$ are the approximations to exact initial data $a$ and $b$, respectively. Then,
provided that $a,b\in L^2\II$ and $T_1$ and $T_2$ are sufficiently large (depending on the smallest eigenvalue of $-\Delta$),  for arbitrarily small $s\in(0,1]$, we have (Theorem \ref{thm:err-fully})
\begin{equation*}
 \| \tilde a_{h,\tau} ^\delta- a \|_{L^2\II}   +  \| \tilde b_{h,\tau}^\delta - b \|_{H^{-s}\II}   \rightarrow0,\qquad \text{as}~~
\gamma, \tau \rightarrow0,~ \frac{\delta}{\gamma} \rightarrow0, ~ \frac{h}{\gamma} \rightarrow0,
\end{equation*}
and for $n\ge 1$
\begin{equation*}
\|  \tilde U_n^\delta-u(t_n)  \|_{L^2\II}
\le c \Big[\gamma t_n^{-\alpha}+ (\tau t_n^{\alpha-1}+ h^2 +\delta \big) \min(\gamma^{-1}, t_n^{-\alpha}) \Big].
\end{equation*}
Here the constant $c$ may depend on $T_1$, $T_2$, $T$, $a$ and $b$, but is always independent of $h$, $\gamma$, $\delta$ and $t_n$.
This estimate is useful since it indicates the way to balance parameters $\gamma$, $h$, $\tau$ according to $\delta$.
The estimates could be further improved if the initial data $a$ and $b$ are more regular and compatible with the boundary condition.
In particular, if $a,b\in \dH q$ with $q\in(0,2]$, there holds
\begin{equation*}
 \| \tilde a_{h,\tau}^\delta- a \|_{L^2\II}   +  \| \tilde b_{h,\tau}^\delta - b \|_{L^2\II}
\le c \big(\gamma^{\frac{q}{2}} + \tau + (h^2+\delta)\gamma^{-1}\big)   .
\end{equation*}
and for $n\ge 1$
\begin{equation*}
\|  \tilde U_n^\delta-u(t_n)  \|_{L^2\II}
\le c \Big[\gamma \min(\gamma^{-(1-\frac q2)}, t_n^{-(1-\frac q2)\alpha})+ (\tau t_n^{\alpha-1}+ h^2 +\delta \big) \min(\gamma^{-1}, t_n^{-\alpha}) \Big].
\end{equation*}
The proof relies heavily on refined properties of (discrete) solution operators
and some non-standard error estimates for the direct problem
in terms of problem data regularity \cite{JinLazarovZhou:SISC2016}.
As far as we know, this is the first work providing rigorous error analysis of numerical methods
for solving the  backward diffusion-wave problem.
Note that the above error estimates are much sharper than the
ones stated in the early work, see e.g., \cite[Theorem 4.1]{ZhangZhou:2020}, for backward subdiffusion problem.
Moreover, the analysis in current work only requires the domain to be convex polygonal while in \cite[Section 4]{ZhangZhou:2020} we assume the
the boundary of $\Omega$ is sufficiently smooth (since we required $H^4\II$-regularity for smooth data).
The sharpness of the error estimates are fully examined by the numerical experiments.

The rest of the paper is organized as follows. In section \ref{sec:prelim}, we provide some preliminary
results about solution representation and asymptotic behaviors of Mittag--Leffler functions.
Stability and regularization for the inverse problem are introduced in Section \ref{sec:stab}.
Then in sections \ref{sec:semi} and \ref{sec:fully}, we propose and analyze
spatially semi-discrete scheme and space-time fully discrete scheme, respectively.
Finally, in section \ref{sec:numerics}, we
present some numerical examples to illustrate and complete the theoretical analysis.
The notation $c$ denotes a generic constant, which may change at each occurrence, but it is always independent
of the noise level $\delta$, the regularization parameter $\gamma$, the mesh size $h$ and time step $\tau$ etc.

\section{Preliminaries}\label{sec:prelim}
In this section, we shall present some preliminary results about the diffusion-wave equation \eqref{eqn:fde-0},
including Mittag-Leffler functions, solution representation, and solution regularity.

In our analysis, a class of special functions, called Mittag--Leffler functions, play an 
important role. The Mittag--Leffler functions are defined by the following power series
\begin{equation*}
E_{\alpha,\beta}(z) = \sum_{k=0}^\infty  \frac{z^k}{\Gamma(k\alpha+\beta)} \qquad \text{for all}~~z\in \mathbb{C}.
\end{equation*}

Then the next lemma provides some useful bounds and asymptotic behaviors for Mittag-Leffler functions.
See detailed proof in \cite[p. 35]{Podlubny1999} and \cite[Theorem 3.2]{Jin:book}.
\begin{lemma}
\label{lem:ml}
Assume that  $0<\alpha<2$ and $\beta>0$. The Mittag-Leffler function $E_{\alpha,\beta}(z)$ is an entire function.
Meanwhile, there exists a positive constant $c$ (depending on $\alpha$ and $\beta$) such that
\begin{equation}\label{bound-ml}
|E_{\alpha,\beta}(-z)|\le \frac{c}{1+z},\ \text{for all}\ z\ge 0.
\end{equation}
Moreover, for large $z$, there holds the following asymptotic behaviours
\begin{equation}\label{asymp-ml}
\begin{aligned}
E_{\alpha,1}(-z) &= \frac{1}{\Gamma(1-\alpha)}\frac{1}{z}+O(\frac{1}{z^2})\quad \text{and}
\quad E_{\alpha,2}(-z) = \frac{1}{\Gamma(2-\alpha)}\frac{1}{z}+O(\frac{1}{z^2})\quad \forall z\to\infty.
\end{aligned}
\end{equation}
\end{lemma}

The solution of the diffusion-wave problem \eqref{eqn:fde-0} could be written as
\begin{equation}\label{eqn:sol-rep}
 u(t) = \FF(t)\begin{bmatrix} a\\b \end{bmatrix} + \int_0^t E(t-s)f(s)\,\d s
 = F(t) a + \bar F(t) b + \int_0^t E(t-s)f(s)\,\d s.
 \end{equation}
where the solution operators $ F(t)$, $\bar F(t)$ and $E(t)$ are respectively defined by
\begin{equation*}
\begin{aligned}
F(t) v        &=  \sum_{j=1}^\infty E_{\alpha,1}(-\lambda_jt^\alpha)(v,\fy_j)\fy_j, \quad \bar F(t) v =  \sum_{j=1}^\infty tE_{\alpha,2}(-\lambda_jt^\alpha)(v,\fy_j)\fy_j,\\
E(t)v         &=  \sum_{j=1}^\infty t^{\alpha-1}E_{\alpha,\alpha}(-\lambda_j t^\alpha)(v,\fy_j)\fy_j
\end{aligned}
\end{equation*}
for any $v\in L^2\II$.
By Laplace Transform, we have the following integral representations of the solution operators:
\begin{equation}\label{eqn:FE-LAP}
\begin{aligned}
F(t) &= \frac{1}{2\pi i }\int_\contour e^{zt} z^{\alpha-1}(z^\alpha-\Delta)^{-1}dz,\quad \bar F(t) &= \frac{1}{2\pi i }\int_\contour e^{zt} z^{\alpha-2}(z^\alpha-\Delta)^{-1}dz,\\
E(t) &= \frac{1}{2\pi i}\int_\contour  e^{zt} (z^\alpha-\Delta)^{-1} dz.
 \end{aligned}
\end{equation}
Here  $\contour$ denotes the integral contour in the complex plane $\mathbb{C}$, defined by
\begin{equation*}
\contour =\{z\in \mathbb{C}: |z| = \delta , |\arg z|\le \theta\} \cup \{ z\in \mathbb{C}: z =\rho e^{\pm i\theta},\rho\ge \sigma\},
\end{equation*}
with $\sigma\ge 0$ and $\frac{\pi}{2}<\theta< \frac{\pi}{\alpha}$,
oriented counterclockwise. 

To discuss the regularity of the solution, we shall need some notation.
Throughout, we denote by $\dH q$ the Hilbert space induced by the norm
\begin{equation*}
  \|v\|_{\dH q}^2:=\|(-\Delta)^\frac{q}{2}v\|_{L^2(\Omega)}^2=\sum_{j=1}^{\infty}\lambda_j^q ( v,\fy_j )^2, \qquad q\ge-1.
\end{equation*}
with $\{\lambda_j\}_{j=1}^\infty$ and $\{\fy_j\}_{j=1}^\infty$ being respectively the eigenvalues and
the $L^2(\Omega)$-orthonormal eigenfunctions of the negative Laplacian $-\Delta$ on the domain
$\Omega$ with a homogeneous Dirichlet boundary condition.
Then $\{\fy_j\}_{j=1}^\infty$ forms orthonormal basis in $L^2(\Omega)$ and hence
 $\|v\|_{\dH 0}=\|v\|_{L^2(\Omega)}$ is the norm in $L^2(\Omega)$.
Besides,  $\|v\|_{\dH 1}= \|\nabla v\|_{L^2(\Omega)}$ is a norm in $H_0^1(\Om)$,
$\|v\|_{\dH {-1}}= \| v\|_{ H^{-1}(\Omega)}$ is a norm in $H^{-1}(\Om) = (H_0^1\II)'$, 
$\|v\|_{\dH 2}=\|\Delta v\|_{L^2(\Omega)}$ is a norm in $H^2(\Om)\cap H^1_0(\Om)$
\cite[Section 3.1]{Thomee:2006}.

The important bounds in Lemma \ref{lem:ml} implies limited smoothing
properties in both space and time for the solution operators  $F(t)$, $\bar F(t)$ and $E(t)$.
Next, we state a few regularity results. The proof of these results can be found in, e.g.,
 \cite{Bajlekov:2001, JinLazarovZhou:SISC2016, Jin:book, SakamotoYamamoto:2011}.

\begin{lemma} \label{lem:reg-u}
Let $u(t)$ be defined in \eqref{eqn:sol-rep}. Then the following statements hold.
\begin{itemize}
  \item[$\rm(i)$] If $a,b \in \dH q$  with $q\in[0,2]$ and $f=0$, then $u(t)$ is the solution to problem \eqref{eqn:fde-0}, and
  $u(t)$ satisfies for any integer $m\ge 0$ and $ q \le p\le 2+q$
\begin{equation*}
\|  \partial_t^{(m)}  u(t) \|_{\dH p} \le c\left(t^{-m-\alpha(p-q)/2 }\|a\|_{\dH q}+t^{1- m-\alpha (p-q)/2}\|b\|_{\dH q}\right).
\end{equation*}
  \item[$\rm(ii)$]If $a=b=0$ and $f\in L^p(0,T;L^2(\Omega))$ with $1/\alpha<p<\infty$, then  $u(t)$ is the solution to problem \eqref{eqn:fde-0}
such that $u\in C([0,T];L^2\II)$ and
\begin{equation*}
\|  u\|_{L^p(0,T;\dot H^2(\Omega))}
+\|\Dal u\|_{L^p(0,T;L^2(\Omega))}
\le c\|f\|_{L^p(0,T;L^2(\Omega))}.
\end{equation*}
\end{itemize}
\end{lemma}

By using Lemma \ref{lem:reg-u}, if $u(t)$ is the solution to the diffusion-wave equation, the function
$w(t) = u(t) - \int_0^t E(t-s) f(s) \,\d s$
satisfies the diffusion-wave equation \eqref{eqn:fde-0} with a trivial source term.
Therefore, without loss of generality, throughout the paper we consider the homogeneous problem
\begin{align}\label{eqn:fde}
\begin{aligned}
\Dal u -\Delta u & = 0,\ &&\text{in}\ \Omega\times(0,T],\\
u &= 0,\ &&\text{on}\ \partial\Omega, \\
u(0) = a,~\partial_t u(0) &= b,\ &&\text{in}\ \Omega. \\
\end{aligned}
\end{align}

\section{Stability and regularization}\label{sec:stab}
The aim of this section is to show the Lipschitz stability of the inverse problem.
Moreover, we shall develop a regularization scheme to regularize the ``mildly'' ill-posed problem 
(with noisy observation data).
A complete analysis of the regularized problem will be provided.

\subsection{Stability of the backward diffusion-wave problems}
To begin with, we intend to examine the well-posedness of the 
backward problem diffusion-wave problem for $0<T_1<T_2\le T$
\begin{align}\label{eqn:back-1}
\begin{aligned}
\Dal u -\Delta u & = 0,\ &&\text{in}\ \Omega\times(0,T],\\
u &= 0,\ &&\text{on}\ \partial\Omega, \\
u(T_1) = g_1,~
 u(T_2) &= g_2,\ &&\text{in}\ \Omega.
\end{aligned}
\end{align}
Using the solution representation \eqref{eqn:sol-rep}, we have the following relation
\begin{equation}\label{eqn:sol-2}
\begin{aligned}
\begin{bmatrix}
g_1\\g_2
\end{bmatrix}
&=\GG(T_1,T_2) \begin{bmatrix} a\\b \end{bmatrix}
:= \begin{bmatrix}
F(T_1) & \bar F(T_1)\\
F(T_2) &\bar F(T_2)
\end{bmatrix}\begin{bmatrix} a\\b \end{bmatrix}=\sum_{j=1}^\infty \begin{bmatrix}
E_{\alpha,1}(-\lambda_j T_1^\alpha)& T_1E_{\alpha,2}(-\lambda_j T_1^\alpha)\\
E_{\alpha,1}(-\lambda_j T_2^\alpha)& T_2E_{\alpha,2}(-\lambda_j T_2^\alpha)
\end{bmatrix}
\begin{bmatrix}
(a,\fy_j) {\fy_j}\\(b,\fy_j) {\fy_j}
\end{bmatrix}.
\end{aligned}
\end{equation}
In order to represent the inverse of the operator $\GG(T_1,T_2)$, we define the function
\begin{equation}\label{eqn:psi}
\psi(T_1, T_2; \lambda_j) =  T_2E_{\alpha,1}(-\lambda_j T_1^\alpha)E_{\alpha,2}
(-\lambda_jT_2^\alpha)-T_1E_{\alpha,1}(-\lambda_j T_2^\alpha)E_{\alpha,2}(-\lambda_jT_1^\alpha).
\end{equation}
Then $\GG(T_1,T_2)^{-1}$ is well-defined, provided that $\psi(T_1, T_2; \lambda_j) \neq0$ for all $j=1,2,\ldots$, and
a direct computation leads to the relation
\begin{equation}\label{eqn:sol-3}
\begin{aligned}
\begin{bmatrix}
a\\b
\end{bmatrix}
&=\GG(T_1,T_2)^{-1} \begin{bmatrix} g_1\\g_2 \end{bmatrix}=\sum_{j=1}^\infty
\psi(T_1,T_2;\lambda_j)^{-1}
\begin{bmatrix}
T_2E_{\alpha,2}(-\lambda_j T_2^\alpha) &  {-T_1E_{\alpha,2}(-\lambda_jT_1^\alpha) }\\  { - E_{\alpha,1}(-\lambda_j T_2^\alpha)}
& E_{\alpha,1}(-\lambda_j T_1^\alpha)
\end{bmatrix}
\begin{bmatrix}
(g_1,\fy_j)\fy_j\\ (g_2,\fy_j)\fy_j
\end{bmatrix}.
\end{aligned}
\end{equation}

The next lemma
clarifies the conditions for $\psi(T_1,T_2;\lambda_j) \neq 0$ for all $j=1,2,\ldots$.
\begin{lemma}\label{lem:unique}
Let $\lambda>0$ and $\psi(T_1, T_2;\lambda)$ be the function defined in \eqref{eqn:psi}.
Then there exists a constant $M(\lambda)$ such that for all $T_2>T_1\ge M(\lambda)$, then
$$ {\psi(T_1,T_2;\lambda) \le \frac{c(T_2-T_1)}{ \Gamma(1-\alpha)
\Gamma(2-\alpha)}\frac{1}{\lambda^2T_1^\alpha T_2^\alpha} < 0  },$$
where the constant c is independent of $\lambda$,  $T_1$ and $T_2$.
\end{lemma}
\begin{proof}
By means of the asymptotic property of Mittag--Leffler functions in \eqref{asymp-ml},  we have
\begin{equation}\label{asymp:determinant}
\psi(T_1,T_2;\lambda) = (T_2-T_1)\left( \frac{1}{\Gamma(1-\alpha)\Gamma(2-\alpha)}\frac{1}{\lambda^2T_1^\alpha T_2^\alpha}
+O\Big(\frac{1}{\lambda^4T_1^{2\alpha}T_2^{2\alpha}}\Big) \right), \qquad \text{for}~~T_1,T_2\to\infty.
\end{equation}
For $\lambda>0$ and $T_2>T_1>0$, we know the leading term
$ {\frac{1}{\Gamma(1-\alpha)\Gamma(2-\alpha)}\frac{1}{\lambda^2T_1^\alpha T_2^\alpha}<0},$
and hence the asymptotic behavior \eqref{asymp:determinant} implies the existence of $M(\lambda)$ such that for all $T_2>T_1\ge M(\lambda)$
$$  {\psi(T_1,T_2;\lambda) \le \frac{T_2-T_1}{2\Gamma(1-\alpha)\Gamma(2-\alpha)}\frac{1}{\lambda^2T_1^\alpha T_2^\alpha} < 0  }.$$
This completes the proof of the lemma.
\end{proof}

Combining Lemmas \ref{lem:ml} and \ref{lem:unique}, we have the following stability estimate.

\begin{theorem}\label{thm:stab}
Let $\lambda_1$ be the smallest eigenvalue of $-\Delta$ with homogeneous Dirichlet boundary condition,
and $M(\lambda_1)$ be the constant defined in Lemma \eqref{lem:unique}.
Suppose that $T_2 > T_1 \ge  M(\lambda_1)$. Then for any
$g_1, g_2\in \dH2$, there exists $a,b\in \L2Om$ such that the solution $u$ to \eqref{eqn:fde} satisfies
\begin{equation*}
u(T_1) = g_1\quad \text{and}\quad u(T_2) = g_2.
\end{equation*}
Meanwhile, there holds the following two-sided Lipschitz stability
\begin{equation}\label{eqn:sol-4}
c_1\Big(\|g_1\|_{\dH 2}+\|g_2\|_{\dH 2}\Big) \le \|a\|_\L2Om+\|b\|_\L2Om\le c_2\Big(\|g_1\|_{\dH 2}+\|g_2\|_{\dH 2}\Big).
\end{equation}
\end{theorem}

\begin{proof}
By Lemma \ref{lem:unique} and the asymptotic estimate \eqref{asymp:determinant}, we have for all $T_2>T_1>M(\lambda_1)$
and $\lambda \ge \lambda_1$
\begin{equation}\label{eqn:est-psi}
  {|\psi(T_1,T_2;\lambda)| \ge \left|\frac{c(T_2-T_1)}{\Gamma(1-\alpha)\Gamma(2-\alpha)}\frac{1}{\lambda^2T_1^\alpha T_2^\alpha}\right| > 0, }
 \end{equation}
where the constant $c$ is independent of $\lambda_j$, $T_1$ and $T_2$. This together with \eqref{eqn:sol-3}
indicates the existence and uniqueness of initial data $a$ and $b$.

Next we turn to the stability estimate. Noting that the first inequality has been confirmed by Lemma \ref{lem:reg-u},
so it suffices to verify the second one. The estimate \eqref{eqn:est-psi} and the relation \eqref{eqn:sol-3} imply
\begin{equation*}
\begin{aligned}
\| a \|_{L^2\II}^2 + \| b \|_{L^2\II}^2 &\le \frac{c}{(T_2-T_1)^2} \sum_{j=1}^\infty \lambda_j^4
 \Big(\frac{(g_1, \fy_j)^2}{(1+\lambda_j T_2^\alpha)^2} + \frac{(g_2, \fy_j)^2}{(1+\lambda_j T_2^\alpha)^2} \Big)\\
 &\le  \frac{c}{(T_2-T_1)^2}  \Big(\|g_1\|_{\dH 2}^2+\|g_2\|_{\dH 2}^2\Big).
\end{aligned}
\end{equation*}

\end{proof}

\begin{remark}
Note that in the stability estimate \eqref{eqn:sol-4} the constant $c_2$ is proportional to $(T_2-T_1)^{-1}$.
This is reasonable since one cannot recover two initial data $u(0)$ and $\partial_t u(0)$ from a single observation $u(T)$.
Throughout our numerical analysis, we shall assume that $T_2>T_1\ge M(\lambda_1)$ and $T_2-T_1 \ge c_0 >0$.
\end{remark}

\subsection{Regularization and convergence analysis}
From now on, we shall assume that our observation is noisy with noise level $\delta$, i.e., \eqref{eqn:noise}.
Note that both $g_1^\delta$ and $g_2^\delta$ are nonsmooth.
Since the backward diffusion-wave problem \eqref{eqn:back-1}  is mildly ill-posed, we shall regularize the problem by using
the quasi boundary value scheme \eqref{eqn:back-2}.
Recalling the definition of the operator $\GG(T_1,T_2)$ in \eqref{eqn:sol-2}, the solution to the regularized problem \eqref{eqn:back-2}
could be written as
\begin{equation*}
\begin{aligned}
\begin{bmatrix}
g_1^\delta \\g_2^\delta
\end{bmatrix}
&= (\gamma \I+ \GG(T_1,T_2))
\begin{bmatrix}
\tu(0)\\ \partial_t \tu(0)
\end{bmatrix}:=\sum_{j=1}^\infty
\begin{bmatrix}
-\gamma +E_{\alpha,1}(-\lambda_jT_1^\alpha)&T_1E_{\alpha,2}(-\lambda_jT_1^\alpha)\\
E_{\alpha,1}(-\lambda_jT_2^\alpha)&\gamma +T_2E_{\alpha,2}(-\lambda_jT_2^\alpha)
\end{bmatrix}
\begin{bmatrix}
(\tud(0),\fy_j)\fy_j\\
(\partial_t \tud(0),\fy_j)\fy_j
\end{bmatrix}
\end{aligned}
\end{equation*}
where $\I$ denotes the matrix of operators
\begin{equation}\label{eqn:I}
\I  = \begin{bmatrix}
-I & 0 \\
0 & I
\end{bmatrix}
\end{equation}
where $I$ is the identity operator.

Now we define an auxiliary function
\begin{equation}\label{form:regular:determinant}
 {\tilde \psi(T_1,T_2;\lambda_j) :=\psi(T_1,T_2;\lambda_j)- \gamma^2+\gamma [E_{\alpha,1}(-\lambda_j T_1^\alpha)-T_2E_{\alpha,2}(-\lambda_jT_2^\alpha)].}
\end{equation}
Lemma \ref{lem:ml} implies that there exists a constant $z_0>0$ such that for $z\ge z_0$,
\begin{equation*}
\begin{aligned}
 {  E_{\alpha,1}(-z) \le\frac{1}{2\Gamma(1-\alpha)}\frac{1}{z}  < 0 } \quad \text{and}\quad
 { E_{\alpha,2}(-z) \ge \frac{1}{2\Gamma(2-\alpha)}\frac{1}{z} > 0}.
\end{aligned}
\end{equation*}
Without loss of generality, we assume that
\begin{equation}\label{ass:z0}
M(\lambda_1)^\alpha > z_0/\lambda_1.
\end{equation}
Then with $T_2>T_1\ge M(\lambda_1)$, 
\begin{equation}\label{eqn:determinant-2-b}
 { \tilde \psi(T_1,T_2;\lambda_j) \le -c\Big( \lambda_j^{-2} +  \gamma \lambda_j^{-1} +\gamma^2 \Big)<0 },
\end{equation}
where $c$ is only dependent on $T_1$, $T_2$ and $\alpha$.
Therefore the operator $\gamma \I+ \GG(T_1,T_2)$ is also invertible and there holds the relation
 \begin{align}\label{eqn:sol-back-2}
\begin{bmatrix}
\tud(0)\\\partial_t \tud(0)
\end{bmatrix}
&=(\gamma \I+\GG(T_1,T_2))^{-1} \begin{bmatrix} g_1^\delta\\g_2^\delta \end{bmatrix}\\
&=\sum_{j=1}^\infty
\widetilde \psi(T_1,T_2;\lambda_j)^{-1}
\begin{bmatrix}
\gamma+T_2E_{\alpha,}(-\lambda_jT_2^\alpha)& -T_1E_{\alpha,2}(-\lambda_jT_1^\alpha)  \\
 -E_{\alpha,1}(-\lambda_jT_2^\alpha) &  -\gamma +E_{\alpha,1}(-\lambda_jT_1^\alpha)
\end{bmatrix}
\begin{bmatrix}
(g_1^\delta,\fy_j)\fy_j\\ (g_2^\delta,\fy_j)\fy_j
\end{bmatrix}.\notag
\end{align}
 Meanwhile, with $\FF(t) = [F(t)~\bar F(t)]$, we know
 \begin{equation}\label{eqn:sol-back-3}
\tu^\delta (t)
=\FF(t)(\gamma \I +\GG(T_1,T_2))^{-1}
\begin{bmatrix}
g_1^\delta \\g_2^\delta
\end{bmatrix}.
\end{equation}

Now we intend to establish estimates for $u(0) - \tu^\delta(0)$, $\partial_tu(0) - \partial_t\tu^\delta(0)$ and $u(t) - \tu^\delta (t)$.
To this end, we need the following auxiliary function
 \begin{equation}\label{eqn:sol-back-4}
\tu (t)
=\FF(t)(\gamma I +\GG(T_1,T_2))^{-1}
\begin{bmatrix}
g_1 \\g_2
\end{bmatrix}=\FF(t)(\gamma I +\GG(T_1,T_2))^{-1} \GG(T_1,T_2)
\begin{bmatrix}
a \\ b
\end{bmatrix},
\end{equation}
which is the solution to the following quasi boundary value problem:
\begin{align}\label{eqn:back-3}
\begin{aligned}
\Dal \tu -\Delta \tu & = 0,&& \text{in}\ \Omega\times(0,T],\\
\tu &= 0,&& \text{on}\ \partial\Omega,\\
-\gamma \tu(0)+\tu(T_1) &= g_1,&& \text{in}\ \Omega, \\
\gamma \partial_t \tu(0)+\tu(T_2) &= g_2, &&\text{in}\ \Omega.
\end{aligned}
\end{align}

The next lemma provides an estimate for the operator $\FF(t)(\gamma \I +\GG(T_1,T_2))^{-1}$.
\begin{lemma} \label{lem:regular:estimate:ope}
Let $M(\lambda_1)$ be the constant defined in Lemma \ref{lem:unique},
and suppose that $T_2 > T_1 \ge  M(\lambda_1)$.
Let $\FF(t)$ and $(\gamma \I +\GG(T_1,T_2))^{-1}$ be defined in
\eqref{eqn:sol-rep} and \eqref{eqn:sol-back-2}, then for all $0<t \le T$, $v,w\in \dH q$, for any $0\le p\le q\le 2+p$, we have
\begin{equation*}
\left\|\Big(\frac{d}{dt}\Big)^\ell\FF(t)(\gamma \I +\GG(T_1,T_2))^{-1}\begin{bmatrix}
v\\w
\end{bmatrix}
\right\|_{\dH p} \le  c t^{-\ell}
\min(\gamma^{-(1+\frac{p-q}{2})},t^{-\alpha(1+\frac{p-q}{2})}) (\|v\|_{\dH q} +\|w\|_{\dH q}).
\end{equation*}
Meanwhile, we have
\begin{equation*}
\left\|(\gamma \I +\GG(T_1,T_2))^{-1}\begin{bmatrix}
v\\w
\end{bmatrix}
\right\|_{\dH p} \le  c\gamma^{-(1+\frac{p-q}{2})} (\|v\|_{\dH q} +\|w\|_{\dH q}).
\end{equation*}
\end{lemma}
\begin{proof}
First of all, for $0<t \le T$, we let
\begin{equation*}
\begin{aligned}
\zeta(t) &=\FF(t)(\gamma I +\GG(T_1,T_2))^{-1}\begin{bmatrix}
v\\w
\end{bmatrix} \\
&= \sum_{j=1}^\infty \tilde\psi(T_1,T_2;\lambda_j)^{-1}
\begin{bmatrix}
E_{\alpha,1}(-\lambda_jt^\alpha)&tE_{\alpha,2}(-\lambda_jt^\alpha)
\end{bmatrix}
\begin{bmatrix}
\gamma+T_2E_{\alpha,2}(-\lambda_jT_2^\alpha)& {-T_1E_{\alpha,2}(-\lambda_jT_1^\alpha)}\\  {-E_{\alpha,1}(-\lambda_jT_2^\alpha)}
 &  {-\gamma}+E_{\alpha,1}(-\lambda_jT_1^\alpha)
\end{bmatrix}
\begin{bmatrix}
(v,\fy_j)\fy_j\\(w,\fy_j)\fy_j
\end{bmatrix}.
\end{aligned}
\end{equation*}
By means of Lemmas \ref{lem:ml}, we arrive at
\begin{equation}\label{bound:regular:mat:F:t}
\begin{bmatrix}
|E_{\alpha,1}(-\lambda_jt^\alpha)|&|tE_{\alpha,2}(-\lambda_jt^\alpha)|
\end{bmatrix}
\le  \frac{c}{1+\lambda_j t^\alpha}\begin{bmatrix}
1 & t
\end{bmatrix}.
\end{equation}
Similarly by Lemma \ref{lem:ml} and the estimate \eqref{eqn:determinant-2-b}
\begin{equation}\label{bound:regular:mat:G}
\begin{aligned}
|\t \psi(T_1,T_2;\lambda_j)|^{-1}
\begin{bmatrix}
|\gamma+T_2E_{\alpha,2}(-\lambda_jT_2^\alpha)| & {|-T_1E_{\alpha,2}(-\lambda_jT_1^\alpha)|} \\
  {|-E_{\alpha,1}(-\lambda_jT_2^\alpha)|} & | {-\gamma}+E_{\alpha,1}(-\lambda_jT_1^\alpha)|
\end{bmatrix} 
&\le
\frac{c \lambda_j}{1+\gamma\lambda_j}\begin{bmatrix}
1 & 1\\
1& 1
\end{bmatrix}
\end{aligned}
\end{equation}
Combining \eqref{bound:regular:mat:F:t} and  \eqref{bound:regular:mat:G} we obtain
\begin{equation*}
\begin{aligned}
\lambda_j^p
(\zeta(t),\fy_j)^2
&\le  c
 \left(\frac{\lambda_j^{1+\frac{p-q}{2}}}{(1+\gamma\lambda_j)(1+\lambda_jt^\alpha)}\right)^2 \lambda_j^q ((v,\fy_j)^2+(w,\fy_j)^2) \\
&\le c
\Big(\min(\gamma^{-(1+\frac{p-q}{2})},t^{-\alpha(1+\frac{p-q}{2})})\Big)^2 \lambda_j^q((v,\fy_j)^2+(w,\fy_j)^2).
\end{aligned}
\end{equation*}
As a result, we conclude that
\begin{equation*}
\begin{aligned}
 \|  \zeta(t) \|_{\dH p}^2 &\le c \Big(\min(\gamma^{-(1+\frac{p-q}{2})},t^{-\alpha(1+\frac{p-q}{2})})\Big)^2 \sum_{j=1}^\infty \lambda_j^q((v,\fy_j)^2+(w,\fy_j)^2)\\
 & = c \Big(\min(\gamma^{-(1+\frac{p-q}{2})},t^{-\alpha(1+\frac{p-q}{2})})\Big)^2 \Big(\|  v \|_{\dH q}^2 + \|  w \|_{\dH q}^2\Big).
\end{aligned}
\end{equation*}

Now we turn to the second estimate. Noting that
\begin{equation*}
\begin{aligned}
\begin{bmatrix}
\zeta\\ \xi
\end{bmatrix}
&=(\gamma I +\GG(T_1,T_2))^{-1}\begin{bmatrix}
v\\w
\end{bmatrix} = \sum_{j=1}^\infty \tilde\psi(T_1,T_2;\lambda_j)^{-1}
\begin{bmatrix}
\gamma+T_2E_{\alpha,}(-\lambda_jT_2^\alpha)&  {-T_1E_{\alpha,2}(-\lambda_jT_1^\alpha)} \\
 {-E_{\alpha,1}(-\lambda_jT_2^\alpha)}&  {-\gamma}+E_{\alpha,1}(-\lambda_jT_1^\alpha)
\end{bmatrix}
\begin{bmatrix}
(v,\fy_j)\fy_j\\(w,\fy_j)\fy_j
\end{bmatrix},
\end{aligned}
\end{equation*}
the estimate \eqref{bound:regular:mat:G} leads to
\begin{equation*}
\begin{aligned}
\|\zeta\|_{\dH p}^2 + \|\xi\|_{\dH p}^2
&\le c \sum_{j=1}^\infty
\Big( \frac{ \lambda_j^{1+\frac{p-q}{2}}}{1+\gamma\lambda_j}\Big)^2
\lambda_j^q \Big(  (v,\fy_j)^2+ (w,\fy_j)^2 \Big)\\
&\le c\gamma^{-(2+(p-q))}  \Big(\|v\|_{\dH q}^2+\|w\|_{\dH q}^2\Big) .
\end{aligned}
\end{equation*}
This completes the proof of the lemma.
\end{proof}

Using the similar argument, we have the following estimate for higher regularity estimate for $\tu(0)$ and $\partial_t \tu(0)$,
which will be intensively used in the the next section.
\begin{corollary}\label{cor:tu-reg}
Let $M(\lambda_1)$ be the constant defined in Lemma \ref{lem:unique},
and suppose that $T_2 > T_1 \ge  M(\lambda_1)$.
Let $\tu$ be the solution to\eqref{eqn:back-3}. Then there holds
\begin{equation*}
\|  \tu(0) \|_{\dH q} + \| \partial_t \tu(0)  \|_{\dH q}  \le c \gamma^{-q/2} \Big( \|  a \|_{L^2\II} + \| b  \|_{L^2\II}\Big).
\end{equation*}
\end{corollary}


Lemma \ref{lem:regular:estimate:ope} with $p=q=0$ immediately leads to the estimate for $\tu^\delta - \tu$.

\begin{corollary} \label{cor:est-tu-tud}
Let $M(\lambda_1)$ be the constant defined in Lemma \ref{lem:unique},
and suppose that $T_2 > T_1 \ge  M(\lambda_1)$.
Let $\tud$ and $\tu$ be solutions to \eqref{eqn:back-2} and \eqref{eqn:back-3}, respectively.
Then for any $a, b\in L^2\II$ we have
\begin{equation*}
\| ( \tu^\delta - \tu)(t) \|_{L^2\II} \le  c\, \delta
\min(\gamma^{-1},t^{-\alpha})  \qquad \text{for all} ~~ t\in(0,T].
\end{equation*}
and
\begin{equation*}
\| ( \tu^\delta - \tu)(0) \|_{L^2\II} + \| \partial_t( \tu^\delta - \tu)(0) \|_{L^2\II}\le  c\,\delta\gamma^{-1} .
\end{equation*}
\end{corollary}

According to Lemma \ref{lem:regular:estimate:ope} we can derive the following estimate of $\tilde u(t) - u(t)$ with $t\in [0,T]$.
\begin{lemma}\label{lem:est-tu-u}
Let $M(\lambda_1)$ be the constant defined in Lemma \ref{lem:unique},
and suppose that $T_2 > T_1 \ge  M(\lambda_1)$.
Let $u(t)$ and $\tu(t)$ be the  solutions of problems \eqref{eqn:fde} and \eqref{eqn:back-3}, respectively.
\begin{itemize}
\item[(i)]
For $a,b\in \dH q$ with $q\in[0,2]$, we have
\begin{equation*}\begin{aligned}
\|(\tu-u)(0)\|_{L^2(\Omega)} +
\|\partial_t (\tu-u)(0)\|_{L^2(\Omega)}
&\le c\gamma^\frac q2
\end{aligned}\end{equation*}
and for all $t\in(0,T]$
\begin{equation*}\begin{aligned}
\|(\tu-u)(t)\|_{L^2(\Omega)}  &\le c\gamma \min(\gamma^{-(1-\frac q2)}, t^{-(1-\frac q2)\alpha}).
\end{aligned}\end{equation*}
\item[(ii)] In case that $a,b\in L^2\II$, we have for any small $s\in(0,1]$
\begin{equation*}
\lim_{\gamma\to 0}
\Big(\|(\tu-u)(0)\|_{\L2Om}  + \|\partial_t(\tu-u)(0)\|_{\dH {-s}} \Big)
= 0.
\end{equation*}
\end{itemize}
\end{lemma}
\begin{proof}
Recalling the definition of the operator $\GG(T_1, T_2)$ in \eqref{eqn:sol-2}, we have the representation
\begin{align*}
\begin{bmatrix}
(\tu-u)(0)\\ \partial_t (\tu-u)(0)
\end{bmatrix}
&=
(\gamma \I+\GG(T_1,T_2))^{-1} \GG(T_1,T_2)
\begin{bmatrix}
a\\b
\end{bmatrix}-
\begin{bmatrix}
a\\b
\end{bmatrix}
= -{\gamma (\gamma \I+\GG(T_1,T_2))^{-1}}\I
\begin{bmatrix} a\\b \end{bmatrix}.
\end{align*}
From lemma \ref{lem:regular:estimate:ope} for $p=0$, we have
\begin{equation*}
\begin{aligned}
\|(\tu-u)(0)\|_{\L2Om} + \|\partial_t(\tu-u)(0)\|_{\L2Om}
&\le  c\gamma^\frac q2(\|a\|_{\dH q}+\|b\|_{\dH q}).
\end{aligned}
\end{equation*}
Similarly, we have the following representation to $(\tu-u)(t)$:
\begin{align*}
(\tu-u)(t)
&=
\FF(t) (\gamma \I +\GG(T_1,T_2))^{-1} \GG(T_1,T_2)
\begin{bmatrix} a\\b\end{bmatrix}
-\FF(t)\begin{bmatrix} a\\b \end{bmatrix}
\\
&= -\gamma\FF(t)(\gamma \I +\GG(T_1,T_2))^{-1}  \I \begin{bmatrix} a\\b\end{bmatrix}.
\end{align*}
We apply Lemma \ref{lem:regular:estimate:ope} with $p=0$ again to obtain
\begin{equation*}
\begin{aligned}
\|(\tu-u)(t)\|_{\L2Om} \le c
\gamma \min(\gamma^{-(1-\frac q2)}, t^{-(1-\frac q2)\alpha}).
\end{aligned}
\end{equation*}
Now we show the estimate (ii) for $q=0$.
In case that $a,b\in L^2\II$, we know that $\tilde u, u\in C([0,T];L^2\II)$. Then for any small $\epsilon$, we choose $t_0$ small enough such that
$$  \|  \tilde u(t_0) - \tilde u(0) \|_{L^2\II}  + \|  u(t_0) -  u(0) \|_{L^2\II} < \epsilon/2. $$
Then by the estimate in (i), we may find $\gamma_0$ small enough such that
$$  \|  \tilde u(t_0)  -  u(t_0) \|_{L^2\II} < \epsilon/2\quad \mbox{for all} ~\gamma<\gamma_0. $$
By triangle inequality , we obtain that for any $\gamma<\gamma_0$
$$  \|  \tilde u(0)  -  u(0) \|_{L^2\II} < \epsilon. $$
Theqrefore, $\tilde u(0)$ converges to $u(0)$ in $L^2$-sense, as $\gamma \rightarrow 0$.
Finally, the convergence of $\partial_t \tilde u(0)$ in $H^{-s}$ follows from (i) and a  shift argument.
\end{proof}

Combining Corollary \ref{cor:est-tu-tud} and Lemma \ref{lem:est-tu-u}, we obtain the following convergence result.

\begin{theorem}\label{thm:est-tud-u}
Let $M(\lambda_1)$ be the constant defined in Lemma \ref{lem:unique},
and suppose that $T_2 > T_1 \ge  M(\lambda_1)$.
Let $u(t)$ and $\tud(t)$ be the  solutions of problems \eqref{eqn:fde} and \eqref{eqn:back-2}, respectively.
\begin{itemize}
\item[(i)]
For $a,b\in \dH q$ with $q\in[0,2]$, we have
\begin{equation*}\begin{aligned}
\|(\tud-u)(0)\|_{L^2(\Omega)} +
\|\partial_t (\tud-u)(0)\|_{L^2(\Omega)}
&\le c\Big(\gamma^\frac q2+\delta\gamma^{-1}\Big)  .
\end{aligned}\end{equation*}
and for all $t\in(0,T]$
\begin{equation*}\begin{aligned}
\|(\tud-u)(t)\|_{L^2(\Omega)}  &\le c\big(\gamma  \min(\gamma^{-(1-\frac{q}{2})},t^{-\alpha(1-\frac{q}{2})}) + \delta
\min(\gamma^{-1},t^{-\alpha}) \big).
\end{aligned}\end{equation*}
\item[(ii)] In case that $a,b\in L^2\II$, we have for any small $s\in(0,1]$
\begin{equation*}
 \|(\tud-u)(0)\|_{\L2Om}  + \|\partial_t(\tud-u)(0)\|_{\dH {-s}}\rightarrow 0\quad \text{for}~~ \delta,\gamma\rightarrow0, \frac{\delta}{\gamma}\rightarrow0.
 \end{equation*}
\end{itemize}
\end{theorem}

\begin{remark}\label{rem:est-tud-u}
To approximate $u(t)$ with $t>0$, Theorem \ref{thm:est-tud-u} indicates an optimal regularized parameter $\gamma\sim \delta$, and the error is of the order $O(\delta)$
which is independent of the smoothness of initial data. Meanwhile, for $t=0$, the choice $\gamma\sim \delta^{\frac{2}{q+2}}$ leads to the optimal approximation $O(\delta^\frac{q}{q+2})$
if $a,b\in \dH q$ with $q\in(0,2]$.
\end{remark}

\section{Spatially semidiscrete scheme and error analysis}\label{sec:semi}
In this section, we shall propose and analyze a spatially semidiscrete scheme for solving the backward diffusion wave problem.
The semidiscrete scheme would give an insite view
to understand the role of the regularity of problem data
and plays an important role in the analysis of fully discrete scheme.

\subsection{Semidiscrete scheme for solving direct problem}

Let ${\{\mathcal{T}_h\}}_{0<h<1}$ be a family
of shape regular and quasi-uniform partitions of the domain $\Omega$ into $d$-simplexes, called
finite elements, with $h$ denoting the maximum diameter of the elements.
We consider the finite element space $X_h$ defined by
\begin{equation}\label{eqn:fem-space}
  X_h =\left\{\chi\in C(\bar\Omega)\cap H_0^1: \ \chi|_{K}\in P_1(K),
 \,\,\,\,\forall K \in \mathcal{T}_h\right\}
\end{equation}
where $P_1(K)$ denotes the space of linear polynomials on $K$.
Then we define the $L^2(\Omega)$ projection $P_h:L^2(\Omega)\to X_h$ and
Ritz projection $R_h:H_0^1(\Omega)\to X_h$, respectively, by
\begin{equation*}
  \begin{aligned}
    (P_h \psi,\chi)   =(\psi,\chi) ~~\forall \chi\in X_h,\psi\in L^2(\Omega) ~~\text{and}~~
    (\nabla R_h \psi,\nabla\chi)   =(\nabla \psi,\nabla\chi) ~~ \forall \chi\in X_h, \psi\in  H_0^1(\Omega).
  \end{aligned}
\end{equation*}
Then $P_h$ and $R_h$ satisfies the following approximation properties \cite[Chapter 1]{Thomee:2006}
\begin{align}
\label{estimate:PR}
\|P_hv-v\|_{\L2Om}+\|R_hv-v\|_{\L2Om}\le ch^2\|v\|_{H^2(\Omega)},\ \forall v\in \dH 2.
\end{align}

Then the semidiscrete standard Galerkin FEM for problem \eqref{eqn:fde} reads: find $u_h(t)\in X_h$ such that
\begin{align}\label{eqn:semi-0}
 \begin{aligned}
(\Dal u_h,\chi) + (\nabla u_h,\nabla \chi) & = (f,\chi),\ \forall \chi \in X_h,\ T\ge t>0,\\
u_h(0) = P_ha ,~~
\partial_t u_h(0) &= P_hb.\\
\end{aligned}
\end{align}
By introducing the discrete Laplacian $-\Delta_h:\, X_h\to X_h$ such that
\begin{equation*}
(-\Delta_h\xi,\chi) = (\nabla\xi,\nabla\chi),\ \forall\xi,\chi\in X_h,
\end{equation*}
spatially semidiscrete problem \eqref{eqn:semi-0} could be written as
\begin{align}\label{eqn:semi-1}
 \begin{aligned}
\Dal u_h-\Delta_hu_h& = f_h,\ T\ge t>0,\\
u_h(0) = P_ha ,~~
\partial_t u_h(0) &= P_hb.\\
\end{aligned}
\end{align}
Let $\{\lambda_j^h,\fy_j^h\}_{j=1}^J$ be eigenpairs of $-\Delta_h$ with $\lambda_1^h\le \lambda_2^h\le \ldots\lambda_J^h$.
By the Courant minimax principle and the fact that $X_h\subset H_0^1\II$, we know
\begin{equation}\label{eqn:minmax}
\lambda_1^h = \min_{\phi_h\in X_h}\frac{(-\Delta_h \phi_h, \phi_h)}{\|  \phi_h \|_{L^2\II}^2}
=  \min_{\phi_h\in X_h}\frac{(\nabla \phi_h, \nabla \phi_h)}{\|  \phi_h \|_{L^2\II}^2}
 \ge \min_{\phi\in H_0^1}\frac{(\nabla \phi, \nabla \phi)}{\|  \phi \|_{L^2\II}^2}  = \lambda_1.
\end{equation}
Analogue to \eqref{eqn:sol-rep}, the solution to the semidiscrete problem \eqref{eqn:semi-1} could be written as
\begin{equation}\label{eqn:sol-rep-semi}
\begin{aligned}
u_h(t) &:=
\FF_h(t) \begin{bmatrix}
P_h a\\ P_h b
\end{bmatrix}
+ \int_0^t E_h(t-s)P_h f_h(s)\, \d s \\
&:= \begin{bmatrix}F_h &  \bar F_h  \end{bmatrix}
\begin{bmatrix}
P_h a\\ P_h b
\end{bmatrix}
+  \int_0^t E_h(t-s)P_h f_h(s)\, \d s.
\end{aligned}
\end{equation}
where the solution operators $ F(t)$, $\bar F(t)$ and $E(t)$ are respectively defined by
\begin{equation}\label{eqn:FE-ML-semi}
\begin{aligned}
F_h(t) v_h        &=  \sum_{j=1}^J E_{\alpha,1}(-\lambda_j^h t^\alpha)(v_h,\fy_j^h)\fy_j^h,
\quad \bar F_h(t) v_h =  \sum_{j=1}^J tE_{\alpha,2}(-\lambda_j^h t^\alpha)(v_h,\fy_j^h)\fy_j^h,\\
E_h(t) v_h         &=  \sum_{j=1}^J t^{\alpha-1}E_{\alpha,\alpha}(-\lambda_j^h t^\alpha)(v_h,\fy_j^h)\fy_j^h
\end{aligned}
\end{equation}
for any $v_h \in X_h$. By Laplace Transform, we have the following integral representations of the solution operators:
\begin{equation}\label{eqn:FE-LAP-semi}
\begin{aligned}
F_h(t) &= \frac{1}{2\pi i }\int_\contour e^{zt} z^{\alpha-1}(z^\alpha-\Delta_h)^{-1}dz,
\quad \bar F_h(t) &= \frac{1}{2\pi i }\int_\contour e^{zt} z^{\alpha-2}(z^\alpha-\Delta_h)^{-1}dz,\\
E_h(t) &= \frac{1}{2\pi i}\int_\contour  e^{zt} (z^\alpha-\Delta_h)^{-1} dz.
 \end{aligned}
\end{equation}

The following Lemma provides an error estimate of the semidiscrete approximation \eqref{eqn:semi-1} with trivial source $f\equiv 0$.
See \cite[Theorem 3.2]{JinLazarovZhou:SISC2016} for detailed proof.

\begin{lemma}
\label{lem:forward-error}
Let $u$ and $u_h$ are the solutions to \eqref{eqn:fde} and \eqref{eqn:semi-1}, respectively, with $a,b\in \dH q$ and
$f\equiv0$.  Then there holds that
\begin{equation*}
\begin{aligned}
\| {(u-u_h)(t)}\|_\L2Om&\le ch^2t^{-\alpha(2-q)/2} \left( \|a\|_{\dH q} + t \|b\|_{\dH q}\right).
\end{aligned}
\end{equation*}
\end{lemma}

\subsection{Semidiscrete scheme for solving backward problem}
In order to solve the inverse problem,
we apply the following regularized semidiscrete scheme: find $\tu_h^\delta(t)\in X_h$ such that
\begin{align}\label{eqn:back-semi-1}
\begin{aligned}
\Dal \tu^\delta_h-\Delta_h\tu^\delta_h& = 0,\ T\ge t>0,\\
 -\gamma  \tu_h^\delta(0) + \tu_h^\delta(T_1) &= P_hg_1^\delta , ~~
\gamma \partial_t \tu^\delta_h(0) +  \tu^\delta_h(T_2)  = P_hg_2^\delta.\\
\end{aligned}
\end{align}
We define the operator $\GG_h(T_1, T_2)$ as
\begin{equation}\label{eqn:GGh}
\GG_h(T_1, T_2) = \begin{bmatrix}  F_h(T_1)  &\bar F_h(T_1) \\
F_h(T_2) & \bar F_h(T_2) \end{bmatrix}.
\end{equation}
Then from \eqref{eqn:sol-rep-semi} the solutions can be represented as
\begin{equation} \label{eqn:back-semi-1-sol}
\begin{bmatrix}
\tu^\delta_h (0)\\ \partial_t \tu^\delta_h(0)
\end{bmatrix}
=  (\gamma \I +\GG_h(T_1, T_2))^{-1}
\begin{bmatrix}
P_hg_1^\delta\\P_hg_2^\delta
\end{bmatrix}\quad \text{and}\quad \tu^\delta_h (t)  = \FF_h(t) (\gamma \I +\GG_h(T_1, T_2))^{-1}
\begin{bmatrix}
P_hg_1^\delta\\P_hg_2^\delta
\end{bmatrix},
\end{equation}
where the operator $\I$ is given by \eqref{eqn:I}.
Meanwhile, we shall introduce an auxiliary function $\tu_h(t)$, a semidiscrete solution satisfying
\begin{align}\label{eqn:back-semi-2}
 \begin{aligned}
\Dal \tu_h -\Delta_h\tu_h& = 0,\ T\ge t>0,\\
 -\gamma  \tu_h(0) + \tu_h(T_1) &= P_hg_1,\\
\gamma \partial_t \tu_h(0)+  \tu_h(T_2) &= P_hg_2.\\
\end{aligned}
\end{align}
Then we would write the solutions as
\begin{equation}\label{eqn:back-semi-2-sol}
\begin{bmatrix}
\tu_h (0)\\ \partial_t \tu_h(0)
\end{bmatrix}
=  (\gamma \I +\GG_h(T_1, T_2))^{-1}
\begin{bmatrix}
P_hg_1\\P_hg_2
\end{bmatrix}\quad \text{and}\quad \tu_h (t)  = {\FF_h(t)} (\gamma \I +\GG_h(T_1, T_2))^{-1}
\begin{bmatrix}
P_hg_1 \\P_hg_2
\end{bmatrix}.
\end{equation}

The next lemma confirms the invertibility of the operator $\gamma \I +\GG_h(T_1, T_2)$.

\begin{lemma}\label{lem:stab-semi}
Let $M(\lambda_1)$ be the constant defined in Lemma \ref{lem:unique},
and suppose that $T_2 > T_1 \ge  M(\lambda_1)$.
Then the operator $\gamma \I +\GG_h(T_1, T_2)$ is invertible.
Meanwhile, there holds for all $v_h, w_h \in X_h$
\begin{equation*}
\left\|\FF_h(t)(\gamma \I +\GG_h(T_1,T_2))^{-1}\begin{bmatrix}
v_h\\ w_h
\end{bmatrix}
\right\|_{L^2\II} \le  c \min(\gamma^{-1},t^{-\alpha} ) \Big( \|  v_h \|_{L^2\II} + \|  w_h \|_{L^2\II}\Big)
\end{equation*}
Meanwhile, we have
\begin{equation*}
\left\|(\gamma \I +\GG_h(T_1,T_2))^{-1}\begin{bmatrix}
v_h\\w_h
\end{bmatrix}
\right\|_{L^2\II} \le  c\gamma^{-1} \Big(\|v_h\|_{L^2\II} +\|w_h\|_{L^2\II}\Big)\begin{bmatrix}
1\\1
\end{bmatrix}.
\end{equation*}
\end{lemma}

\begin{proof}
By \eqref{eqn:est-psi} and the the fact \eqref{eqn:minmax}, we obtain for any $1\le j\le J$
\begin{equation}\label{eqn:esi-psi-h}
 |\psi(T_1,T_2;\lambda_j^h)| \ge \left|\frac{c(T_2-T_1)}{\Gamma(1-\alpha)\Gamma(2-\alpha)}
 \frac{1}{(\lambda_j^h)^2T_1^\alpha T_2^\alpha}\right| > 0 ,
 \end{equation}
where the constant $c$ is independent of $\lambda_j^h$, $T_1$ and $T_2$. Then by the assumption \eqref{ass:z0}
we have
\begin{equation}\label{eqn:esi-tpsi-h}
\begin{aligned}
 \tilde \psi(T_1,T_2;\lambda_j^h)  &=
\psi(T_1,T_2;\lambda_j^h)- \gamma^2+\gamma [E_{\alpha,1}(-\lambda_j^h T_1^\alpha)-T_2E_{\alpha,2}(-\lambda_j^hT_2^\alpha)]\\
&\le   -c\Big( (\lambda_j^h)^{-2} +  \gamma (\lambda_j^h)^{-1} +\gamma^2 \Big) <0.
\end{aligned}
\end{equation}
and hence the operator $\gamma\I + \GG_h(T_1, T_2)$  is invertible. Finally, the desired two stability estimates follows by the
same argument in the proof of Lemma \ref{lem:regular:estimate:ope} with $p=q=0$.
\end{proof}

This lemma together with the representations \eqref{eqn:back-semi-1-sol} and \eqref{eqn:back-semi-2-sol}
implies the following estimate
\begin{corollary}\label{cor:err-tuh-tuhd}
Suppose that $M(\lambda_1)$ is the constant defined in Lemma \ref{lem:unique},
and $T_2 > T_1 \ge  M(\lambda_1)$.
Let $\tu_h^\delta(t)$ and $\tu_h(t)$ be the solutions of problems \eqref{eqn:back-semi-1} and \eqref{eqn:back-semi-2}.
Then there holds for all $0<t \le T$
\begin{equation*}
 \|(\tu_h^\delta-\tu_h)(t)\|_{\L2Om} \le c\delta \min(\gamma^{-1}, t^{-\alpha}) \quad \text{and}\quad
\begin{bmatrix}
\|(\tu_h^\delta-\tu_h)(0)\|_{\L2Om}\\ \|\partial_t(\tu_h^\delta-\tu_h)(0)\|_{\L2Om}
\end{bmatrix}
\le c\delta \gamma^{-1} \begin{bmatrix}1\\1\end{bmatrix},
\end{equation*}
where $c$ is independent on $\delta$, $\gamma$, $h$ and $t$.
\end{corollary}

Next, we aim to derive a bound for the discretization error $\tu_h - \tu$. To this end, we need the following preliminary estimate.
\begin{lemma}\label{lem:est-tu-h}
Suppose that $M(\lambda_1)$ is the constant defined in Lemma \ref{lem:unique},
and $T_2 > T_1 \ge  M(\lambda_1)$.
Let $\tu$ be the solution to the backward regularization problem
\eqref{eqn:back-3}. Then there holds for $0\le q \le 2$
\begin{equation*}
\|  (\E_h*\Delta_h(P_h-R_h)\tu )(t) \|_{ {\L2Om}} \le c h^2 t^{-\alpha(2-q)/2}\Big( \|  \tu(0) \|_{\dH q} +   t  \|  \partial_t\tu(0) \|_{\dH q} \Big)
\end{equation*}
\end{lemma}

\begin{proof}
Let $w_h$ be the solution to the semidiscrete problem
\begin{align}\label{eqn:back-semi-w}
\begin{aligned}
\Dal w_h-\Delta_h w_h& = 0,\ T\ge t>0,\\
w_h(0)   = P_h \tu(0) ,~
\partial_t w_h(0) &= P_h \partial_t \tu(0).\\
\end{aligned}
\end{align}
Then Lemma \ref{lem:forward-error} implies the estimate
\begin{equation}\label{eqn:est-tu-01}
\|  (w_h-\tu) (t) \| \le c h^2 t^{-\alpha(2-q)/2} \Big( \|  \tu(0) \|_{\dH q} +   t  \|  \partial_t\tu(0) \|_{\dH q} \Big).
\end{equation}
Meanwhile, we apply the following splitting
$$ (w_h - \tu )(t) = (w_h - P_h \tu) (t)+ (P_h \tu - \tu) (t)= : \zeta(t) + \rho(t).$$
From the approximation of $L^2$ projection \eqref{estimate:PR} and the regularity estimate in Lemma \ref{lem:reg-u}, we arrive at
\begin{equation}\label{eqn:est-tu-02}
\|\rho(t)\|_\L2Om  \le ch^2 \|\tu(t)\|_{\dH 2}  \le ch^2  t^{-\alpha(2-q)/2}\Big(\|  \tu(0) \|_{\dH q} +   t  \|  \partial_t\tu(0) \|_{\dH q} \Big).
\end{equation}
Moreover, we observe that the function $\zeta(t)$ satisfies
\begin{align*}
\begin{aligned}
\Dal \zeta(t)-\Delta_h  {\zeta}(t)& = \Delta_h (P_h - R_h) \tu(t),\ T\ge t>0,\\
\zeta(0)   = 0 ,~
\partial_t  {\zeta} (0) &= 0.\\
\end{aligned}
\end{align*}
Then \eqref{eqn:sol-rep-semi} indicates the representation
$\zeta(t) =  (\E_h*\Delta_h(P_h-R_h)\tu) (t)$.
Then the desired result follows immediately from \eqref{eqn:est-tu-01}, \eqref{eqn:est-tu-02} and the triangle inequality.
\end{proof}

Then we are ready to state a key lemma providing an estimate for the discretization error $\tu_h - \tu$.

\begin{lemma}\label{lem:err-tuh-tu}
Assume that $a,b\in L^2\II$. Let $\tu$ be the solution to the regularized  {problem
\eqref{eqn:back-3}} and $\tu_h$ be the solution to the
corresponding semidiscrete problem \eqref{eqn:back-semi-2}, then there holds for all $0<t \le T$
\begin{equation*}
\|(\tu_h-\tu)(t)\|_{\L2Om}
\le ch^2 \min(\gamma^{-1},t^{-\alpha}) \Big(\|a\|_{L^2} +\|b\|_{L^2\II}\Big)
\end{equation*}
and
\begin{equation*}
\|(\tu_h-\tu)(0)\|_{\L2Om}+
\|\partial_t(\tu_h-\tu)(0)\|_{\L2Om}
\le ch^2\gamma^{-1} \Big(\|a\|_{L^2} +\|b\|_{L^2\II}\Big)
\end{equation*}
where both $c$ are independent on $\gamma$, $h$ and $t$.
\end{lemma}
\begin{proof}
First of all, for $t\in(0, T] $, we use the splitting 
\begin{equation*}
 (\tu_h-\tu)(t) =  (\tu_h-P_h\tu)(t)+ (P_h\tu-\tu)(t)=:\zeta(t)+ \rho(t).
\end{equation*}
From the approximation property of the $L^2$-projection in \eqref{estimate:PR}, we arrive at
\begin{equation*}
\begin{aligned}
\|\rho(t)\|_\L2Om & \le ch^2 \|\tu(t)\|_{\dH 2}  \le ch^2  \min\big(\gamma^{-1},t^{-\alpha}\big)
 \big( \|g_1\|_{\dH 2}+ \|g_2\|_{\dH 2}\big) \\
 &\le  ch^2  \min\big(\gamma^{-1},t^{-\alpha}\big)
 \big( \|a\|_{L^ 2\II}+ \|b\|_{L^2\II}\big) \\
 \end{aligned}
\end{equation*}
where the second inequality  follows from \eqref{eqn:sol-back-4} and Lemma \ref{lem:regular:estimate:ope} (with $p=q=2$),
and the last inequality follows from the regularity estimate in Lemma \ref{lem:reg-u}.

Now we turn to the term $\zeta = \tu_h-P_h\tu$ which satisfies the error equation
\begin{align*}
\left\{\begin{aligned}
\Dal \zeta - \Delta_h \zeta & =\Delta_h(P_h-R_h)\tu(t),~~ T\ge t>0,\\
 {-\gamma} \zeta(0) + \zeta(T_1) &= 0 ,\\
\gamma \partial_t \zeta(0) +  \zeta(T_2) &= 0.\\
\end{aligned}\right.
\end{align*}
From solution representation we have
\begin{equation*}
\begin{bmatrix}
\zeta(T_1)\\ \zeta(T_2)
\end{bmatrix} = \GG_h(T_1,T_2)\vecfunc{\zeta}{0} +
\begin{bmatrix}
(\E_h*\Delta_h(P_h-R_h)\tu)(T_1)\\
(\E_h*\Delta_h(P_h-R_h)\tu)(T_2)
\end{bmatrix}.
\end{equation*}
Then we add $( {-\gamma\zeta(0)},\gamma\partial_t\zeta(0))^T$ at both sides and derive
\begin{equation}\label{eqn:xi-0}
\begin{aligned}
\begin{bmatrix}
0\\0
\end{bmatrix}
=  (\gamma \I+\GG_h(T_1,T_2))
\vecfunc{\zeta}{0}
+ \begin{bmatrix}
(\E_h* \Delta_h(P_h-R_h)\tu)(T_1)\\
(\E_h* \Delta_h(P_h-R_h)\tu)(T_2)
\end{bmatrix}.
\end{aligned}
\end{equation}
This immediately implies a representation to $ {\zeta(t)}$:
\begin{align*}
 \zeta(t) &= \FF_h(t)\vecfunc{\zeta}{0} + (\E_h*\Delta_h(P_h-R_h)\tu)(t)\\
 &=
-\FF_h(t)(\gamma \I+\GG_h(T_1,T_2))^{-1}
\begin{bmatrix}
(\E_h*\Delta_h(P_h-R_h)\tu)(T_1)\\
(\E_h*\Delta_h(P_h-R_h)\tu)(T_2)
\end{bmatrix}
+  (\E_h*\Delta_h(P_h-R_h)\tu)(t)  \\
&=:I_1(t) +I_2(t).
\end{align*}
Then Lemmas \ref{lem:stab-semi} and \ref{lem:est-tu-h} lead to the estimate for all $t\in(0,T]$
\begin{equation*}
\begin{aligned}
\|I_1(t)\|_\L2Om
 &\le c\min(\gamma^{-1},t^{-\alpha})\sum_{i=1}^2\|(\E_h*\Delta_h(P_h-R_h)\tu)(T_i)\|_\L2Om \\
 &\le ch^2 \min(\gamma^{-1},t^{-\alpha})(\|\tu(0)\|_{L^2\II}+\|\partial_t\tu(0)\|_{L^2\II}).\\
\end{aligned}
\end{equation*}
Recalling Corollary \ref{cor:tu-reg} with $q=0$, we derive for all $t\in(0,T]$
\begin{equation*}
\begin{aligned}
\|I_1(t)\|_\L2Om
 &\le ch^2  \min(\gamma^{-1},t^{-\alpha})(\|a\|_{L^2\II}+\|b\|_{L^2\II}).\\
\end{aligned}
\end{equation*}
Similarly, using Lemma \ref{lem:est-tu-h} with $q=2$ and Corollary \ref{cor:tu-reg} with $q=2$, we bound the  term $I_2$ by
\begin{equation*}
\begin{aligned}
\|I_2(t)\|_\L2Om
&\le ch^2 (\|\tu(0)\|_{\dH2}+\|\partial_t\tu(0)\|_{\dH2}) \le ch^2 \gamma^{-1} (\|a\|_{L^2\II}+\|b\|_{L^2\II})
\end{aligned}
\end{equation*}
for all $t\in(0,T]$. Meanwhile, using Lemma \ref{lem:est-tu-h} with $q=0$ and Corollary \ref{cor:tu-reg} with $q=0$, we have
\begin{equation*}
\begin{aligned}
\|I_2(t)\|_\L2Om
&\le ch^2 t^{-\alpha} (\|\tu(0)\|_{L^2}+\|\partial_t\tu(0)\|_{L^2}) \le ch^2 t^{-\alpha} (\|a\|_{L^2\II}+\|b\|_{L^2\II}).
\end{aligned}
\end{equation*}
Therefore we conclude that
$$  \|   (\tu - \tu_h)(t)  \|_{L^2\II} \le c h^2 \min(\gamma^{-1}, t^{-\alpha}) \Big(\|a\|_{L^2} +\|b\|_{L^2\II}\Big).  $$

Similarly, for $t=0$, the relation \eqref{eqn:xi-0} implies
\begin{align*}
\vecfunc{\zeta}{0}=
&- (\gamma \I+\GG_h(T_1,T_2))^{-1}
\begin{bmatrix}
(\E_h*\Delta_h(P_h-R_h)\tu)(T_1)\\
(\E_h*\Delta_h(P_h-R_h)\tu)(T_2)
\end{bmatrix}
\end{align*}
Then  Lemmas \ref{lem:regular:estimate:ope} (with $p=0$ and $q=0$), \ref{lem:stab-semi}
(with $q=0$) and Corollary \ref{lem:est-tu-h} (with $q=0$) yield
\begin{align*}
 \| \zeta(0) \|_{L^2\II}+ \| \partial_t\zeta(0) \|_{L^2\II}
 &\le c\gamma^{-1}
\sum_{i=1}^2\|\E_h*\Delta_h(P_h-R_h)\tu(T_i) \|_{L^2\II} \\
&\le c h^2 \gamma^{-1} \Big(  \|\tu(0)\|_{L^2\II}+\|\partial_t\tu(0)\|_{L^2\II} \Big) \\
&\le c h^2 \gamma^{-1} \Big(  \|a\|_{L^2\II}+\|b\|_{L^2\II} \Big).
\end{align*}
This completes the proof of the lemma.
\end{proof}

Then Lemma \ref{lem:est-tu-u}, Corollary \ref{cor:err-tuh-tuhd} and Lemma \ref{lem:err-tuh-tu} would lead to the following error estimate.
\begin{theorem}
\label{thm:err-semi}
Assume that $a,b\in \dH q$, $q\in[0,2]$. Let $u$ be the solution to the problem \eqref{eqn:fde}
and $\tu_h^\delta$ be the solution to the regularized semidiscrete problem \eqref{eqn:back-semi-1}, then there holds
\begin{equation*}
\begin{aligned}
\|(\tu_h^\delta-u)(t)\|_\L2Om
\le c\Big[\gamma \min(\gamma^{-(1-\frac q2)}, t^{-(1-\frac q2)\alpha})+(h^2+\delta)\min(\gamma^{-1},t^{-\alpha})\Big]\quad \forall t\in (0,T],\\
\end{aligned}
\end{equation*}
and
\begin{equation*}
\begin{aligned}
\|(\tu_h^\delta-u)(0)\|_\L2Om + \|\partial_t(\tu_h^\delta-u)(0)\|_\L2Om \le c\Big[\gamma^\frac{q}{2}+ \gamma^{-1} (h^2+\delta)\Big].
\end{aligned}
\end{equation*}
where  $c$ dependes on $T_1$, $T_2$, $a$ and $b$, but is always independent of $h$, $\gamma$, $\delta$ and $t$.
\end{theorem}

\begin{remark}
For $a,b\in \dH q$ and $t \ge t_0$, then Theorem \ref{thm:err-semi} provides an estimate
\begin{equation*}
\begin{aligned}
\|(\tu_h^\delta-u)(t)\|_\L2Om
\le c(\gamma + (h^2+\delta) ).
\end{aligned}
\end{equation*}
With the \textsl{a priori} choice of parameter $\gamma\sim \delta$ and $h\sim \sqrt{\delta}$,
we obtain the optimal convergence rate $ \|(\tu_h^\delta-u)(t)\|_\L2Om \le c \delta$.
For $t=0$, according to Theorem \ref{thm:err-semi}, we choose
$\gamma\sim \delta^{\frac{2}{2+q}}$ and $h\sim \sqrt{\delta}$ to obtain the best convergence rate
\begin{equation*}
\begin{aligned}
\|(\tu_h^\delta-u)(0)\|_\L2Om + \|\partial_t(\tu_h^\delta-u)(0)\|_\L2Om \le c \delta^{\frac{q}{2+q}}.
\end{aligned}
\end{equation*}
In case that $q=0$,  we can also show the convergence, provided a suitable choice of parameters.
According to Lemma \ref{lem:est-tu-u}, Corollary \ref{cor:err-tuh-tuhd} and Theorem \ref{thm:err-semi}, there holds for any $s\in(0,1]$
\begin{equation*}
\|(\tu_h^\delta-u)(0)\|_\L2Om + \|\partial_t(\tu_h^\delta-u)(0)\|_{\dH {-s}} \rightarrow0,\ \quad \text{as } \delta,\gamma,h \to 0,~ \frac{\delta}{\gamma} \to 0~\text{and} ~\frac{h^2}{\gamma}\to 0.
\end{equation*}
\end{remark}\vskip15pt

\section{Fully discrete scheme and error analysis}\label{sec:fully}
Now we intend to propose and analyze a fully discrete scheme for approximately solving the backward diffusion-wave problem.

\subsection{Fully discrete scheme for the direct problem}
To begin with, we introduce the fully discrete scheme for the direct problem.
We divide the time interval $[0,T]$ into a uniform grid, with $ t_n=n\tau$, $n=0,\ldots,N$, and $\tau=T/N$ being
the time step size. In case that $\fy(0)=0$ and $\fy'(0)=0$, we approximate the Riemann-Liouville fractional derivative
\begin{equation*}
^{RL}\Dal \varphi(t)=\frac{1}{\Gamma(2-\alpha)}\frac{\d^2}{\d t^2}\int_0^t(t-s)^{1-\alpha}\varphi(s)\d s
\end{equation*}
by the backward Euler convolution quadrature (with $\varphi_j=\varphi(t_j)$) \cite{Lubich:1986, JinLiZhou:2017sisc}:
\begin{equation*}
  ^{RL}\Dal \varphi(t_n) \approx \tau^{-\alpha} \sum_{j=0}^nb_j\varphi_{n-j}:=\bar\partial_\tau^\alpha \varphi_n,\quad\mbox{ with } \sum_{j=0}^\infty b_j\xi^j = (1-\xi)^\alpha.
\end{equation*}

The fully discrete scheme for problem \eqref{eqn:fde-0} reads: find ${U_n}\in X_h$ such that
\begin{equation}\label{eqn:fully}
\bDal (U_n-  {P_ha} - t_n  {P_hb})-\Delta_h U_n= P_h f(t_n),\quad n=1,2,\ldots,N,
\end{equation}
with the initial condition $U_0=P_h a \in X_h$.
Here we use the relation between Riemann-Liouville and Caputo fractional derivatives with $\alpha\in(1,2)$ \cite[p. 91]{KilbasSrivastavaTrujillo:2006}:
$$ \Dal u (t_n) = \Dal (u(t_n) - a - t b) = {^{RL}\Dal} (u(t_n) - a - t b) \approx \bar \partial_\tau^\alpha (u(t_n) - a - t b). $$
 By means of discrete Laplace transform, the fully discrete solution $U_n$ is given by
\begin{equation}\label{eqn:Sol-expr-uhtau}
\begin{aligned}
  U_n &= \FF_{h,\tau}^n  \begin{bmatrix}
P_h a\\ P_h b
\end{bmatrix}  + \tau \sum_{k=1}^{n} E_{h,\tau}^{n-k} P_hf(t_k) :=\begin{bmatrix}
F_{h,\tau}^n & \bar F_{h,\tau}^n
\end{bmatrix}
\begin{bmatrix}
P_h a \\ P_h b
\end{bmatrix}
 + \tau \sum_{k=1}^{n} E_{h,\tau}^{n-k} P_hf(t_k) ,
 \end{aligned}
\end{equation}
 with $n=1,2,\ldots,N$, where the  discrete operators $F_{h,\tau}^n$, $\bar F_{h,\tau}^n$ and $E_{h,\tau}^n$ are respectively defined by  \cite{JinLiZhou:2017sisc}
\begin{equation}\label{eqn:FE_ht-0}
\begin{aligned}
F_{h,\tau}^n &= \frac{1}{2\pi\mathrm{i}}\int_{\Gamma_{\theta,\sigma}^\tau } e^{zt_n} {e^{-z\tau}} \delta_\tau(e^{-z\tau})^{\alpha-1}({ \delta_\tau(e^{-z\tau})^\alpha}-\Delta_h)^{-1}\,\d z,\\
\bar F_{h,\tau}^n &= \frac{1}{2\pi\mathrm{i}}\int_{\Gamma_{\theta,\sigma}^\tau } e^{zt_n}  {e^{-z\tau}}\delta_\tau(e^{-z\tau})^{\alpha-2}({ \delta_\tau(e^{-z\tau})^\alpha}-\Delta_h)^{-1}\,\d z ,\\
E_{h,\tau}^n &= \frac{1}{2\pi\mathrm{i}}\int_{\Gamma_{\theta,\sigma}^\tau } e^{zt_n} ({ \delta_\tau(e^{-z\tau})^\alpha}-\Delta_h)^{-1}\,\d z ,
\end{aligned}
\end{equation}
with $\delta_\tau(\xi)=(1-\xi)/\tau$ and the contour
$\Gamma_{\theta,\sigma}^\tau :=\{ z\in \Gamma_{\theta,\sigma}:|\Im(z)|\le {\pi}/{\tau} \}$ where $\theta\in(\pi/2,\pi)$ is close to $\pi/2$.
(oriented with an increasing imaginary part).
The next lemma gives elementary properties of the kernel $\delta_\tau(e^{-z\tau})$. The detailed proof has been given in
\cite[Lemma B.1]{JinLiZhou:2017sisc}.
\begin{lemma}\label{lem:delta}
For a fixed $\theta'\in(\pi/2,\pi/\alpha)$, there exists $\theta\in(\pi/2,\pi)$
and positive constants $c,c_1,c_2$ $($independent of $\tau$$)$ such that for all $z\in \Gamma_{\theta,\sigma}^\tau$
\begin{equation*}
\begin{aligned}
& c_1|z|\leq
|\delta_\tau(e^{-z\tau})|\leq c_2|z|,\qquad
\delta_\tau(e^{-z\tau})\in \Sigma_{\theta'}. \\
& |\delta_\tau(e^{-z\tau})-z|\le c\tau |z|^{2},\qquad
 |\delta_\tau(e^{-z\tau})^\alpha-z^\alpha|\leq c\tau |z|^{1+\alpha}.
 \end{aligned}
\end{equation*}
\end{lemma}


In case that $f\equiv0$, with the spectral decomposition,  we can write
\begin{equation}\label{eqn:sol-uhtau-1}
  U_n = F_{h,\tau}^n P_h a + \bar F_{h,\tau}^n P_h b
  = \sum_{j=1}^J \Big[F_{\tau}^n(\lambda_j^h) (a,\fy_j^h){\varphi_j^h} + \bar F_{\tau}^n(\lambda_j^h) (b,\fy_j^h){\varphi_j^h} \Big]
\end{equation}
where $F_{\tau}^n(\lambda_j^h)$ and $\bar F_{\tau}^n(\lambda_j^h)$ are the solutions to the discrete initial value problems
$$\bDal [F_{\tau}^n(\lambda_j^h)-1] + \lambda_j^h F_{\tau}^n(\lambda_j^h) = 0,\quad \mbox{with}\quad F_{\tau}^0(\lambda_j^h) = 1$$
and
$$\bDal [\bar F_{\tau}^n(\lambda_j^h)-t_n] + \lambda_j^h \bar F_{\tau}^n(\lambda_j^h) = 0,\quad \mbox{with}\quad \bar F_{\tau}^0(\lambda_j^h) = 0$$
respectively.
{From \eqref{eqn:FE_ht-0}}, we write ${F_{\tau}^n(\lambda_j^h)}$ and $\bar F_{\tau}^n(\lambda_j^h)$ as
\begin{equation}\label{eqn:FE_ht}
\begin{aligned}
F_{\tau}^n(\lambda_j^h) &= \frac{1}{2\pi\mathrm{i}}\int_{\Gamma_{\theta,\sigma}^\tau } e^{zt_n}  {e^{-z\tau}}\delta_\tau(e^{-z\tau})^{\alpha-1}({ \delta_\tau(e^{-z\tau})^\alpha}+\lambda_j^h)^{-1}\,\d z\\
\bar F_{\tau}^n(\lambda_j^h) &= \frac{1}{2\pi\mathrm{i}}\int_{\Gamma_{\theta,\sigma}^\tau } e^{zt_n}  {e^{-z\tau}}\delta_\tau(e^{-z\tau})^{\alpha-2}({ \delta_\tau(e^{-z\tau})^\alpha}+\lambda_j^h)^{-1}\,\d z.
\end{aligned}
\end{equation}

Next we derive several useful properties of $ F_\tau^n(\lambda_j^h)$ and $\bar F_\tau^n(\lambda_j^h)$. The proof is standard but lengthy, and
hence deferred to the appendix.
\begin{lemma}\label{lem:fully-approx}
Let $F_\tau^n(\lambda)$ and $\bar F_\tau^n(\lambda)$ be defined as in \eqref{eqn:FE_ht}.
Then for $\lambda>0$, there holds for $1\le n\le N$,
\begin{equation}\label{eqn:es-01}
\big|E_{\alpha,1}(-\lambda t_n^\alpha)-F_\tau^n(\lambda)\big|
+ t_n^{-1} \big| t_n E_{\alpha,2}(-\lambda t^\alpha)-\bar F_\tau^n(\lambda)\big|
\le  \frac{cn^{-1}}{1+\lambda t_n^\alpha}.
\end{equation}
Meanwhile, there holds
\begin{equation}\label{eqn:es-02}
\lambda^{-1} \Big(\big|E_{\alpha,1}(-\lambda t_n^\alpha) - F_{\tau}^n(\lambda)\big|
+ t_n^{-1}  \big| t_n E_{\alpha,2}(-\lambda t_n^\alpha) - \bar F_{\tau}^n(\lambda)\big| \Big) \le c\tau t_{n}^{\alpha-1}.
\end{equation}
Here  $c$ is the  generic positive constant independent of $\lambda$, $t$ and $\tau$.
\end{lemma}

Then Lemmas \ref{lem:ml} and \ref{lem:fully-approx}  leads to the following asymptotic behaviors of $F_\tau^n(\lambda)$ and $\bar F_\tau^n(\lambda)$.

\begin{corollary}\label{cor:fully-asymp}
Let $F_\tau^n(\lambda)$ and $\bar F_\tau^n(\lambda)$ be defined as in \eqref{eqn:FE_ht}. Then
there exists $\tau_0>0$ such that for all $\tau \in (0, \tau_0)$, $\lambda>\lambda_1$ and  $t_n \ge M(\lambda_1)$
 {\begin{align*}
-c_0  \lambda^{-1} t_n^{-\alpha}  \le  F_\tau^n(\lambda)  \le  -c_1\lambda^{-1} t_n^{-\alpha}\quad \text{and}\quad
\tilde c_0  \lambda^{-1} t_n^{1-\alpha}  \le \bar F_\tau^n(\lambda)   \le \tilde c_1\lambda^{-1} t_n^{1-\alpha},
\end{align*}}
 with positive constants $c_0$, $c_1$, $\tilde c_0$, $\tilde c_1$ independent of $\lambda$, $t$ and $\tau$.
\end{corollary}


Now we define two integers $N_1$ and $N_2$ such that $N_1 \tau = T_1$ and $N_2 \tau = T_2$, and define
\begin{equation}\label{eqn:Ght}
\GG_{h,\tau} (T_1, T_2)
= \begin{bmatrix}F_{h,\tau}^{N_1} & \bar F_{h,\tau}^{N_1} \\ F_{h,\tau}^{N_2}  & \bar F_{h,\tau}^{N_2} \end{bmatrix},\quad
G_{\tau}(T_1,T_2;\lambda_j^h) = \begin{bmatrix}
F_\tau^{N_1}(\lambda_j^h) & \bar F_\tau^{N_1}(\lambda_j^h)\\
F_\tau^{N_2}(\lambda_j^h) & \bar F_\tau^{N_2}(\lambda_j^h)
\end{bmatrix}.
\end{equation}
Then according to \eqref{eqn:sol-uhtau-1}, we have the representation
\begin{equation*}
\begin{aligned}
 \begin{bmatrix}U_{N_1} \\ U_{N_2}  \end{bmatrix} &= \GG_{h,\tau} (T_1, T_2)  \begin{bmatrix} P_h a \\  P_h b  \end{bmatrix}
 =\sum_{j=1}^J
G_{\tau}(T_1,T_2;\lambda_j^h)
\begin{bmatrix} ( a,
\fy_j^h)\fy_j^h\\ ( b,\fy_j^h)\fy_j^h \end{bmatrix}=
\sum_{j=1}^J \begin{bmatrix}
F_\tau^{N_1}(\lambda_j^h) & \bar F_\tau^{N_1}(\lambda_j^h)\\
F_\tau^{N_2}(\lambda_j^h) & \bar F_\tau^{N_2}(\lambda_j^h)
\end{bmatrix}
\begin{bmatrix} ( a,\fy_j^h)\fy_j^h\\ ( b,\fy_j^h)\fy_j^h \end{bmatrix}.
\end{aligned}
 \end{equation*}

The next lemma provides the invertibility of $\gamma\I + \GG_{h,\tau} (T_1, T_2) $.

\begin{lemma}\label{lem:inv-fully}
Let $M(\lambda_1)$ be the constant defined in Lemma \ref{lem:unique},
and suppose that $T_2 > T_1 \ge  M(\lambda_1)$. Then the operator $\gamma\I + \GG_{h,\tau} (T_1, T_2) $ is invertible,
and there holds for $v_h, w_h \in X_h$
\begin{equation*}
\Big\|   \FF_{h,\tau}^n (\gamma\I +  \GG_{h,\tau} (T_1, T_2)) ^{-1}   \begin{bmatrix}v_h\\ w_h\end{bmatrix}  \Big\|_{L^2\II}
\le c \min(\gamma^{-1}, t_n^{-\alpha})\Big(\|  v_h \|_{L^2\II} + \| w_h \|_{L^2\II}\Big)
\end{equation*}
and
\begin{equation*}
\Big\|    (\gamma\I +  \GG_{h,\tau} (T_1, T_2)) ^{-1}   \begin{bmatrix}v_h\\ w_h\end{bmatrix}  \Big\|_{L^2\II}
\le c  \gamma^{-1} \Big(\|  v_h \|_{L^2\II} + \| w_h \|_{L^2\II}\Big)  \begin{bmatrix}1\\ 1\end{bmatrix}  .
\end{equation*}
\end{lemma}

\begin{proof}
Let $\psi_\tau(T_1,T_2;\lambda_j^h)$ be the determinant of $G_\tau(T_1,T_2;\lambda_j^h)$. We define
\begin{equation*}
 {\t \psi_\tau(T_1,T_2;\lambda_j^h) = \psi_\tau(T_1,T_2;\lambda_j^h)-\gamma^2+\gamma[F_{h,\tau}^{N_1} -\bar F_{h,\tau}^{N_2}]},
\end{equation*}
Then from Lemma \ref{lem:fully-approx}  and Corollary \ref{cor:fully-asymp} we have for {$\lambda>\lambda_1$}
\begin{equation*}
\begin{aligned}
&|\psi_\tau(T_1,T_2;\lambda) - \psi(T_1,T_2;\lambda)|\\
&\le |(F_{\tau}^{N_1}(\lambda)-E_{\alpha,1}(-\lambda T_1^\alpha)\bar F_{\tau}^{N_2}(\lambda)|+|E_{\alpha,1}(-\lambda T_1^\alpha)(\bar F_{\tau}^{N_2}(\lambda)-T_2E_{\alpha,2}(-\lambda T_2^\alpha))|\\
&+|(T_1E_{\alpha,2}(-\lambda T_1^\alpha)-\bar F_{\tau}^{N_1}(\lambda))F_{\tau}^{N_2}(\lambda)|+|T_1E_{\alpha,2}(-\lambda T_1^\alpha)(E_{\alpha,1}(-\lambda T_2^\alpha)-F_{\tau}^{N_2}(\lambda))|\le c\frac{\tau}{\lambda^2 T_1^\alpha T_2^\alpha},
\end{aligned}
\end{equation*}
Combining \eqref{eqn:est-psi} with the fact $\lambda_j^h \ge \lambda_1^h>\lambda_1$ by \eqref{eqn:minmax} we have
$
 {\psi_\tau(T_1,T_2;\lambda_j^h) \le c(\lambda_j^h)^{-2} T_1^{-\alpha} T_2^{-\alpha}<0.}
$
This together with the Corollary \ref{cor:fully-asymp} leads to
\begin{equation}\label{eqn:determinant-2-c}
 | {\t \psi_\tau(T_1,T_2;\lambda_j^h)} |
\ge c\Big( (\lambda_j^h)^{-2} +  \gamma (\lambda_j^h)^{-1} +\gamma^2 \Big) > 0,
\end{equation}
where $c$ is only dependent on $T_1$, $T_2$ and $\alpha$.
Therefore, the operator $\gamma\I + \GG_{h,\tau}(T_1, T_2)$  is invertible. Finally, the desired stability estimates follows by an argument similar to
the proof of Lemma \ref{lem:regular:estimate:ope} with $p=q=0$ and Corollary \ref{cor:fully-asymp}.
\end{proof}

\subsection{Fully discrete scheme for the inverse problem}
Now, we propose a fully discrete scheme for solving the backward diffusion-wave problem.
Given $g_1^\delta$ and $g_2^\delta$, we look for $\tilde a_{h,\tau}^\delta$,
$\tilde b_{h,\tau}^\delta$ and $\tilde U_n^\delta\in X_h$ with $n=1,2,\ldots,N$ such that
\begin{equation}\label{eqn:fully-back}
\begin{aligned}
\bDal (\tilde U_n^\delta - \tilde a_{h,\tau}^\delta - t_n \tilde b_{h,\tau}^\delta)-\Delta_h\tilde U_n^\delta&=0,\quad \forall ~n=1,2,\ldots,N,\\
  {-\gamma}\tilde  a_{h,\tau}^\delta +\tilde U_{N_1}^\delta&=P_hg_1^\delta,\\
 \gamma\tilde  b_{h,\tau}^\delta +\tilde U_{N_2}^\delta&=P_hg_2^\delta
\end{aligned}
\end{equation}
with $\tilde  U_0^\delta = \tilde a_{h,\tau}^\delta$.
Then by Lemma \ref{lem:inv-fully}, the problem \eqref{eqn:fully-back} is uniquely solvable, and $\tilde U_n^\delta$ could be represented as
\begin{equation}\label{eqn:tildeUdn}
  \tilde U_n^\delta = \FF_{h,\tau}^n  \begin{bmatrix}\tilde  a_{h,\tau}^\delta\\ \tilde  b_{h,\tau}^\delta\end{bmatrix}
  = \FF_{h,\tau}^n  (\gamma\I +  \GG_{h,\tau} (T_1, T_2)) ^{-1}  \begin{bmatrix}  P_h g_1^\delta\\ P_h g_2^\delta\end{bmatrix}
\end{equation}
while $\tilde a_{h,\tau}^\delta$ and  $\tilde b_{h,\tau}^\delta$ could be written as
\begin{equation}\label{eqn:abd}
  \begin{bmatrix}\tilde  a_{h,\tau}^\delta\\ \tilde  b_{h,\tau}^\delta\end{bmatrix}
  =  (\gamma\I +  \GG_{h,\tau} (T_1, T_2)) ^{-1}  \begin{bmatrix}  P_h g_1^\delta\\ P_h g_2^\delta\end{bmatrix}.
\end{equation}

Similarly, we could define auxiliary functions $\tilde a_{h,\tau}$,
$\tilde b_{h,\tau}$ and $\tilde U_n\in X_h$ with $n=1,2,\ldots,N$ such that
\begin{equation}\label{eqn:fully-back-2}
\begin{aligned}
\bDal (\tilde U_n - \tilde a_{h,\tau} - t_n \tilde b_{h,\tau})-\Delta_h\tilde U_n&=0,\quad \forall ~n=1,2,\ldots,N,\\
 {- \gamma}\tilde  a_{h,\tau} +\tilde U_{N_1}&=P_hg_1 , ~~
 \gamma\tilde  b_{h,\tau} +\tilde U_{N_2} =P_hg_2
\end{aligned}
\end{equation}
with $\tilde  U_0  = \tilde a_{h,\tau}$.
Then the function $\tilde U_n^\delta$ could be represented as
\begin{equation}\label{eqn:tildeUn}
  \tilde U_n = \FF_{h,\tau}^n  \begin{bmatrix}\tilde  a_{h,\tau} \\ \tilde  b_{h,\tau} \end{bmatrix}
  = \FF_{h,\tau}^n  (\gamma\I +  \GG_{h,\tau} (T_1, T_2)) ^{-1}  \begin{bmatrix}  P_h g_1\\ P_h g_2 \end{bmatrix}
\end{equation}
while $\tilde a_{h,\tau} $ and  $\tilde b_{h,\tau} $ could be written as
\begin{equation}\label{eqn:ab}
  \begin{bmatrix}\tilde  a_{h,\tau} \\ \tilde  b_{h,\tau} \end{bmatrix}
  =  (\gamma\I +  \GG_{h,\tau} (T_1, T_2)) ^{-1}  \begin{bmatrix}  P_h g_1\\ P_h g_2 \end{bmatrix}.
\end{equation}

Then Lemma \ref{lem:inv-fully} immediately implies following estimates for $\tilde a_{h,\tau} - \tilde a_{h,\tau}^\delta$,
$\tilde b_{h,\tau}-\tilde b_{h,\tau}^\delta$ and $\tilde U_n-\tilde U_n^\delta$.

\begin{lemma}\label{lem:fully-est-01}
Let $M(\lambda_1)$ be the constant defined in Lemma \eqref{lem:unique},
and suppose that $T_2 > T_1 \ge  M(\lambda_1)$. Let $\tilde a_{h,\tau}^\delta$,
$\tilde b_{h,\tau}^\delta$ and $\tilde U_n^\delta $ be solutions to \eqref{eqn:fully-back},
and $\tilde a_{h,\tau}$,
$\tilde b_{h,\tau}$ and $\tilde U_n$ be solutions to \eqref{eqn:fully-back-2}.
Then there holds
\begin{equation*}
\|  \tilde U_n-\tilde U_n^\delta  \|_{L^2\II}
\le c \delta \min(\gamma^{-1}, t_n^{-\alpha})\big(\|  a \|_{L^2\II} + \| b \|_{L^2\II}\big)
\end{equation*}
and
\begin{equation*}
 \| \tilde a_{h,\tau} - \tilde a_{h,\tau}^\delta  \|_{L^2\II}   +  \| \tilde b_{h,\tau} - \tilde b_{h,\tau}^\delta  \|_{L^2\II}
\le c \delta \gamma^{-1} \big(\|  a \|_{L^2\II} + \| b \|_{L^2\II}\big)   .
\end{equation*}
\end{lemma}

Next, we aim to compare two auxiliary problems, i.e. \eqref{eqn:fully-back-2} and \eqref{eqn:back-semi-2}.

\begin{lemma}\label{lem:fully-est-02}
Let $M(\lambda_1)$ be the constant defined in Lemma \ref{lem:unique},
and suppose  $T_2 > T_1 \ge  M(\lambda_1)$. Let
$\tilde a_{h,\tau}$, $\tilde b_{h,\tau}$ and $\tilde U_n$ be the solutions to \eqref{eqn:fully-back-2},
and $\tu_h (t)$ be the solution to the semidiscrete problem \eqref{eqn:back-semi-2}.
Then
\begin{equation*}
 \| \tilde a_{h,\tau} - \tu_h(0)  \|_{L^2\II}   +  \| \tilde b_{h,\tau} - \partial_t\tu_h(0)  \|_{L^2\II}
\le c \big(\tau + h^2\gamma^{-1}\big) \big(\|  a  \|_{L^2\II} + \| b \|_{L^2\II}\big)   .
\end{equation*}
and
\begin{equation*}
\|  \tilde U_n-\tu_h(t_n)  \|_{L^2\II}
\le c \big(\tau t_n^{\alpha-1}+ h^2\big) \min(\gamma^{-1}, t_n^{-\alpha})\big(\|  a\|_{L^2\II} + \| b \|_{L^2\II}\big).
\end{equation*}
\end{lemma}
\begin{proof}
Using representations \eqref{eqn:ab} and \eqref{eqn:back-semi-2-sol}, we derive

\begin{equation*}
\begin{aligned}
  \begin{bmatrix}  \tilde a_{h,\tau} - \tu_h(0)  \\ \tilde b_{h,\tau} - \partial_t\tu(0)  \end{bmatrix}
 &=
\Big(\gamma\I +  \GG_{h,\tau}(T_1,T_2)\Big) ^{-1}                                                                         	\begin{bmatrix} P_h g_1  \\ P_h g_2 \end{bmatrix}
-\Big( \gamma\I +  \GG_{h}(T_1,T_2)\Big) ^{-1}
    \begin{bmatrix} P_h g_1  \\ P_h g_2 \end{bmatrix}\\
& =
\Big(\gamma\I +  \GG_{h,\tau}(T_1,T_2)\Big) ^{-1}                                                                         	\begin{bmatrix} (P_h-R_h) g_1  \\ (P_h-R_h) g_2 \end{bmatrix}
+
 \Big( \gamma\I +  \GG_{h}(T_1,T_2)\Big) ^{-1}
    \begin{bmatrix} (R_h-P_h) g_1  \\ (R_h-P_h) g_2 \end{bmatrix}\\
&+
\Big(\GG_{h}(T_1,T_2) -\GG_{h,\tau}(T_1,T_2)\Big)\Big(\gamma\I +  \GG_{h,\tau}(T_1,T_2)\Big) ^{-1} \Big( \gamma\I +  \GG_{h}(T_1,T_2)\Big) ^{-1}
\begin{bmatrix} R_h g_1 \\ R_hg_2 \end{bmatrix}\\
& = I_1+I_2 + I_3.
\end{aligned}
\end{equation*}
Using Lemmas \ref{lem:stab-semi} and \ref{lem:inv-fully} we can obtain an estimate for $I_1$ and $I_2$:
\begin{equation*}
\begin{aligned}
\|I_1\|_\L2Om + \|I_2\|_\L2Om&\le ch^2\gamma^{-1} (\|g_1\|_{\dH 2}+\|g_2\|_{\dH 2})
	\begin{bmatrix} 1 \\ 1	\end{bmatrix}
 \le ch^2\gamma^{-1} (\|a\|_\L2Om+\|b\|_\L2Om)
	\begin{bmatrix} 1 \\ 1	\end{bmatrix},
\end{aligned}
\end{equation*}
where in the last inequality we use the regularity estimate in Lemma \ref{lem:reg-u}.
Then for the term $I_3$,   we apply  Lemma \ref{lem:fully-approx} and Corollary \ref{cor:fully-asymp} again to derive
\begin{equation*}
\begin{aligned}
\| {I_3}\|_\L2Om^2
&\le  c \sum_{j=1}^J  \frac{(R_hg_1,\fy_j^h)^2+(R_hg_2,\fy_j^h)^2} {\t \psi_\tau(T_1,T_2;\lambda_j^h)^2\t \psi(T_1,T_2;\lambda_j^h)^2( \lambda_j^h T_1^\alpha)^6 N_1^{2}}
		\begin{bmatrix} 1\\1 \end{bmatrix}\\
&\le c\tau^2 \sum_{j=1}^J   (\lambda_j^h)^2
		\Big((R_hg_1,\fy_j^h)^2+(R_hg_2,\fy_j^h)^2\Big)
		\begin{bmatrix} 1\\1 \end{bmatrix}.
\end{aligned}
\end{equation*}
Noting that $\Delta_hR_h = P_h\Delta$, then we apply Lemma \ref{lem:reg-u} to obtain
\begin{equation}\label{eqn:Phg}
\begin{aligned}
\|\Delta_h R_hg_1\|_\L2Om+\|\Delta_h R_hg_2\|_\L2Om&=  \|P_h\Delta g_1\|_\L2Om+ \|P_h\Delta g_2\|_\L2Om\\
&\le (\|\Delta g_1\|_\L2Om + \|\Delta g_2\|_\L2Om)\\
&\le c (\|a\|_\L2Om+\|b\|_\L2Om),
\end{aligned}
\end{equation}
In conclusion, we obtain
\begin{equation*}
\|\t a_{h,\tau} - \tu_h(0)\|_\L2Om+\|\t b_{h,\tau}-\partial_t \tu_h(0)\|_\L2Om \le c(\tau + h^2\gamma^{-1})(\|a\|_\L2Om+\|b\|_\L2Om).
\end{equation*}
Next, from \eqref{eqn:back-semi-1-sol} and \eqref{eqn:tildeUn} we derive the splitting that
\begin{equation*}
\begin{aligned}
&\t U_n - \tu_h(t_n) \\
&=
 \FF_{h,\tau}^n\Big(\gamma\I+\GG_{h,\tau}(T_1,T_2)\Big)^{-1}
 \begin{bmatrix}	P_hg_1 \\ P_hg_2	\end{bmatrix}
 - \FF_h(t_n)\Big(\gamma\I+\GG_h(T_1,T_2)\Big)^{-1}
\begin{bmatrix}	P_hg_1 \\ P_hg_2	\end{bmatrix}\\
&=
\bigg( \FF_{h,\tau}^n\Big(\gamma\I+\GG_{h,\tau}(T_1,T_2)\Big)^{-1}
 \begin{bmatrix}	(P_h-R_h)g_1 \\ (P_h-R_h)g_2	\end{bmatrix}
 + \FF_h(t_n)\Big(\gamma\I+\GG_h(T_1,T_2)\Big)^{-1}
\begin{bmatrix}	(R_h-P_h)g_1 \\ (R_h-P_h)g_2	\end{bmatrix}\bigg)\\
&\quad + \bigg(
\FF_{h,\tau}^n\Big(\gamma\I+\GG_{h,\tau}(T_1,T_2)\Big)^{-1}
		 \begin{bmatrix}	R_hg_1 \\ R_hg_2	\end{bmatrix}
-\FF_h(t_n)\Big(\gamma\I+\GG_h(T_1,T_2)\Big)^{-1}	
		\begin{bmatrix}	R_hg_1 \\ R_hg_2	\end{bmatrix}\bigg)\\
&=: I_1 + I_2.
\end{aligned}
\end{equation*}
To bound the first term $I_1$, we apply approximation properties of  $P_h$ and $R_h$, Lemmas \ref{lem:stab-semi}
and \ref{lem:inv-fully}, and the argument \eqref{eqn:Phg} to obtain
\begin{equation*}
\begin{aligned}
\|I_1\|_\L2Om &\le  ch^2\min(\gamma^{-1},t_n^{-\alpha}) (\|\Delta_h R_hg_1\|_\L2Om+\|\Delta_h R_hg_2\|_\L2Om)\\
&\le ch^2\min(\gamma^{-1},t_n^{-\alpha})(\|g_1\|_{\dH 2}+\|g_2\|_{\dH 2})\\
&\le  ch^2\min(\gamma^{-1},t_n^{-\alpha}) (\|a\|_\L2Om+\|b\|_\L2Om),
\end{aligned}
\end{equation*}
where in the last inequality we use the regularity estimate in Lemma \ref{lem:reg-u}.
For the other term $I_2$, we split it into three parts
\begin{equation*}
\begin{aligned}
I_2 &=
 {\gamma (\FF_{h,\tau}^n-\FF_h(t_n))\I}\Big(\gamma\I+\GG_{h,\tau}(T_1,T_2)\Big)^{-1}\Big(\gamma\I+\GG_h(T_1,T_2)\Big)^{-1}	
		\begin{bmatrix} R_hg_1 \\ R_hg_2 \end{bmatrix}\\
&\quad +
\FF_{h,\tau}^n(\GG_h(T_1,T_2)-\GG_{h,\tau}(T_1,T_2))\Big(\gamma\I+\GG_{h,\tau}(T_1,T_2)\Big)^{-1}\Big(\gamma\I+\GG_h(T_1,T_2)\Big)^{-1}	
		\begin{bmatrix} R_hg_1 \\ R_hg_2 \end{bmatrix}\\
&\quad +
\GG_{h,\tau}(T_1,T_2)(\FF_{h,\tau}^n-\FF_h(t_n))\Big(\gamma\I+\GG_{h,\tau}(T_1,T_2)\Big)^{-1}\Big(\gamma\I+\GG_h(T_1,T_2)\Big)^{-1}		
		\begin{bmatrix} R_hg_1 \\ R_hg_2 \end{bmatrix}  =: \sum_{i=1}^3 I_{2,i}.
\end{aligned}
\end{equation*}
Then we intend to establish bounds for those terms one by one. For the term $I_{2,1}$, we apply the spectral decomposition to obtain
\begin{equation*}
\begin{aligned}
I_{2,1} &= \sum_{j=1}^J \gamma
\begin{bmatrix}  {-(F_\tau^n (\lambda_j^h) - E_{\alpha,1}(-\lambda_j^h t_n^\alpha))}
& \bar F_\tau^n (\lambda_j^h) - t_nE_{\alpha,2}(-\lambda_j^h t_n^\alpha)   \end{bmatrix} \\
&  \qquad  \qquad   {\t \psi_\tau(T_1,T_2;\lambda_j^h)^{-1}}  \begin{bmatrix}\gamma + \bar F_\tau^{N_2}(\lambda_j^h)  & -\bar F_\tau^{N_1}(\lambda_j^h)\\
  -F_\tau^{N_2}(\lambda_j^h)  &  {-\gamma} + F_\tau^{N_1}(\lambda_j^h)   \end{bmatrix} \\
&   \qquad  \qquad  \t \psi(T_1,T_2;\lambda_j^h)^{-1}  \begin{bmatrix}  \gamma + T_2E_{\alpha,2}(-\lambda_j^h T_2^\alpha)    & -T_1E_{\alpha,2}(-\lambda_j^h T_1^\alpha)\\
-  E_{\alpha,1}(-\lambda_j^h T_2^\alpha)   &  {-\gamma} + E_{\alpha,1}(-\lambda_j^h T_1^\alpha) \end{bmatrix}
\begin{bmatrix}(R_h g_1, \fy_j^h) \fy_j^h \\  (R_h g_2, \fy_j^h) \fy_j^h   \end{bmatrix}
\end{aligned}
\end{equation*}
Using Corollary \ref{cor:fully-asymp}  and the estimate \eqref{eqn:determinant-2-c}, we obtain
\begin{equation*}
\begin{aligned}
|\t \psi_\tau (T_1,T_2;\fy_j^h)|^{-1}
\begin{bmatrix} |\gamma + \bar F_\tau^{N_2}(\lambda_j^h)|  & |-\bar F_\tau^{N_1}(\lambda_j^h)|\\
|  -F_\tau^{N_2}(\lambda_j^h) | &  | {-\gamma} + F_\tau^{N_1}(\lambda_j^h)|   \end{bmatrix} 
&\le
\frac{c \lambda_j}{1+\gamma\lambda_j}\begin{bmatrix}
1 & 1\\
1& 1
\end{bmatrix}
\le c \min(\gamma^{-1}, \lambda_j^h)  \begin{bmatrix}
1 & 1\\
1& 1
\end{bmatrix}.
\end{aligned}
\end{equation*}
This, the first estimate in Lemma \ref{lem:fully-approx} and the estimates \eqref{bound:regular:mat:G} and \eqref{eqn:Phg} imply
\begin{equation*}
\begin{aligned}
\|I_{2,1}\|_{L^2\II}^2 &\le c\tau^2 t_n^{-2}
\sum_{j=1}^J \left(\frac{\lambda_j^h}{1+\lambda_j^h t_n^\alpha }\right)^2
\big((R_hg_1,\fy_j^h)^2+(R_hg_2,\fy_j^h)^2\big)\\
&\le c\tau^2 t_n^{-2}
\sum_{j=1}^J(\lambda_j^h)^2
\big((R_hg_1,\fy_j^h)^2+(R_hg_2,\fy_j^h)^2\big),\\
&=c\tau^2 t_n^{-2}
\big(\|\Delta_hR_hg_1\|_\L2Om^2+\|\Delta_hR_hg_2\|_\L2Om^2 \big)\\
&\le c\tau^2 t_n^{-2}  \big(\|a\|_\L2Om+\|b\|_\L2Om\big)
\end{aligned}
\end{equation*}
while the second estimate in Lemma \ref{lem:fully-approx} indicates
\begin{equation*}
\begin{aligned}
\|I_{2,1}\|_{L^2\II}^2 &\le c\tau^2 t_n^{2\alpha-2} \gamma^{-2}
\sum_{j=1}^J (\lambda_j^h)^2
\big((R_hg_1,\fy_j^h)^2+(R_hg_2,\fy_j^h)^2\big)\\
&\le  c\tau^2 t_n^{2\alpha-2} \gamma^{-2}\big(\|a\|_\L2Om+\|b\|_\L2Om\big)
\end{aligned}
\end{equation*}
 Combining this two estimates we arrive at
 \begin{equation*}
\begin{aligned}
\|I_{2,1}\|_{L^2\II} &\le c\tau t_n^{\alpha-1} \min(\gamma^{-1}, t_n^{-\alpha}) \big(\|a\|_\L2Om+\|b\|_\L2Om\big)
\end{aligned}
\end{equation*}
The estimates for $I_{2,2}$ and $I_{2,3}$ follows analogously.
\end{proof}

Then we combine Lemmas \ref{lem:est-tu-u}, \ref{lem:err-tuh-tu}, \ref{lem:fully-est-01} and \ref{lem:fully-est-02} to obtain the following error estimate
for the fully discrete scheme \eqref{eqn:fully-back}.
\begin{theorem}
\label{thm:err-fully}
Let $M(\lambda_1)$ be the constant defined in Lemma \ref{lem:unique},
and suppose that $T_2 > T_1 \ge  M(\lambda_1)$. Let
$\tilde a_{h,\tau}^\delta$, $\tilde b_{h,\tau}^\delta$ and $\tilde U_n^\delta$ be the solutions to \eqref{eqn:fully-back},
and $u$ be the exact solution to the problem \eqref{eqn:fde}.
If $a,b\in \dH q$ with $q\in[0,2]$, then there holds
\begin{equation*}
 \| \tilde a_{h,\tau}^\delta- a \|_{L^2\II}   +  \| \tilde b_{h,\tau}^\delta - b \|_{L^2\II}
\le c \big(\gamma^{\frac{q}{2}} + \tau + (h^2+\delta)\gamma^{-1}\big)   .
\end{equation*}
and
\begin{equation*}
\|  \tilde U_n^\delta-u(t_n)  \|_{L^2\II}
\le c \Big[\gamma \min(\gamma^{-(1-\frac q2)}, t_n^{-(1-\frac q2)\alpha})+ (\tau t_n^{\alpha-1}+ h^2 +\delta \big) \min(\gamma^{-1}, t_n^{-\alpha}) \Big].
\end{equation*}
Moreover, if $a,b\in L^2\II$, then for any $s\in(0,1]$
\begin{equation*}
 \| \tilde a_{h,\tau} ^\delta- a \|_{L^2\II}   +  \| \tilde b_{h,\tau}^\delta - b \|_{H^{-s}\II}   \rightarrow0,\qquad \text{as}~~
\gamma, \tau \rightarrow0,~ \frac{\delta}{\gamma} \rightarrow0, ~ \frac{h}{\gamma} \rightarrow0
\end{equation*}
In the estimate, the constant $c$ may depend on $T_1$, $T_2$, $T$, $a$ and $b$, but is always independent of $\tau$, $h$, $\gamma$, $\delta$ and $t$.
\end{theorem}

\section{Numerical results}\label{sec:numerics}
In this section, we illustrate our theoretical results by
presenting some one- and two-dimensional examples. Throughout, we consider the observation data
\begin{equation*}
g_\delta=u(T)+\varepsilon \delta \sup_{x\in\Omega}u(x,T)\quad \text{and}
\quad g_\delta=u(T)+\varepsilon \delta \sup_{x\in\Omega}u(x,T),
\end{equation*}
$\varepsilon$ is generated following the standard Gaussian distribution and $\delta$ denotes the (relative) noise level.
Throughout this section, we fix $T_1=1$ and $T_2 = 1.2$.
To examine  the \textsl{a priori} estimates in Sections \ref{sec:semi} and \ref{sec:fully},
we begin with a one-dimensional diffusion-wave
model \eqref{eqn:fde} in the unit interval  $\Omega=(0,1)$.
We use the standard piecewise linear FEM with uniform mesh size $h=1/(J+1)$
for the space discretization, and the backward Euler convolution quadrature method with
uniform step size $\tau=T/N$ for the time discretization.
To solve the discrete system \eqref{eqn:fully-back}, we
apply the following direct method by spectral decomposition.
For the uniform mesh size $h=1/(J+1)$, we let $x_i = ih$ for all $i=0,1,\ldots,J+1$. Then the eigenvalues and eigenfunctions of $-\Delta_h$ have the closed form:
\begin{equation}\label{eqn:semi-eigenpair}
\lambda^h_j=\frac{6}{h^2}\frac{1-\cos(j\pi h)}{2+\cos(j\pi h)},
\quad \fy^h_j(x_i)=\sqrt{2}\sin(j\pi x_i),\quad i,j=1,2,\cdots,J.
\end{equation}
We compute the observation data $u(T_1)$, $u(T_2)$ and reference solution $u(t)$
by using the semidiscrete scheme with a very fine mesh size, i.e., $h=1/2000$.

For each example, we measure the errors of semidiscrete scheme
\begin{align*}
e_{\text{ini},s} &=  \frac{\| \tu_h^\delta(0)-a\|_\L2Om}{\|a\|_\L2Om}+\frac{\|\partial_t \tu_h^\delta(0) - b\|_\L2Om}{\|b\|_\L2Om},~~
e_{s}(t)  = \frac{\|\tu_h^\delta(t)-u(t)\|_\L2Om}{\|u(t)\|_\L2Om}\quad \text{for some}~~ t>0,
\end{align*}
and  the errors of fully discrete  scheme
\begin{align*}
e_{\text{ini},f} &=  \frac{\|\tilde a_{h,\tau}^\delta-a\|_\L2Om}{\|a\|_\L2Om}+\frac{\|\tilde b_{h,\tau}^\delta - b\|_\L2Om}{\|b\|_\L2Om},
e_{f}^n = \frac{\|\tilde U_n^\delta-u(t_n)\|_\L2Om}{\|u(t_n)\|_\L2Om}\quad \text{for some}~~ {n\ge 1}.
\end{align*}
The normalization enables us to observe the behaviour of the error with respect to $\alpha$ and $t$.
\vskip5pt

 \paragraph{\bf Example (1): smooth initial data.} We start with the smooth initial condition
$$a(x)= -\sin(\pi x),\quad b(x) = x(1-x) \in \dH2 = H^2\II\cap H_0^1\II,$$
and source term $f\equiv0$.
We compute the solution of the regularized semidiscrete scheme \eqref{eqn:back-semi-1-sol},
by using the  formulae
\begin{equation*} 
\begin{aligned}
&\begin{bmatrix}
\tu_h^\delta(0) \\ \partial_t \tu_h^\delta(0)
\end{bmatrix}
 =\sum_{j=1}^J \t \psi(T_1,T_2;\lambda_j^h)^{-1}
 \begin{bmatrix}
 \gamma + T_2E_{\alpha,2}(-\lambda_j^h T_2^\alpha) & -T_1E_{\alpha,2}(-\lambda_j^hT_1^\alpha)\\
 -E_{\alpha,1}(-\lambda_j^h T_2^\alpha) & -\gamma +E_{\alpha,1}(-\lambda_j^h T_1^\alpha)
 \end{bmatrix}
 \begin{bmatrix}
 (P_hg_1^\delta,\fy_j^h)\fy_j^h \\ (P_hg_2^\delta,\fy_j^h)\fy_j^h
 \end{bmatrix},\\
&\tu_h^\delta(t)
= \sum_{j=1}^J \tilde\psi(T_1,T_2;\lambda_j^h)^{-1}
\begin{bmatrix}
E_{\alpha,1}(-\lambda_j^ht^\alpha)&tE_{\alpha,2}(-\lambda_j^ht^\alpha)
\end{bmatrix}\\
&\hspace{.3\textwidth}
\begin{bmatrix}
\gamma+T_2E_{\alpha,2}(-\lambda_j^hT_2^\alpha)& {-T_1E_{\alpha,2}(-\lambda_j^hT_1^\alpha)}\\  {-E_{\alpha,1}(-\lambda_j^hT_2^\alpha)}
 &  {-\gamma}+E_{\alpha,1}(-\lambda_j^hT_1^\alpha)
\end{bmatrix}
\begin{bmatrix}
(P_hg_1^\delta,\fy_j^h)\fy_j^h\\(P_hg_2^\delta,\fy_j^h)\fy_j^h
\end{bmatrix},
\end{aligned}
\end{equation*}
where  $(\lambda_j^h,\fy_j^h)$, for $j=1,\cdots, {J}$ are given by \eqref{eqn:semi-eigenpair}.
To accurately evaluate
the Mittag-Leffler functions, we employ the numerical algorithm developed in \cite{Seybold:2008}.

\begin{figure}[htbp]
\centering
\begin{subfigure}{.4\textwidth}
\centering
\includegraphics[scale=0.4]{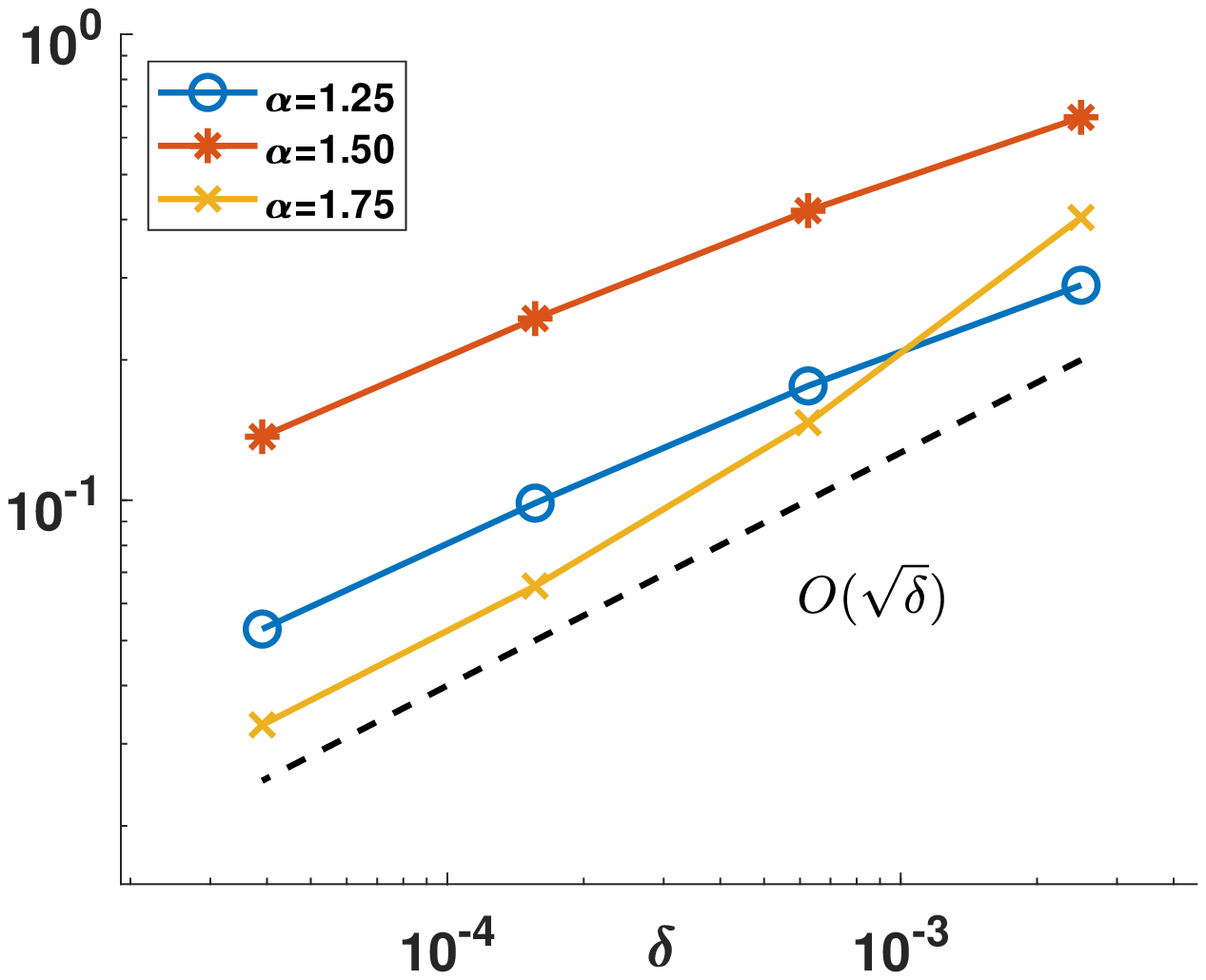}
\caption{$e_{\text{ini},s}$.}
\end{subfigure}%
\begin{subfigure}{.4\textwidth}
\centering
\includegraphics[scale=0.4]{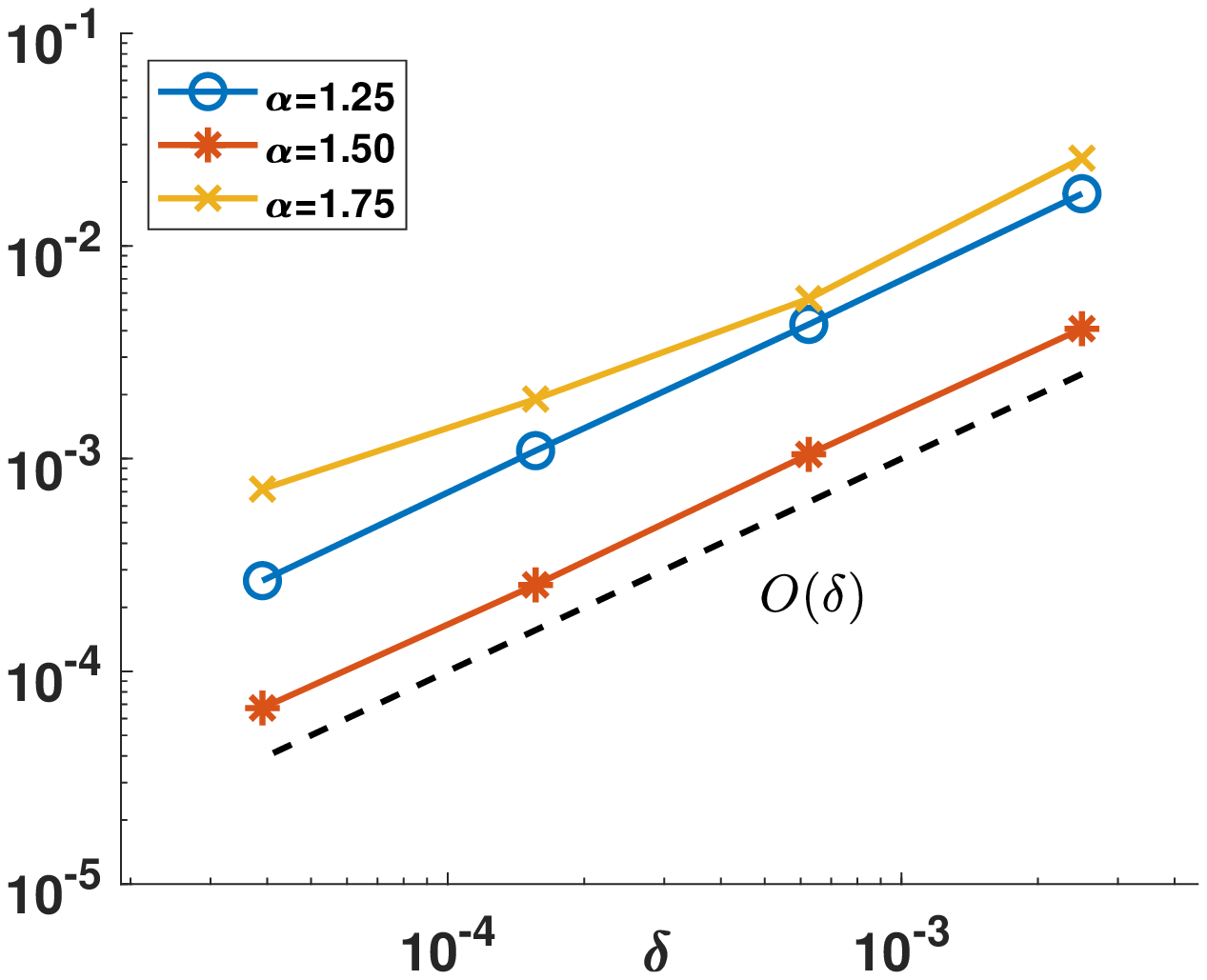}
\caption{$e_{s}(t)$ with $t=0.5$.}
\end{subfigure}\\%
\begin{subfigure}{.4\textwidth}
\centering
\includegraphics[scale=0.4]{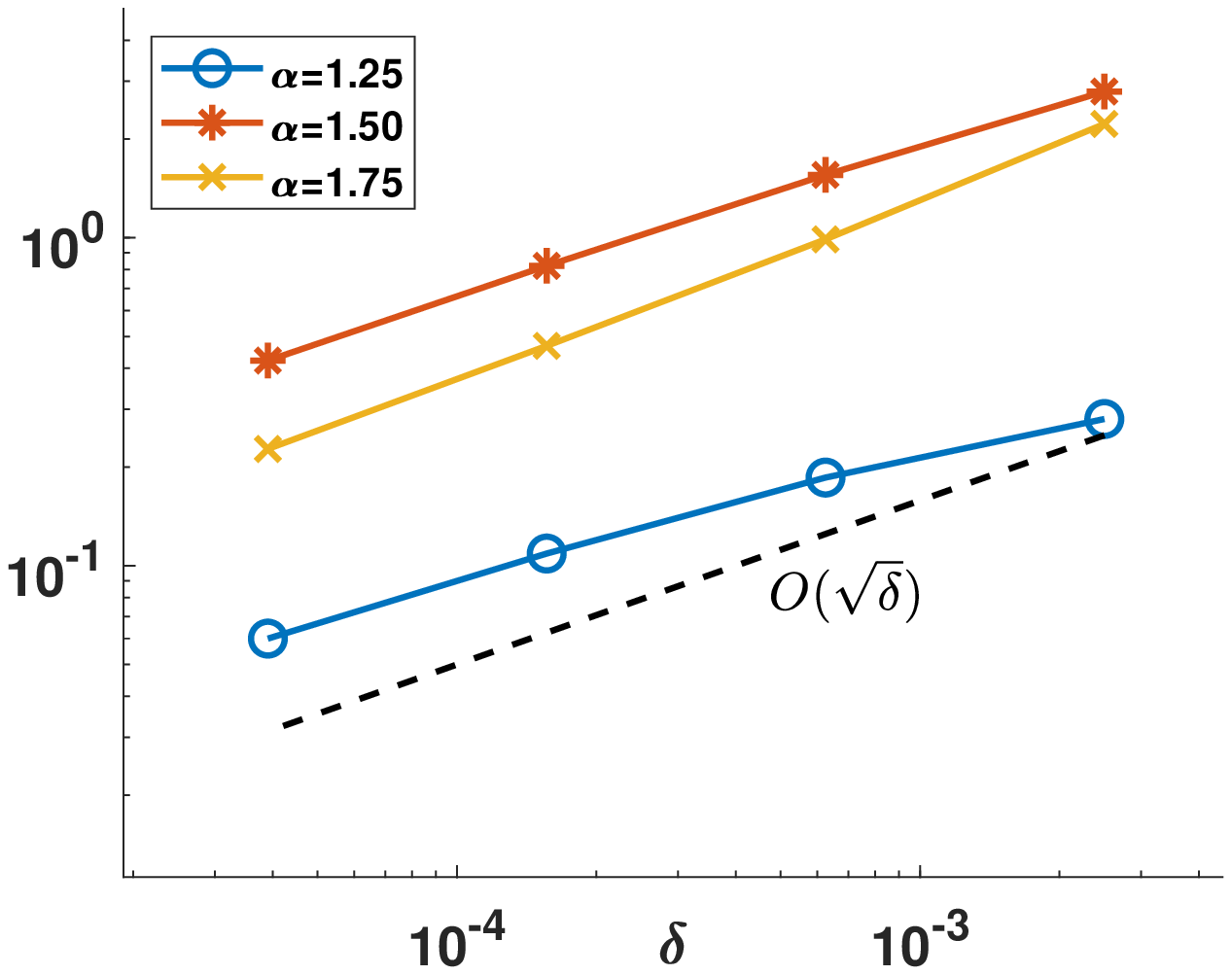}
\caption{$e_{ini,f}$.}
\end{subfigure}%
\begin{subfigure}{.4\textwidth}
\centering
\includegraphics[scale=0.4]{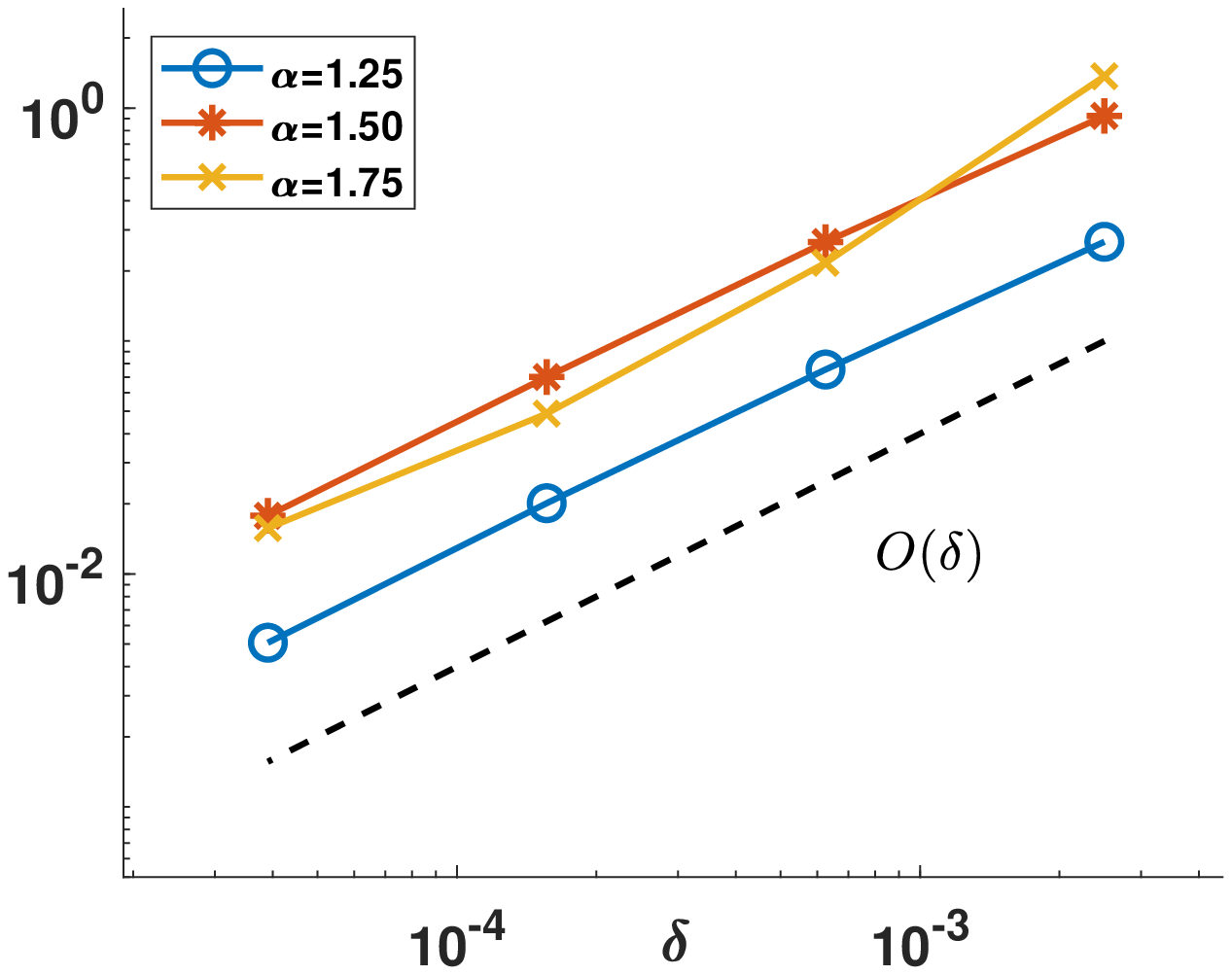}
\caption{$e_f^n$ with $t_n=0.5$.}
\end{subfigure}
\caption{Example (1): semidiscrete errors ((a) and (b)) and fully discrete errors ((c) and (d)).
(a): $h=\sqrt{\delta}$ and $\gamma=\sqrt{\delta}/12,\sqrt{\delta},\sqrt{\delta}/2$ for $\alpha=1.25,1.5,1.75$  respectively.
(b):  $h=\sqrt{\delta}$ and $\gamma=\sqrt{\delta}/5,\sqrt{\delta}/5,\sqrt{\delta}/2$  for $\alpha=1.25,1.5,1.75$respectively.
(c): $h=\sqrt{\delta}$, $\tau = \sqrt{\delta}/2$ and $\gamma = \sqrt{\delta}/10$,
$\sqrt{\delta}/10$, $\sqrt{\delta}/15$ for $\alpha=1.25,1.5,1.75$ respectively.
(d):  $h=\sqrt{\delta}$, $\tau = 10\delta$ and $\gamma = \delta$, $\delta/2$, $\delta/2$ for $\alpha=1.25,1.5,1.75$ respectively.}
\label{fig:smooth}
\end{figure}

By Theorem \ref{thm:err-semi}, we compute $\tu^\delta_h(0)$ and
$\partial_t \tu^\delta_h(0)$ by choosing the parameters $\gamma \sim \sqrt{\delta}$ and
and $h \sim \sqrt{\delta}$ for a given $\delta$, and expect a convergence of order $O(\sqrt{\delta})$.
For $t>0$, we compute $\tu_h^\delta(t)$ by choosing the parameters $h\sim \sqrt{\delta}$, $\gamma \sim \delta$
for a given $\delta$, and  expect a convergence of order $O(\delta)$.
In Figure \ref{fig:smooth} (a) and (b), we plot the errors of semidiscrete solutions 
with different fractional order $\alpha$.
Our numerical experiments fully support our theoretical results in Theorem \ref{thm:err-semi}.
It is interesting to observe that the error in case of $\alpha=1.5$ is bigger when reconstructing the initial condition, while
the error for $\alpha=1.5$ becomes smaller when we compute the solution at time level $t>0$.

Similarly, we compute the numerical solutions to the fully discrete scheme \eqref{eqn:fully-back}
by using the formulae
\begin{equation*} 
\begin{aligned}
&\begin{bmatrix}
\t a_{h,\tau}^\delta \\ \t b_{h,\tau}^\delta
\end{bmatrix}
 =\sum_{j=1}^J \t \psi_\tau(T_1,T_2;\lambda_j^h)^{-1}
 \begin{bmatrix}
 \gamma + \bar F_{h,\tau}^{N_2} & -\bar F_{h,\tau}^{N_1} \\
-F_{h,\tau}^{N_2} & -\gamma +F_{h,\tau}^{N_1}
 \end{bmatrix}
 \begin{bmatrix}
 (P_hg_1^\delta,\fy_j^h)\fy_j^h \\ (P_hg_2^\delta,\fy_j^h)\fy_j^h
 \end{bmatrix},\\
&\t U_n^\delta
= \sum_{j=1}^J \tilde\psi_\tau(T_1,T_2;\lambda_j^h)^{-1}
\begin{bmatrix}
F_{h,\tau}^{n} & \bar F_{h,\tau}^{n}
\end{bmatrix}
 \begin{bmatrix}
 \gamma + \bar F_{h,\tau}^{N_2} & -\bar F_{h,\tau}^{N_1} \\
F_{h,\tau}^{N_2} & -\gamma +F_{h,\tau}^{N_1}
 \end{bmatrix}
\begin{bmatrix}
(P_hg_1^\delta,\fy_j^h)\fy_j^h\\(P_hg_2^\delta,\fy_j^h)\fy_j^h
\end{bmatrix}.
\end{aligned}
\end{equation*}
Then Theorem \ref{thm:err-fully} implies that for $a,b\in \dH2$
\begin{equation*}
\|a_{h,\tau}^\delta-a\|_\L2Om+\|b_{h,\tau}^\delta-b\|_\L2Om \le c(\gamma + \tau + (h^2+\delta)\gamma^{-1}),
\end{equation*}
and
\begin{equation*}
\|\t U_n^\delta - u(t_n)\|_\L2Om \le c(\gamma +\tau + h^2+\delta),\quad \text{for a fixed}~~ t_n > 0.
\end{equation*}
Therefore, with a given noise level $\delta$, to recover the initial data $a$ and $b$, we choose parameters
$h\sim\sqrt{\delta}$, $\tau \sim \sqrt{\delta}$ and $\gamma \sim \sqrt{\delta}$,
while to approximate solution $u(t_n)$ with some $t_n>0$,
we let  $h\sim\sqrt{\delta}$, $\tau \sim \delta$, $\gamma \sim \delta $.
According to Theorem \ref{thm:err-fully}, we expect that the convergence rate for the error $e_{\text{ini},f}$ is
$O(\sqrt{\delta})$ while the error $e_{f}^n$ converges to zero as $O(\delta)$ for any fixed $t_n>0$.
They are fully supported by numerical results plotted in Figure \ref{fig:smooth} (c) and (d).
\begin{figure}[htbp]
\centering
\begin{subfigure}{.4\textwidth}
\centering
\includegraphics[scale=0.4]{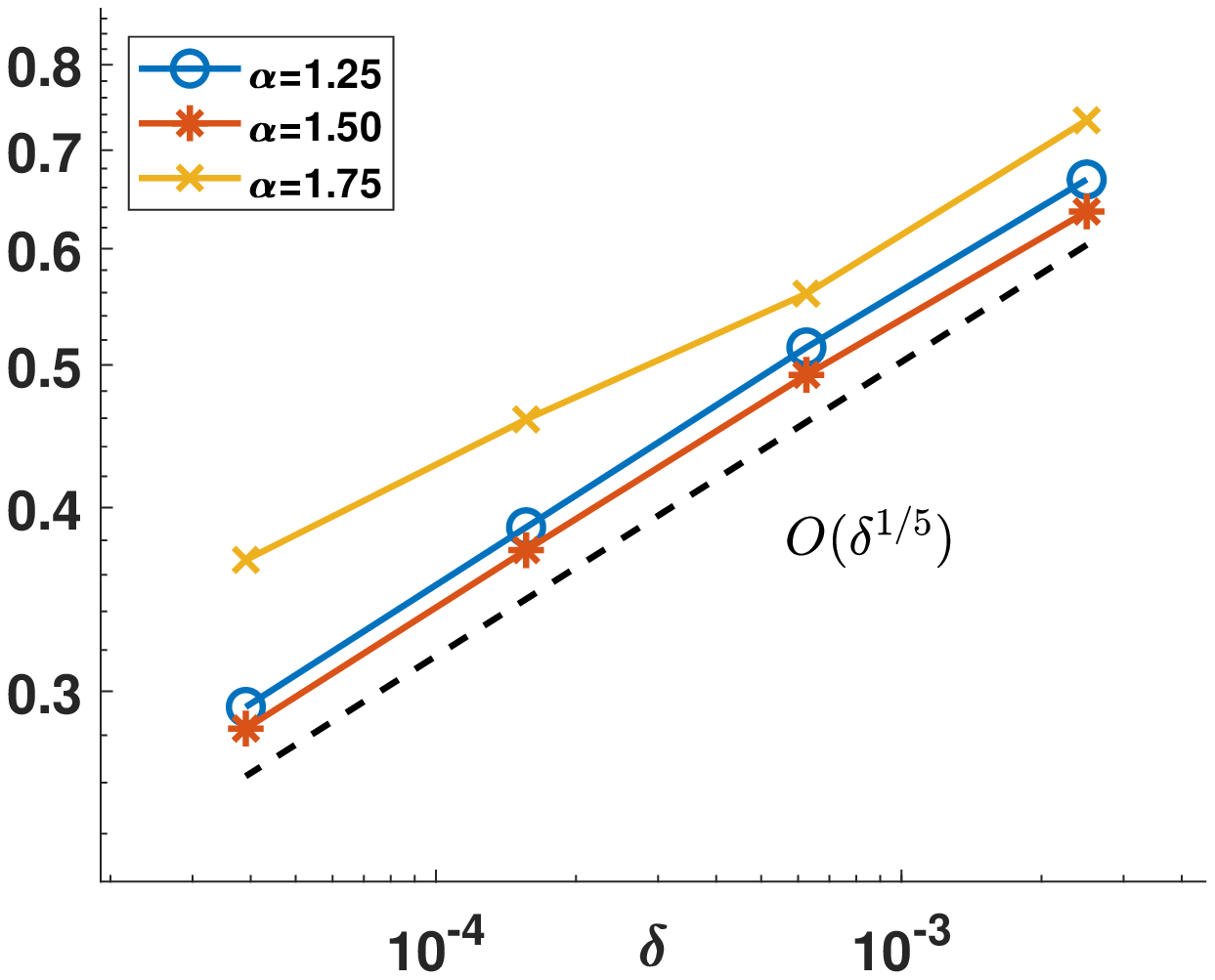}
\caption{$e_{ini,s}$.}
\end{subfigure}%
\begin{subfigure}{.4\textwidth}
\centering
\includegraphics[scale=0.4]{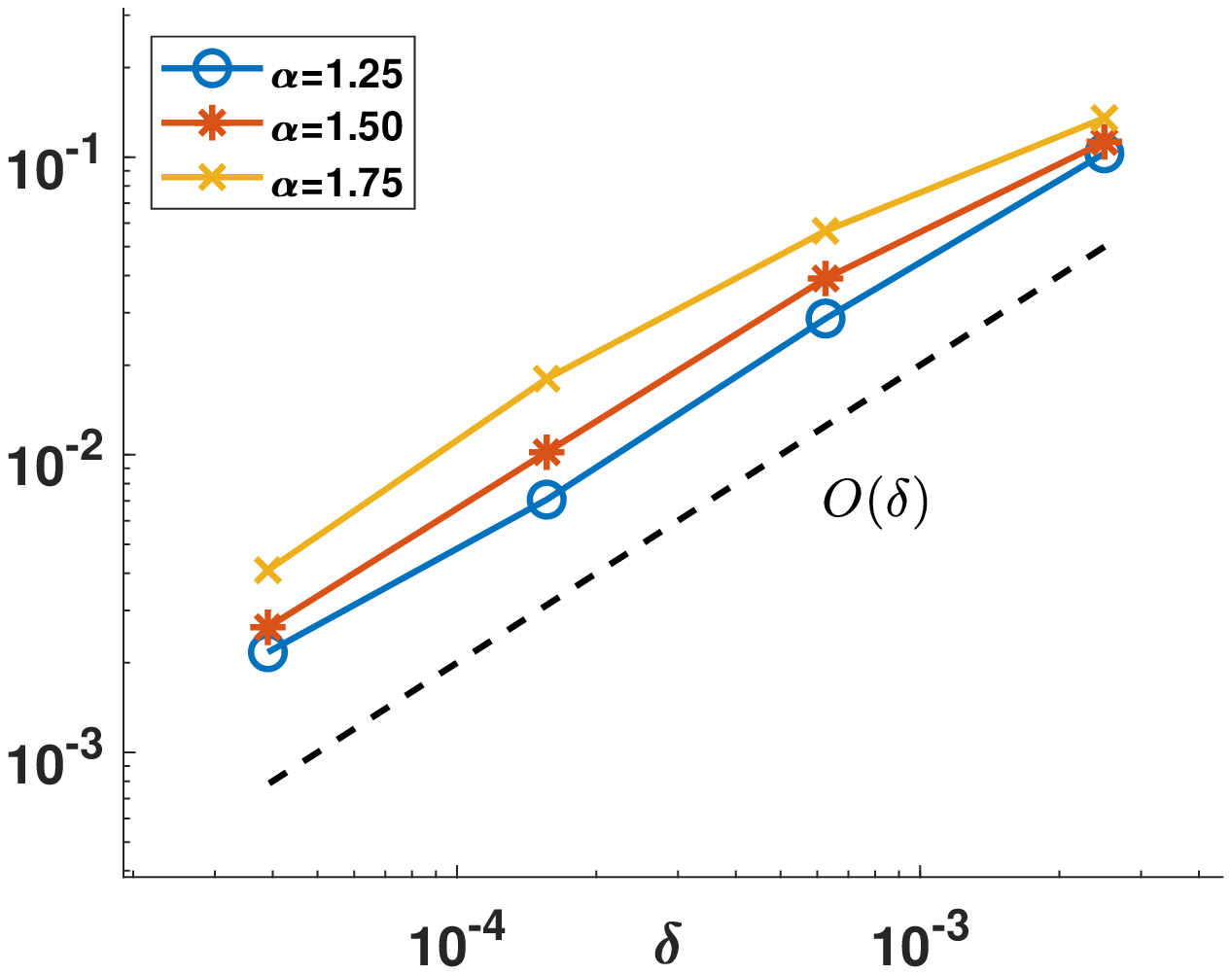}
\caption{$e_s(t)$ with $t=0.5$.}
\end{subfigure}\\%
\begin{subfigure}{.4\textwidth}
\centering
\includegraphics[scale=0.4]{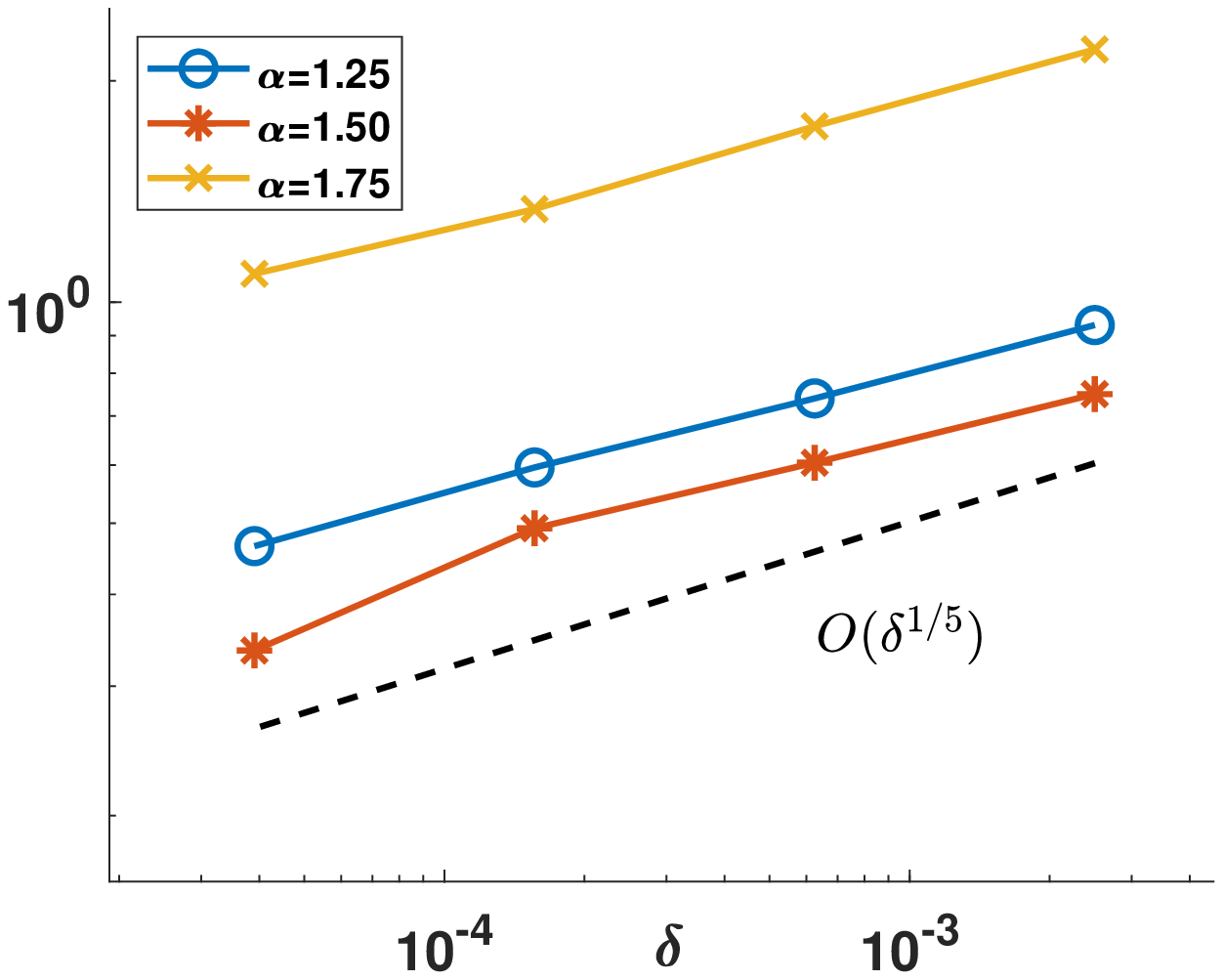}
\caption{$e_{ini,f}$.}
\end{subfigure}%
\begin{subfigure}{.4\textwidth}
\centering
\includegraphics[scale=0.4]{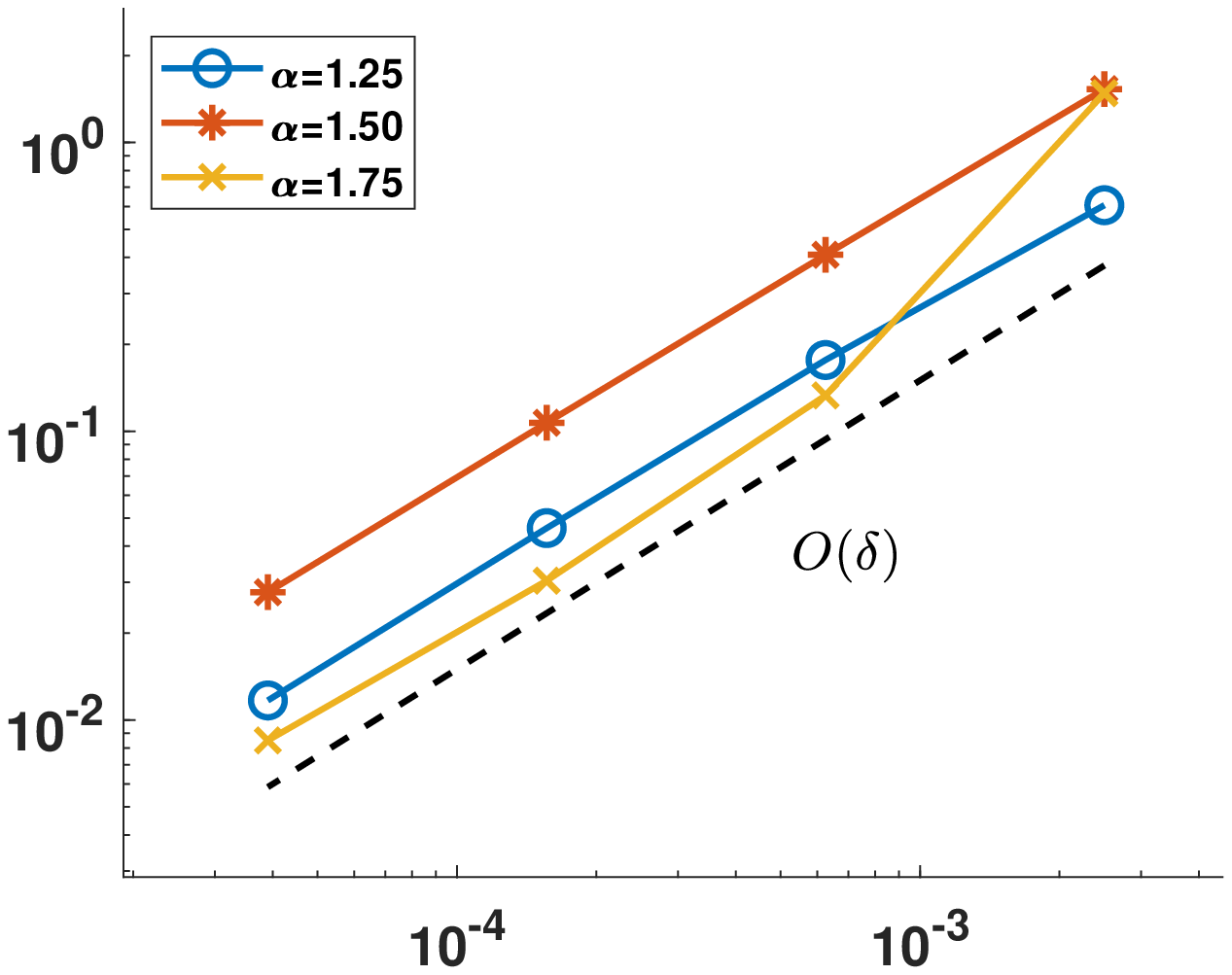}
\caption{$e^n_f$ with $t_n=0.5$.}
\end{subfigure}
\caption{Example (2): semidiscrete errors ((a) and (b)) and fully discrete errors ((c) and (d)).
(a):  $h=\sqrt{\delta}$  and $\gamma = \delta^{4/5}/15,\delta^{4/5}/15,\delta^{4/5}/8$ for $\alpha=1.25,1.5,1.75$ respectively.
(b):  $h = \sqrt{\delta}$ and $\gamma = \delta /10,\delta /5,\delta /5$ for $\alpha=1.25,1.5,1.75$ respectively.
(c):  $h=\sqrt{\delta}$,  $\tau = \delta^{1/5}/20$ and $\gamma = \delta^{4/5}/ 2, \delta^{4/5}/ 15, \delta^{4/5}/ 2 $ for $\alpha =  1.25,1.5,1.75 $ respectively.
(d):  $h=\sqrt{\delta}$,  {$\tau = 10\delta $}, $\gamma = \delta /10,\delta ,\delta /2$ for $\alpha = 1.25,1.5,1.75 $  respectively.
}
\label{fig:nonsmooth}
\end{figure}


\paragraph{\bf Example (2): non-smooth initial data.} Next, we turn to the case of nonsmooth data and expect to examine the influence of
weak  regularity of problem data. Consider
\begin{equation*}
a(x) = \begin{cases}
0, \ 0\le x\le 0.5;\\
1, \ 0.5\le x\le 1.
\end{cases}
,\quad
b(x) = \begin{cases}
1, \ 0\le x\le 0.5;\\
0, \ 0.5\le x\le 1.
\end{cases}
\end{equation*}
and source term $f\equiv0$.
It is well-known that $a,b\in \dH{\frac 12-\varepsilon}$ for any $\varepsilon\in(0,\frac12]$.
According to Theorem \ref{thm:err-semi}, the error of the semidiscrete discrete solution satisfies
\begin{equation*}
\begin{aligned}
&\| \tu_h^\delta -a \|_\L2Om  + \|\partial_t  \tu_h^\delta- b \|_\L2Om \le c(\gamma^{\frac q2} +(h^2+\delta)\gamma^{-1}),\\
&\|(\tu_h^\delta-u)(t)\|_\L2Om  \le c(\gamma+h^2+\delta),\quad \text{for a given}~~ t>0.
\end{aligned}
\end{equation*}
Therefore, for given $\delta$, to numerically reconstruct the initial data $a$ and $b$, we let $h=\sqrt{\delta}$,
and $\gamma \sim \delta^{4/5}$ and expect that the error converges to zero as $O(\delta^{\frac15})$, while
to approximate $u(t)$ for some $t>0$, we let $h \sim \sqrt{\delta}$ and
 $\gamma \sim \delta$ and expect a convergence of order $O(\delta)$.
The theoretical results agrees well with the numerical
results in Figure \ref{fig:nonsmooth} (a) and (b).

In Figure \ref{fig:nonsmooth} (c) and (d) we plot errors of the numerical reconstruction by fully discrete scheme \eqref{eqn:fully-back}.
According to Theorem \ref{thm:err-fully} we have the error estimate that (with $q=\frac12-\varepsilon $)
\begin{equation*}
\begin{aligned}
&\|a_{h,\tau}^\delta-a\|_\L2Om  + \|b_{h,\tau}^\delta-b\|_\L2Om \le c(\gamma^{\frac q2} +\tau +(h^2+\delta)\gamma^{-1}),\\
&\|\t U_n^\delta -u(t_n)\|_\L2Om \le c(\gamma + \tau +h^2 + \delta),\quad \text{for any fixed}~~t_n > 0.
\end{aligned}
\end{equation*}
Therefore we choose parameters  $h \sim \sqrt{\delta}$,  $\tau \sim \delta^{1/5} $ and $\gamma \sim \delta^{4/5}$ for the numerical reconstruction of initial data,
while we let $h\sim \sqrt{\delta}$, {$\tau \sim \delta$ and $\gamma \sim \delta$} for approximately solving the solution $u(t_n)$ for some $t_n>0$.
The empirical convergence results show that $e_{\text{ini},f} \sim \delta^{\frac15}$ and $e_f^n \sim \delta$,
which are consistent with our theoretical findings.
Finally, in figure \ref{fig:sol:nonsmooth}, we provide the profiles
of solutions to semidiscrete and fully discrete schemes with different noise levels, which show clearly the convergence of the discrete
approximation as the noise level $\delta$ decreases.

\begin{figure}[htbp]
\centering
\begin{subfigure}{.33\textwidth}
\centering
\includegraphics[scale=0.3]{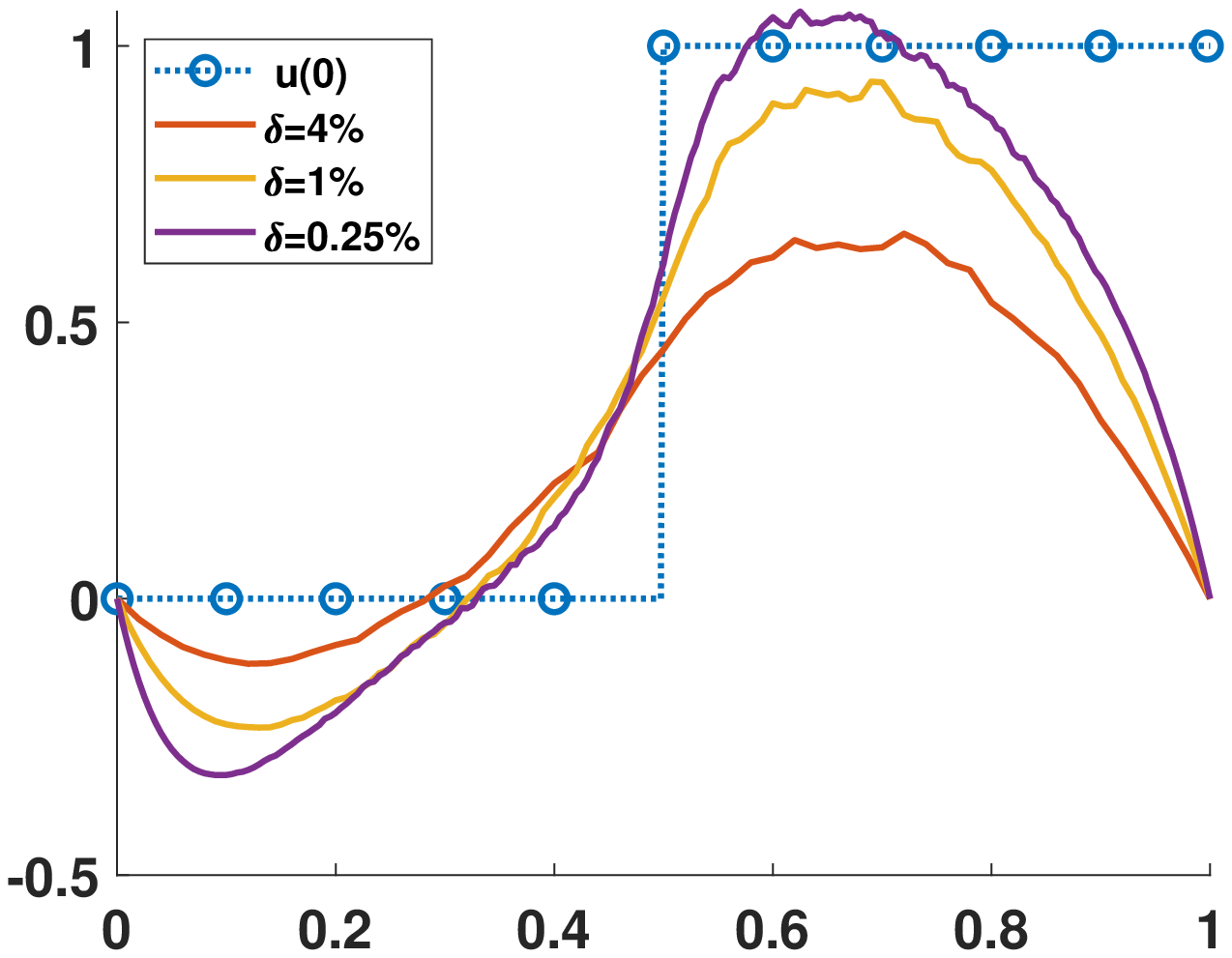}
\caption{$\tu^\delta_h(0)$}
\end{subfigure}%
\begin{subfigure}{.33\textwidth}
\centering
\includegraphics[scale=0.3]{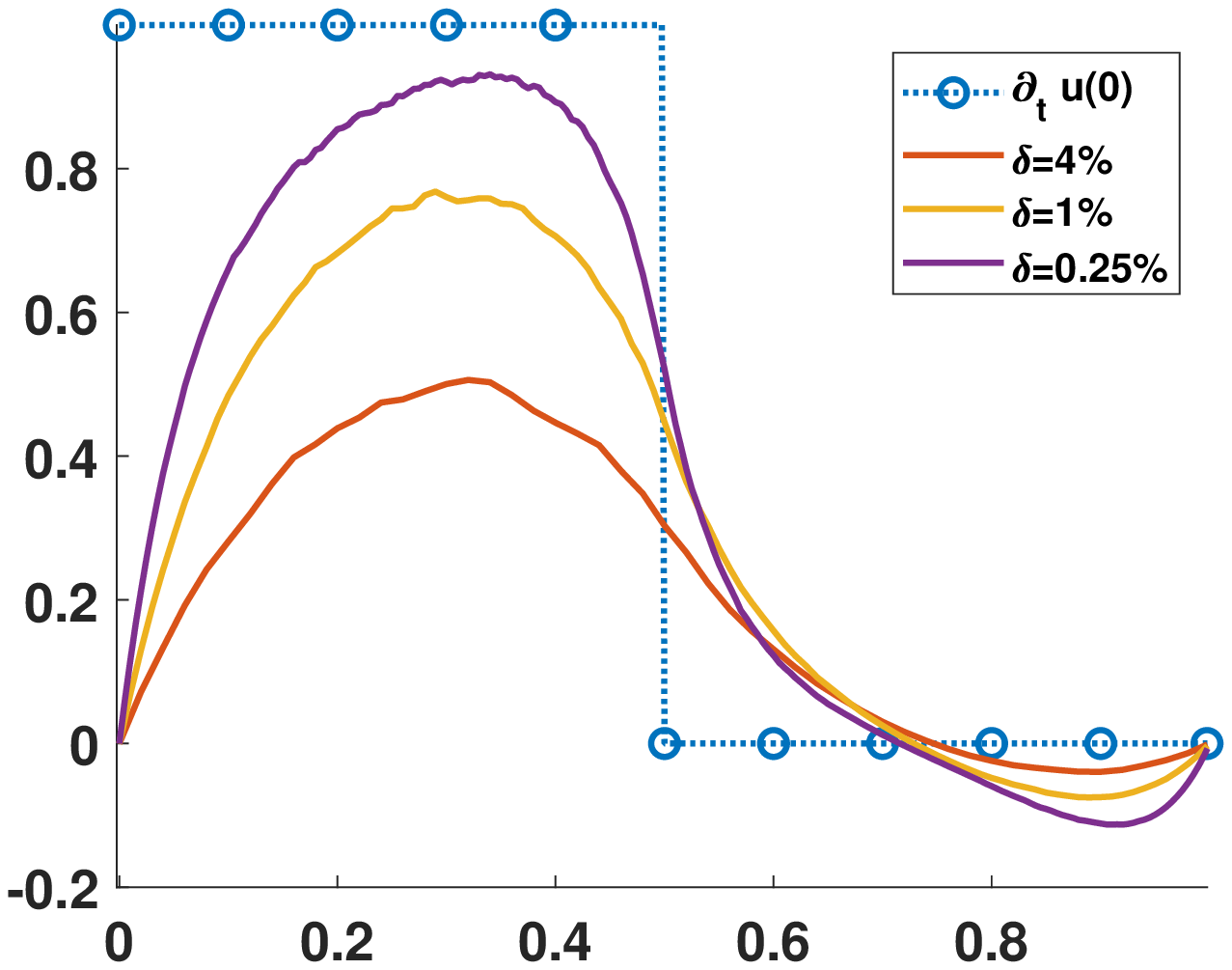}
\caption{$\partial_t \tu^\delta_h(0)$}
\end{subfigure}%
\begin{subfigure}{.33\textwidth}
\centering
\includegraphics[scale=0.3]{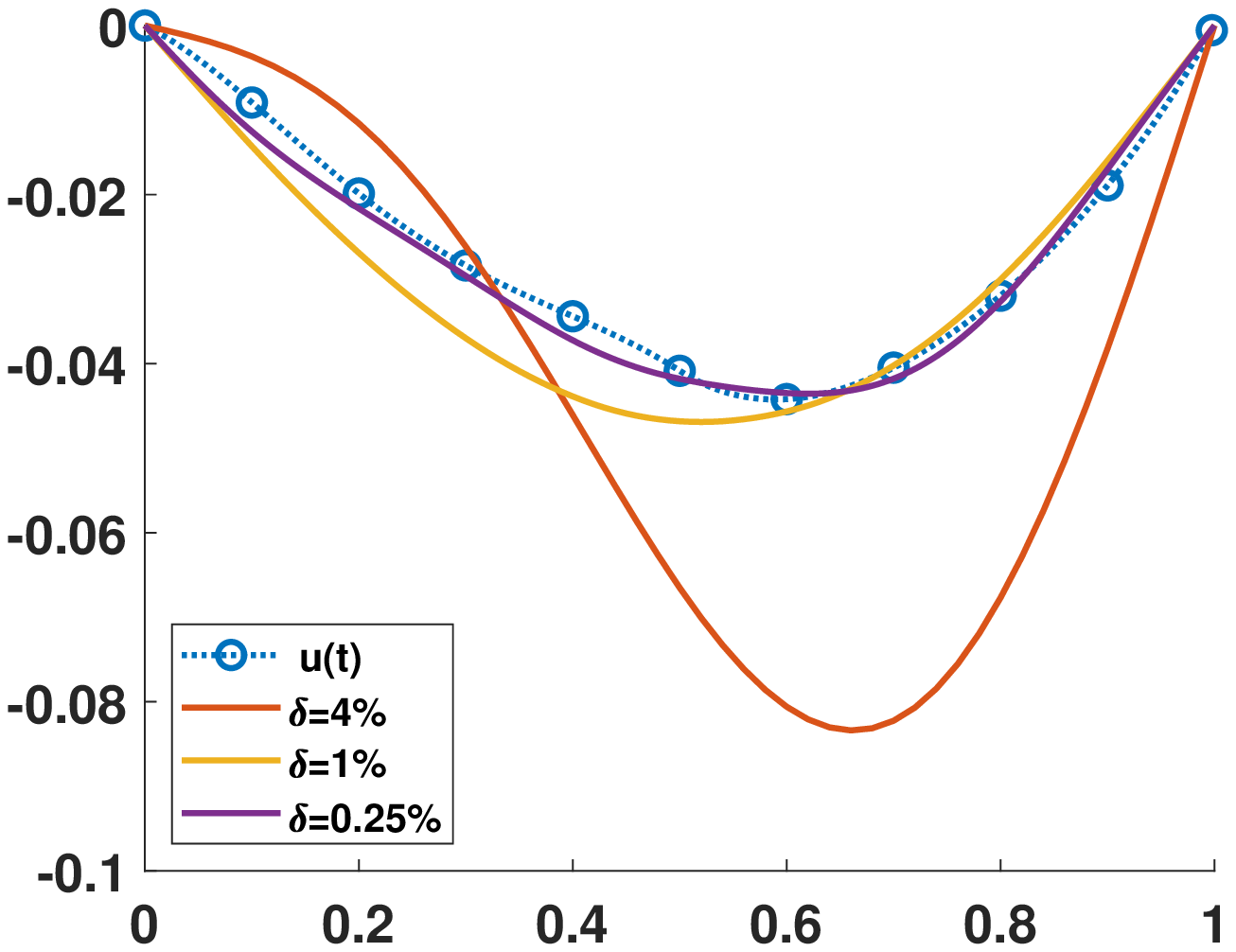}
\caption{$\tu^\delta_h(t)$ with $t=0.5$}
\end{subfigure}%
\newline
\begin{subfigure}{.33\textwidth}
\centering
\includegraphics[scale=0.3]{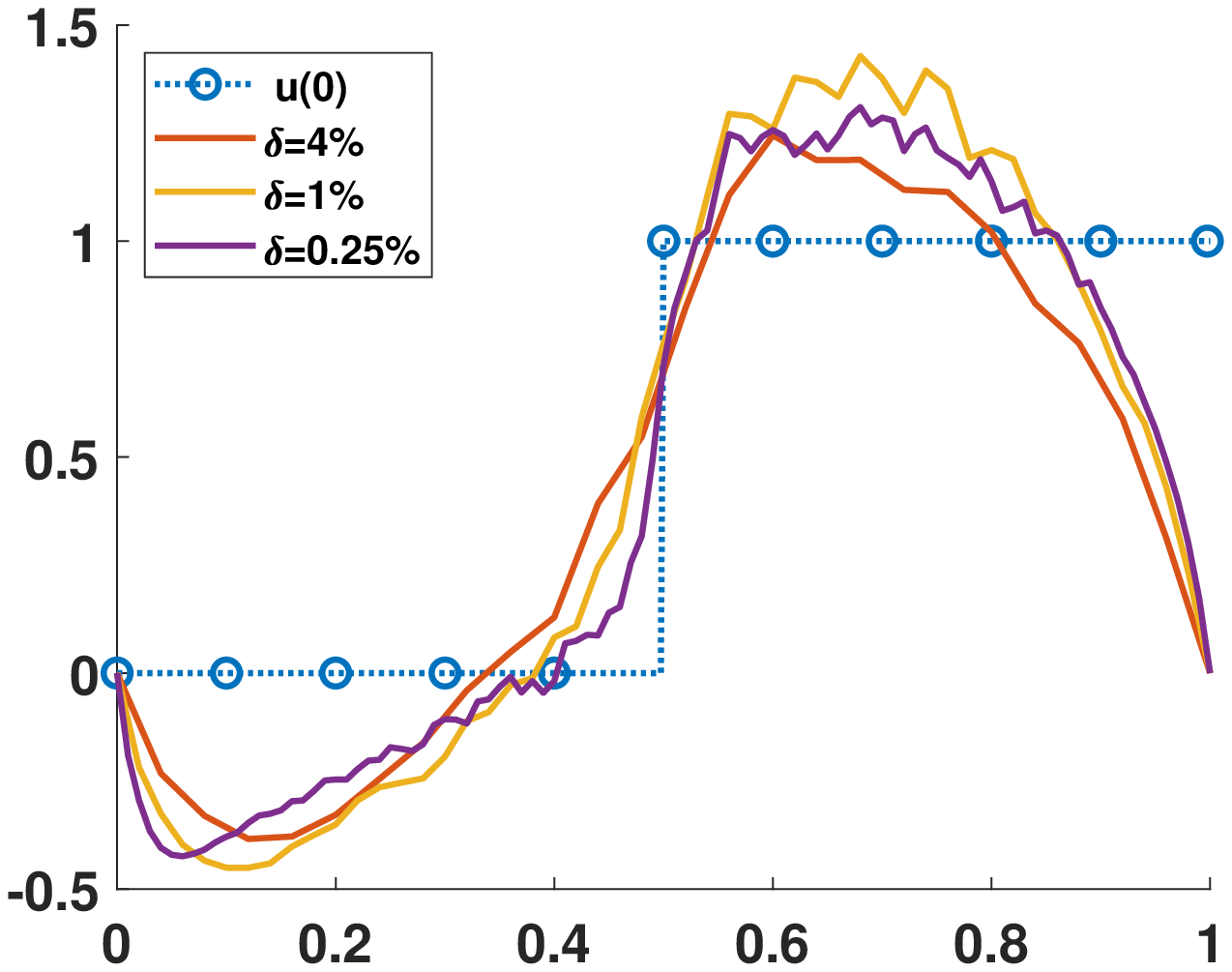}
\caption{$\tilde a_{h,\tau}^\delta$}
\end{subfigure}%
\begin{subfigure}{.33\textwidth}
\centering
\includegraphics[scale=0.3]{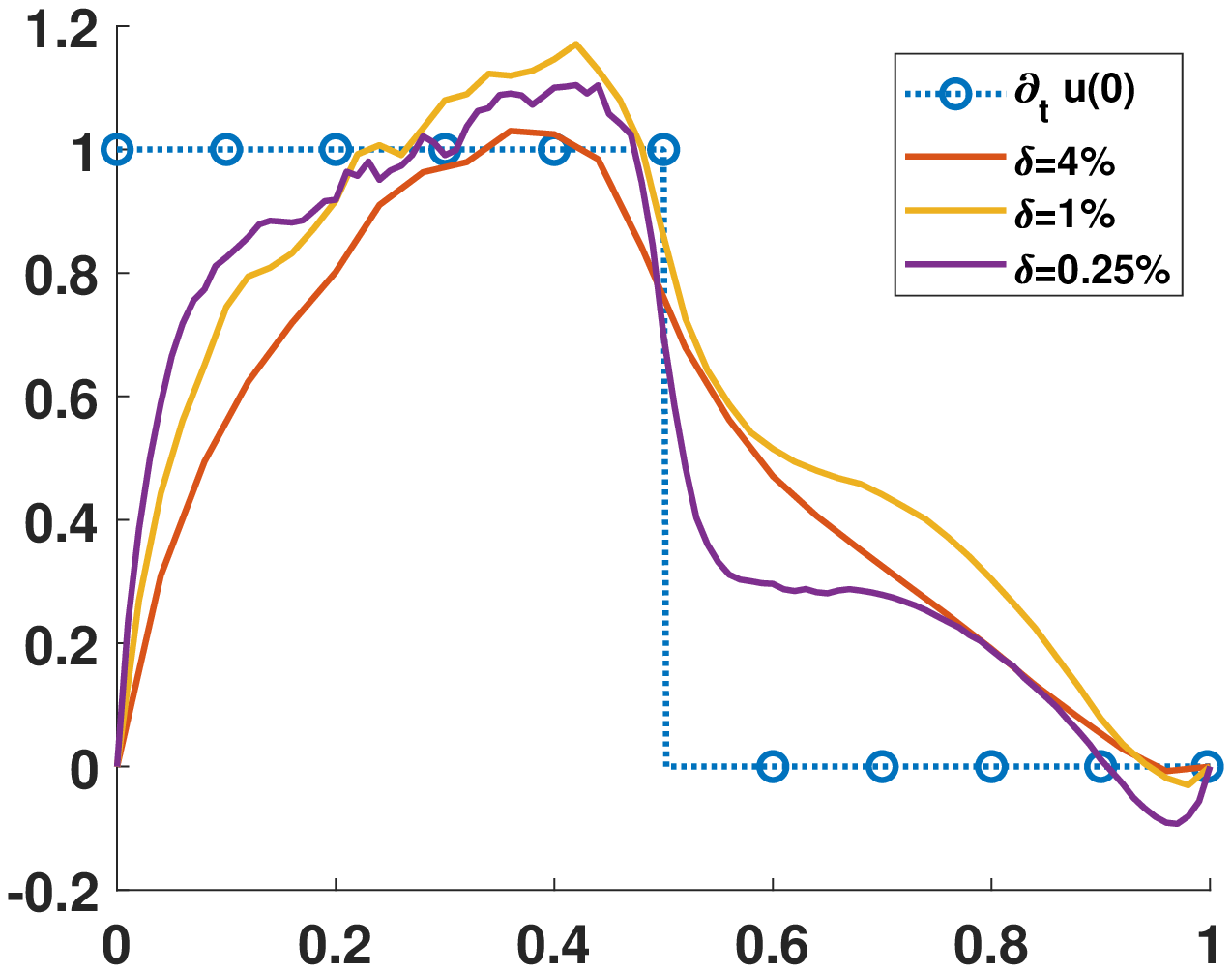}
\caption{$\tilde b_{h,\tau}^\delta$}
\end{subfigure}%
\begin{subfigure}{.33\textwidth}
\centering
\includegraphics[scale=0.3]{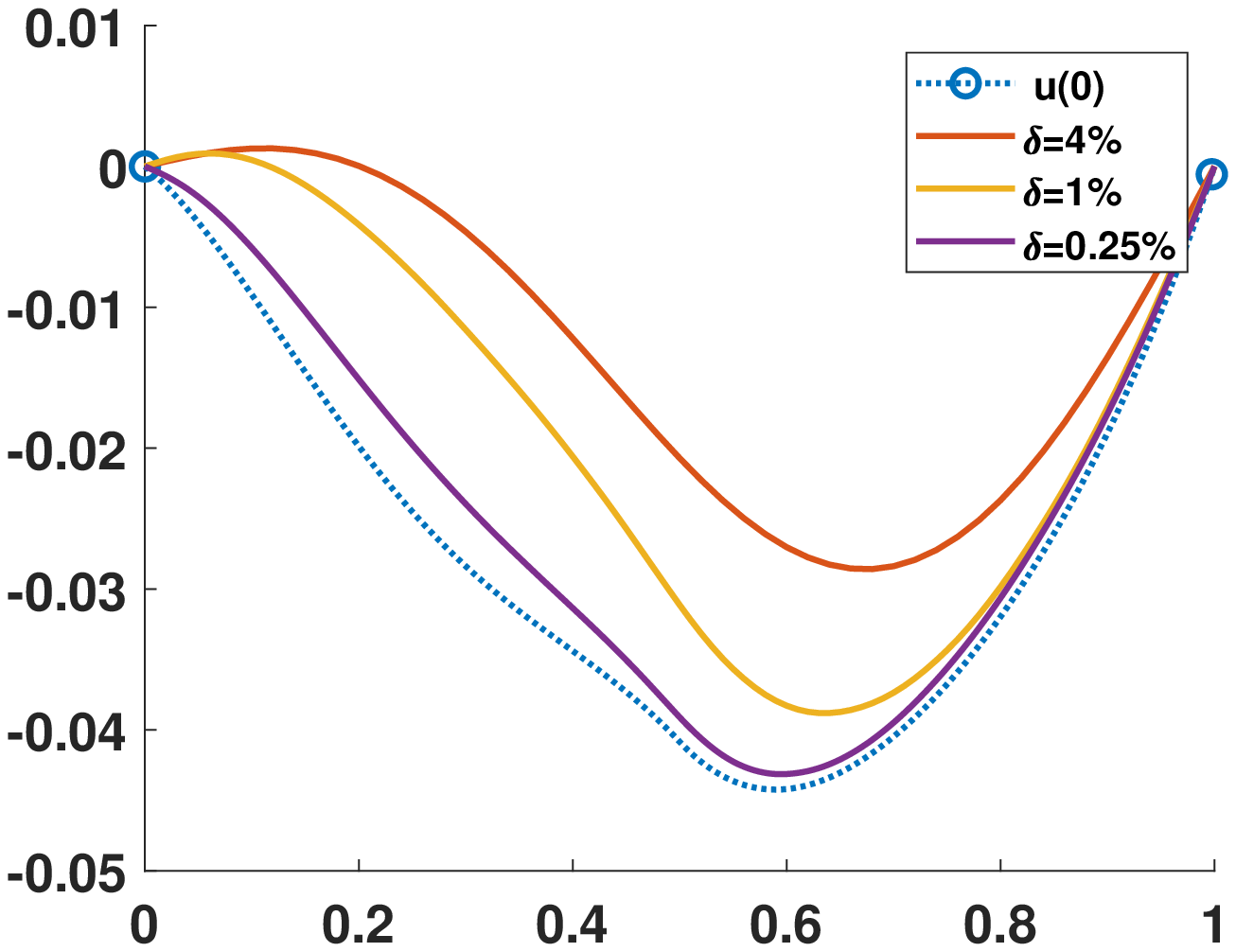}
\caption{$\tilde U_{n}^\delta$, with $t_n=0.5$}
\end{subfigure}%
\caption{Example(b): profiles of numerical solutions with $\alpha=1.5$ for $\delta = 4\%,1\%,0.25\%$.
First row: $h=\sqrt{\delta}/10$, $\gamma = \delta^{4/5}/5$ for both (a) and (b); $h=\sqrt{\delta}/10$, $\gamma = \delta/5$ for (c). Second row:
$h=\sqrt{\delta}/10$, $\tau = \delta^{1/5}/10$, $\gamma =\delta^{4/5}/15$ for both (d) and (e); $h=\sqrt{\delta}/10$,
$\tau = \delta$, $\gamma = \delta/10$ for (f).  }
\label{fig:sol:nonsmooth}
\end{figure}

 \paragraph{\bf Example (c): 2D  examples.} Finally, we test a two dimensional diffusion-wave models in $\Omega = (0,1)^2$
 with smooth initial conditions:
$$a(x,y)=\sin(2\pi x)\sin(2\pi y),\quad b(x,y) = 4x(1-x)y(1-y) \in \dH2 = H^2\II\cap H_0^1\II,$$
and source term $f\equiv0$.
The reference solution is computed with $h=1/150$, $\tau=1/1000$.
Noting that the fully discrete system is not symmetric,
we apply the biconjugate gradient stabilized method \cite{Vorst:1992}.

In Figure \ref{fig:2D:a} and \ref{fig:2D:b},
we plot profiles of (numerical) reconstruction of initial data $a$, $b$ and approximation errors,
with different noise level $\delta$ as well as different parameters $\gamma,h,\tau$ chosen according to $\delta$.
The empirical observations are in excellent agreement with
theoretical results, e.g., convergence as the noise level $\delta$ decreases
to zero.

\begin{figure}[htbp]
\begin{subfigure}{.24\textwidth}
\centering
\includegraphics[scale=0.25]{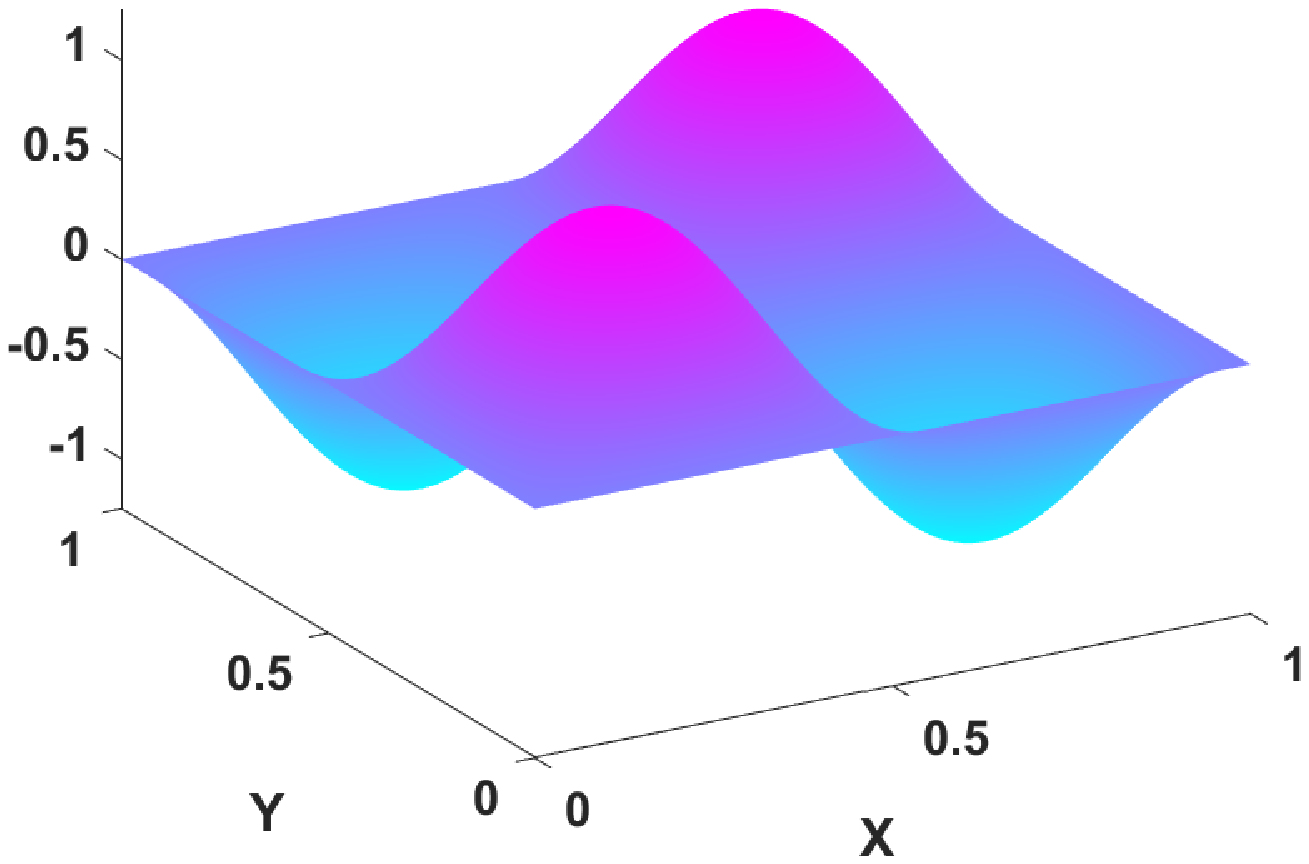}
\end{subfigure}%
\begin{subfigure}{.24\textwidth}
\centering
\includegraphics[scale=0.25]{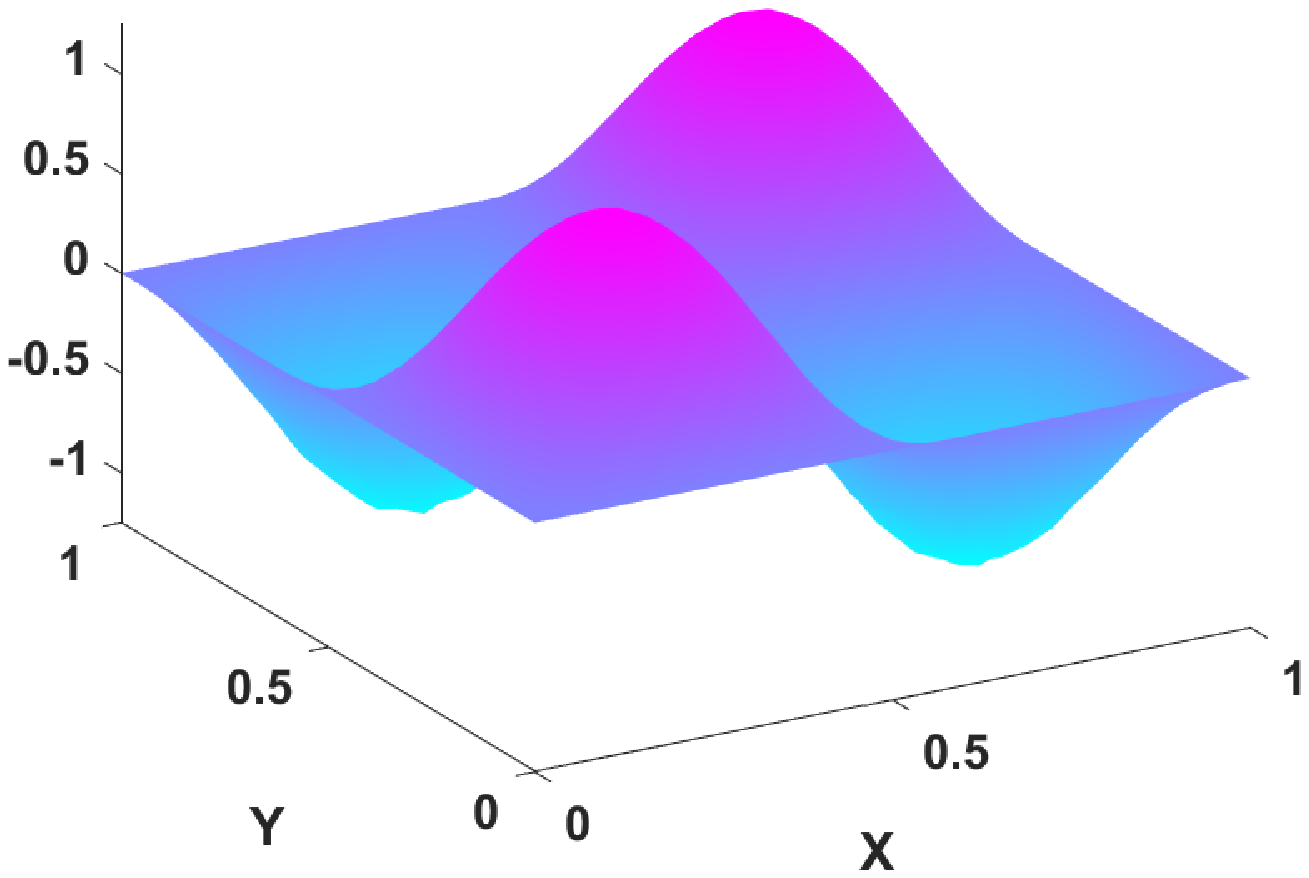}
\end{subfigure}%
\begin{subfigure}{.24\textwidth}
\centering
\includegraphics[scale=0.25]{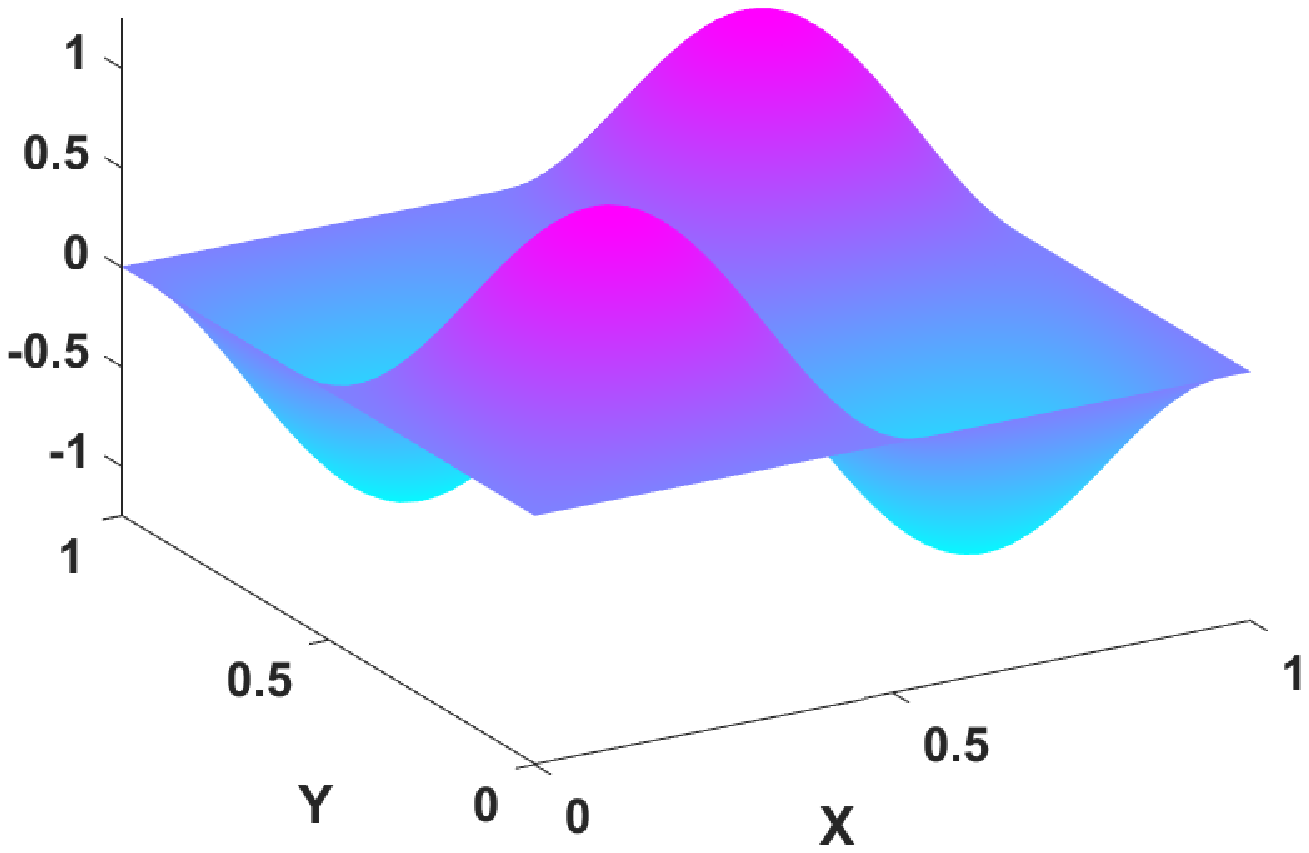}
\end{subfigure}
\begin{subfigure}{.24\textwidth}
\centering
\includegraphics[scale=0.25]{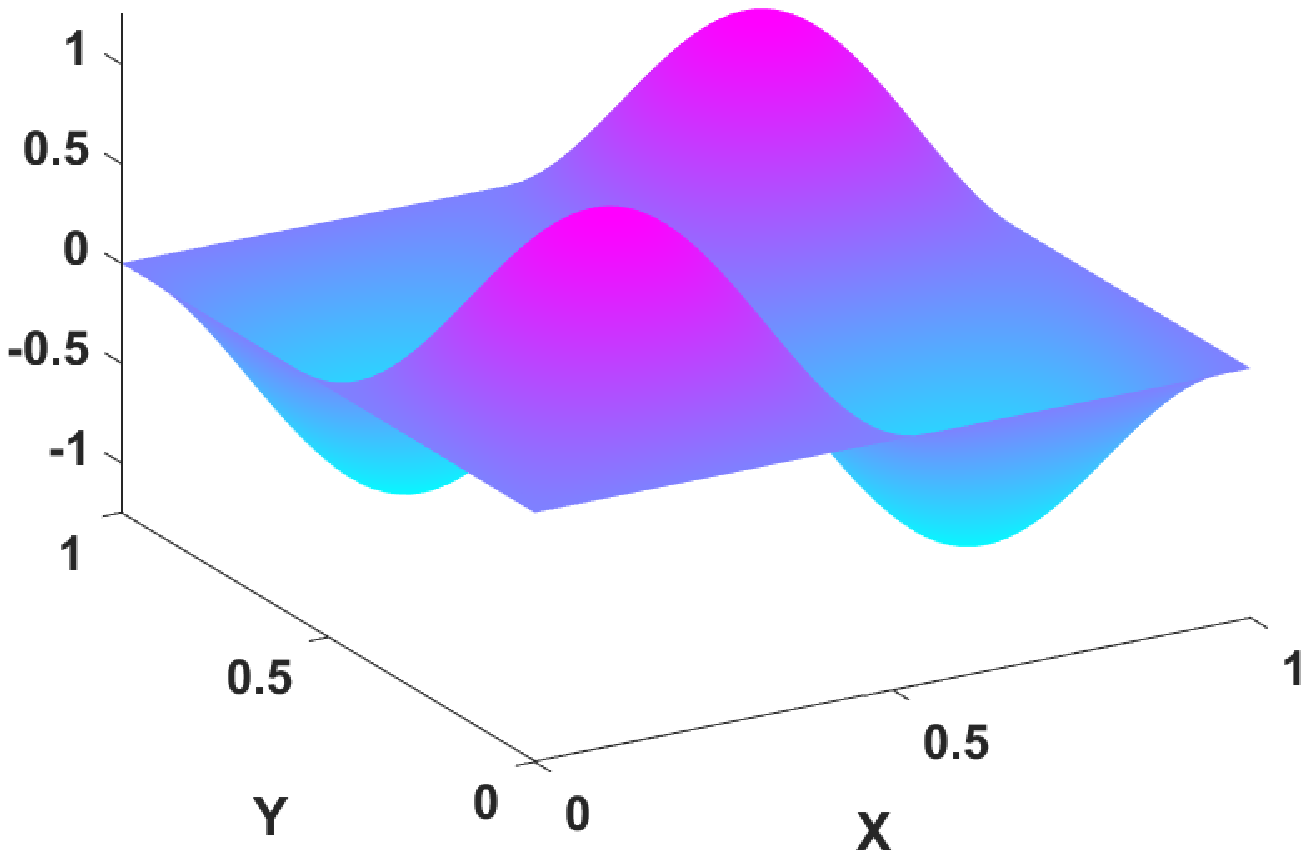}
x
\end{subfigure}
\newline
\raggedleft
\begin{subfigure}{.24\textwidth}
\centering
\includegraphics[scale=0.25]{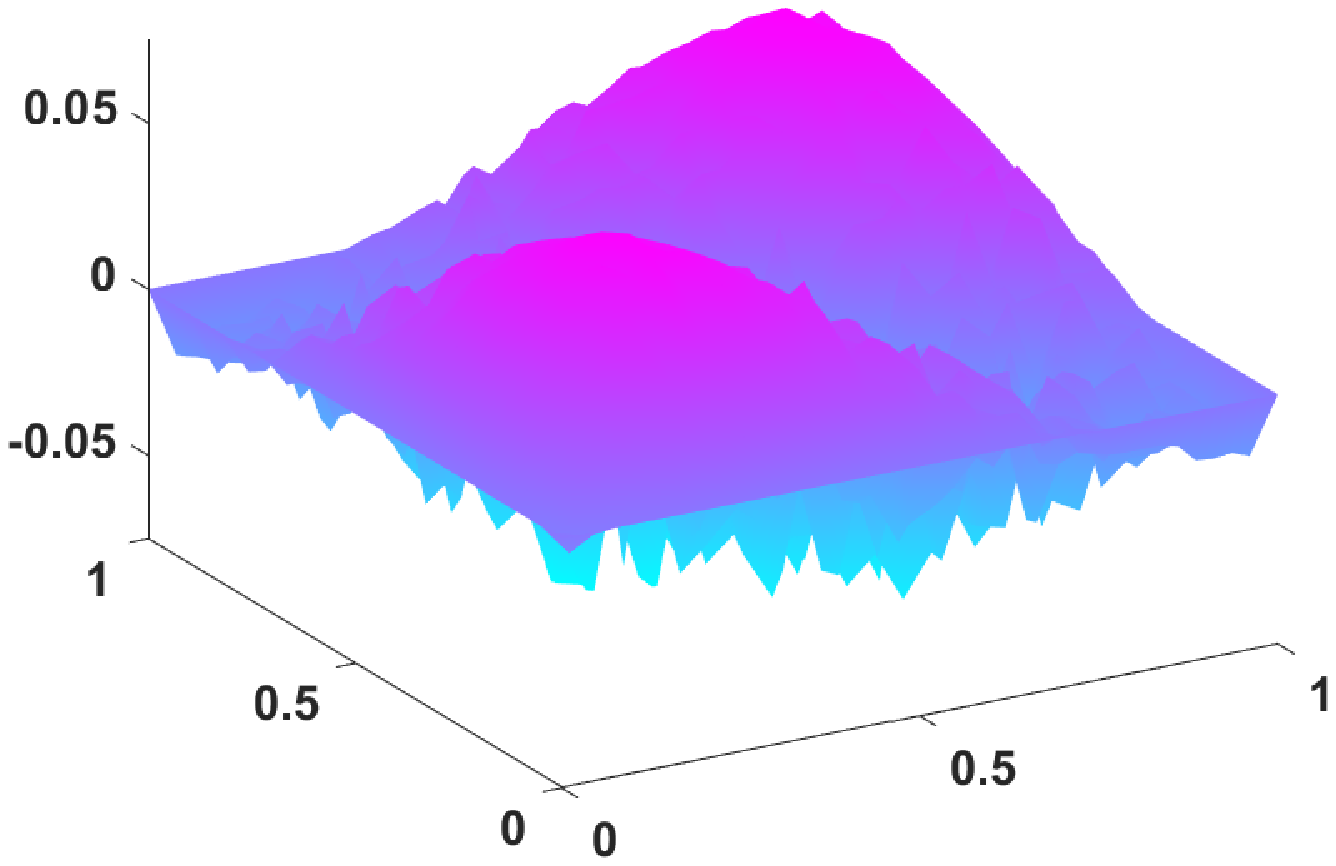}
\caption{$\delta =1e-2$.}
\end{subfigure}%
\begin{subfigure}{.24\textwidth}
\centering
\includegraphics[scale=0.25]{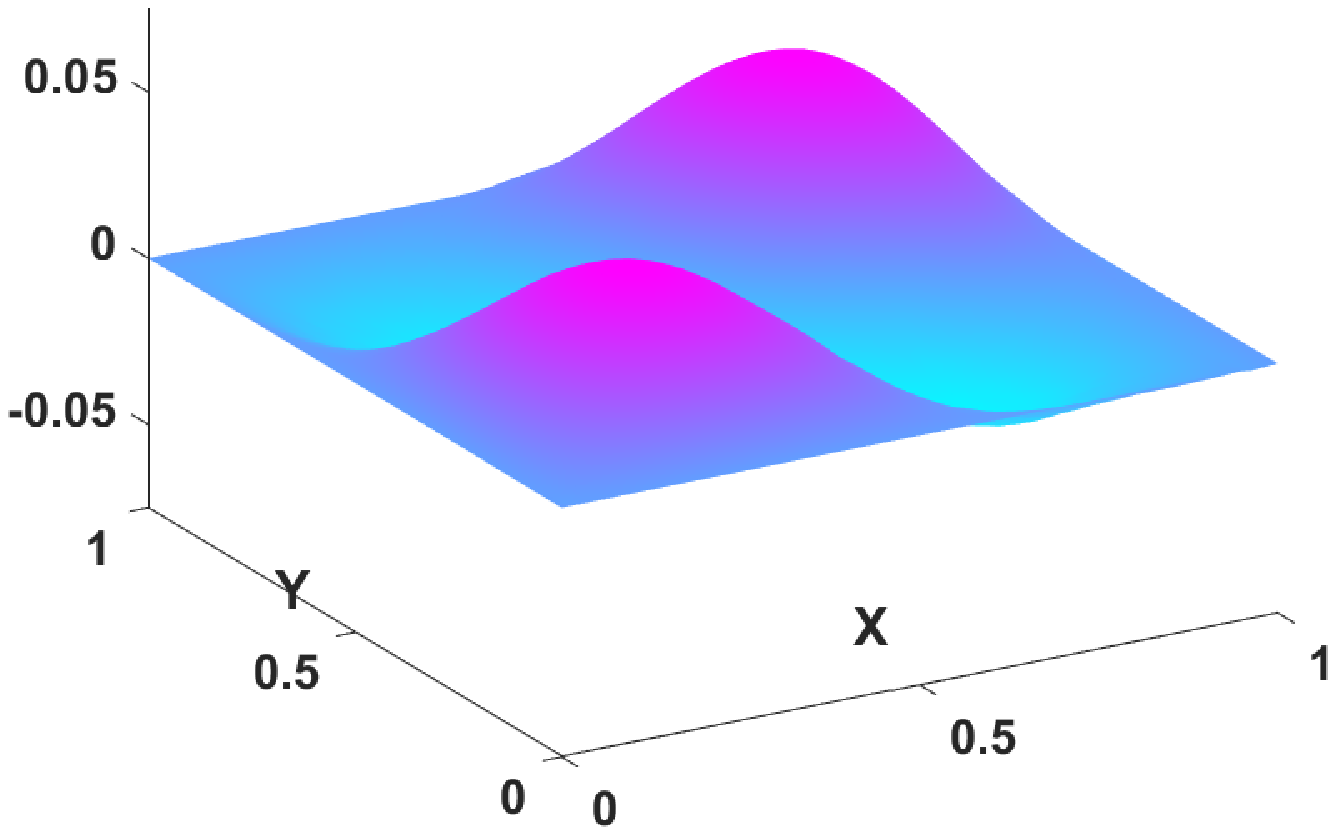}
\caption{$\delta = 5e-3$.}
\end{subfigure}
\begin{subfigure}{.24\textwidth}
\centering
\includegraphics[scale=0.25]{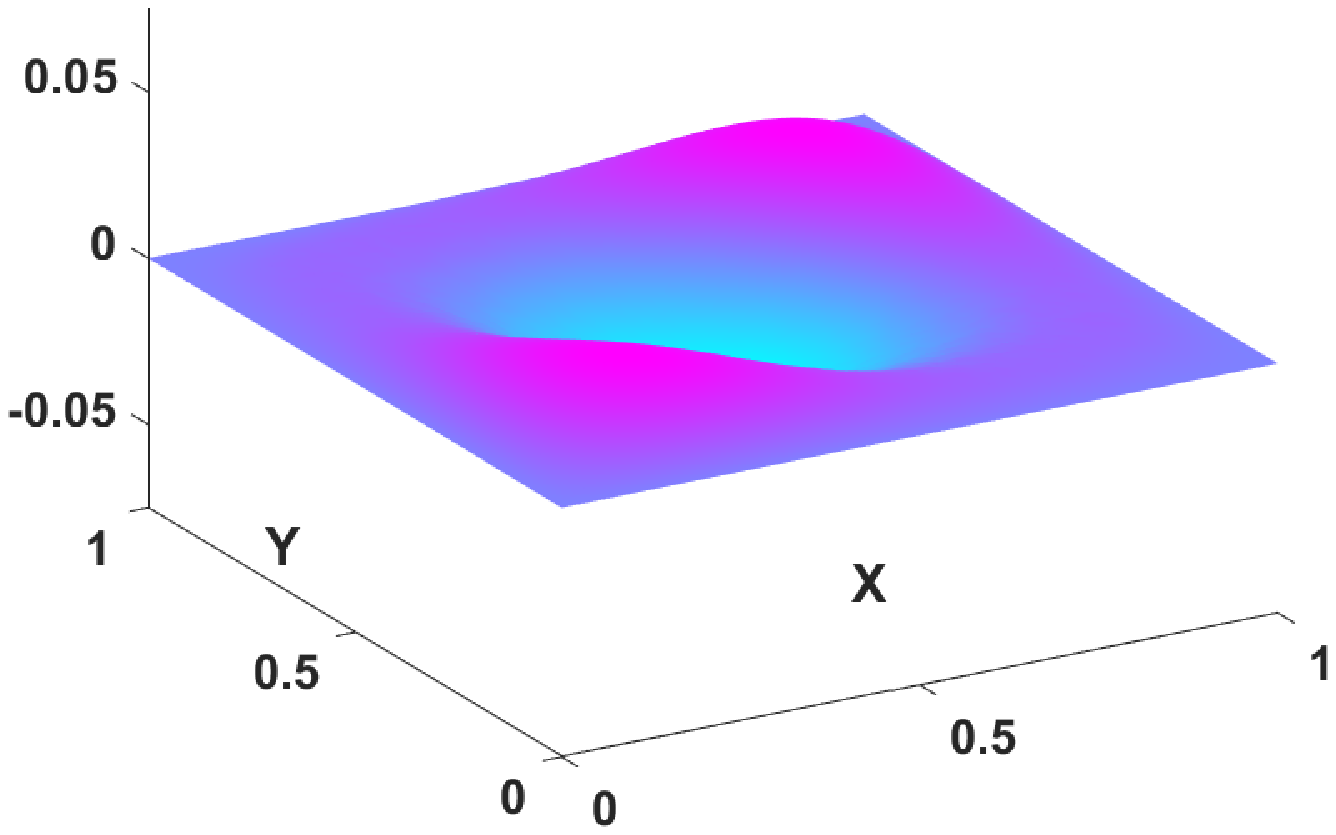}
\caption{$\delta= 2.5e-3$.}
\end{subfigure}
\caption{Example(c): Top left: Exact initial data $a$. The remain three columns are profiles of numerical reconstructions $a_{h,\tau}^\delta$ and theirs errors, with $h=\sqrt{\delta}/4$, $\tau = \sqrt{\delta}/20$, $\gamma=\sqrt{\delta}/4000$.}
\label{fig:2D:a}
\end{figure}
\begin{figure}[htbp]
\begin{subfigure}{.24\textwidth}
\centering
\includegraphics[scale=0.25]{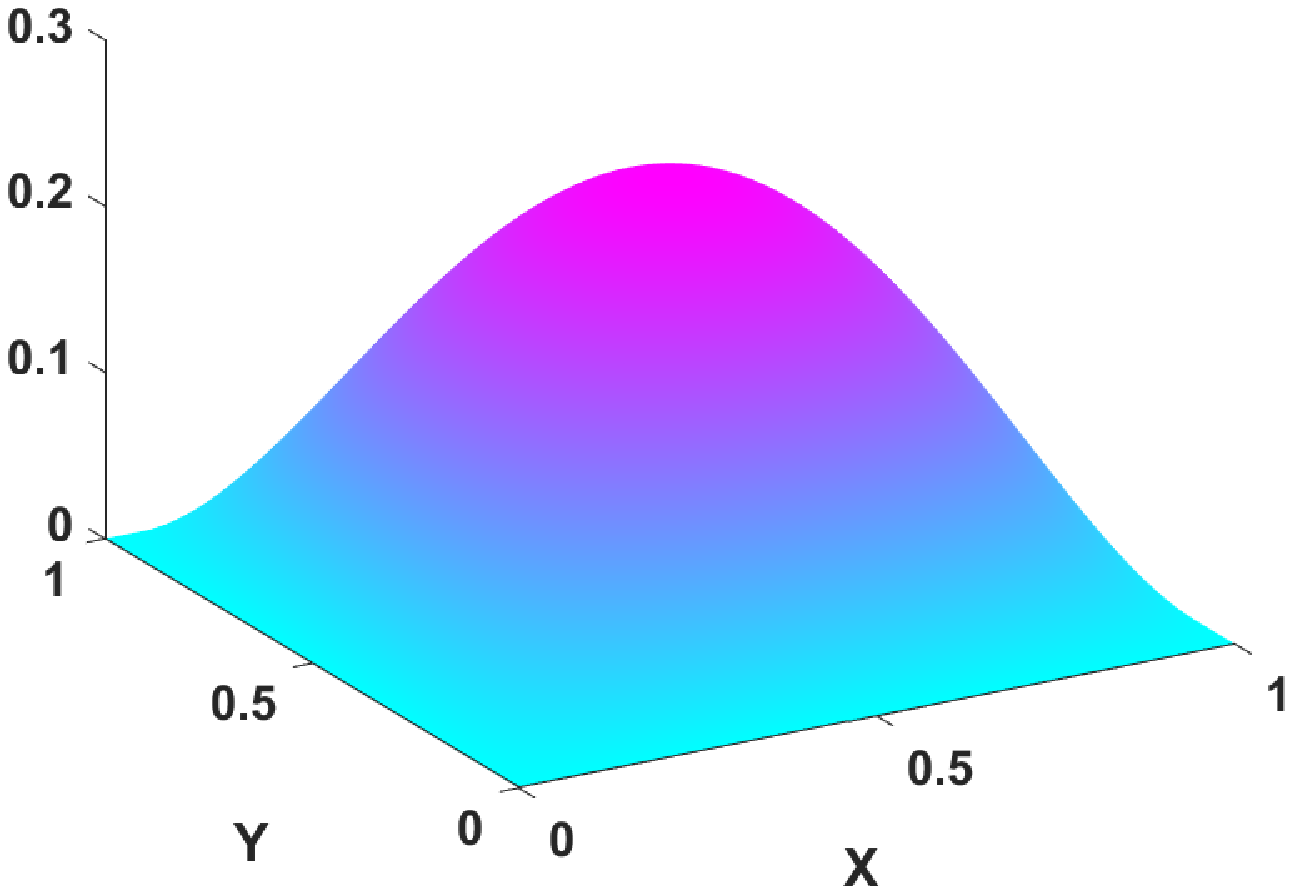}
\end{subfigure}%
\begin{subfigure}{.24\textwidth}
\centering
\includegraphics[scale=0.25]{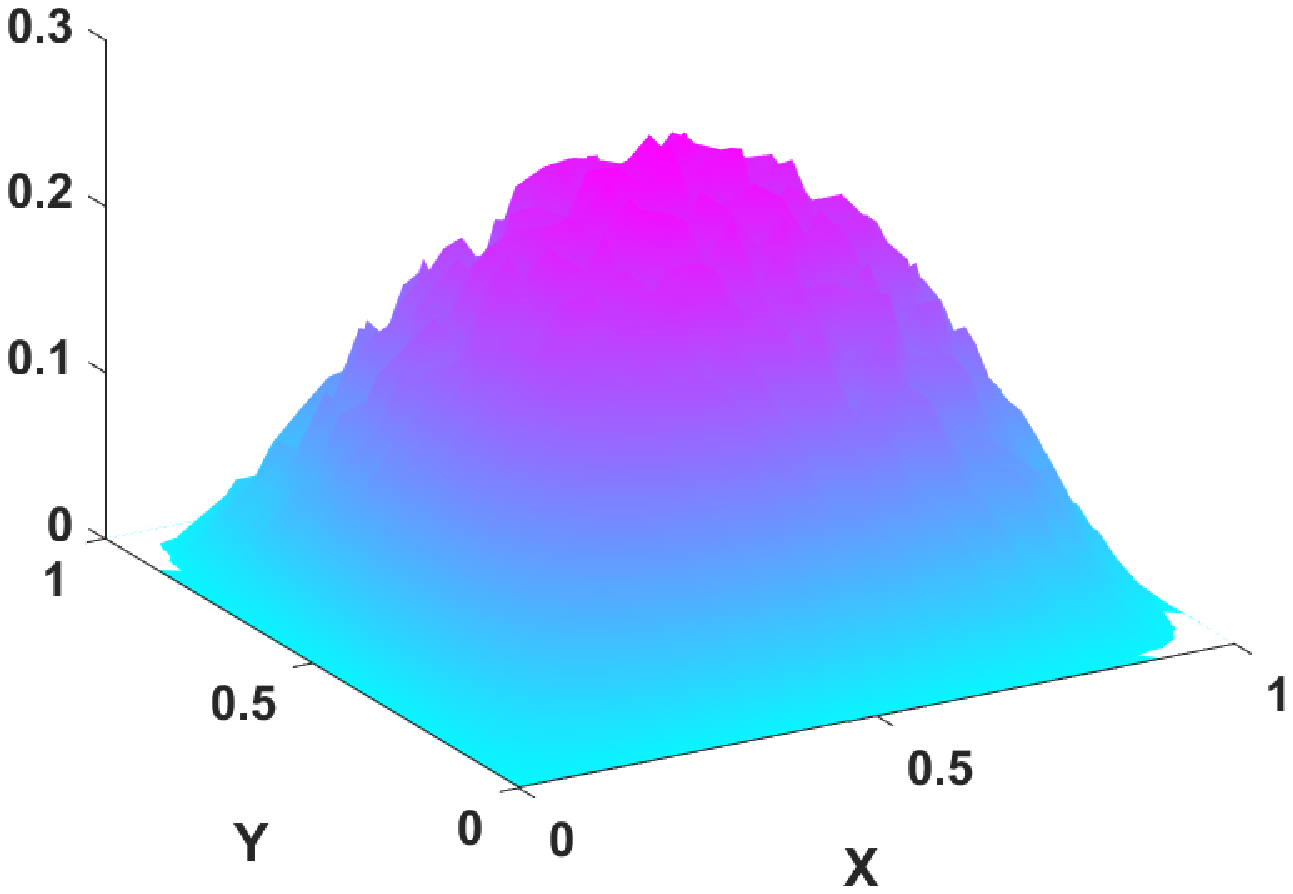}
\end{subfigure}%
\begin{subfigure}{.24\textwidth}
\centering
\includegraphics[scale=0.25]{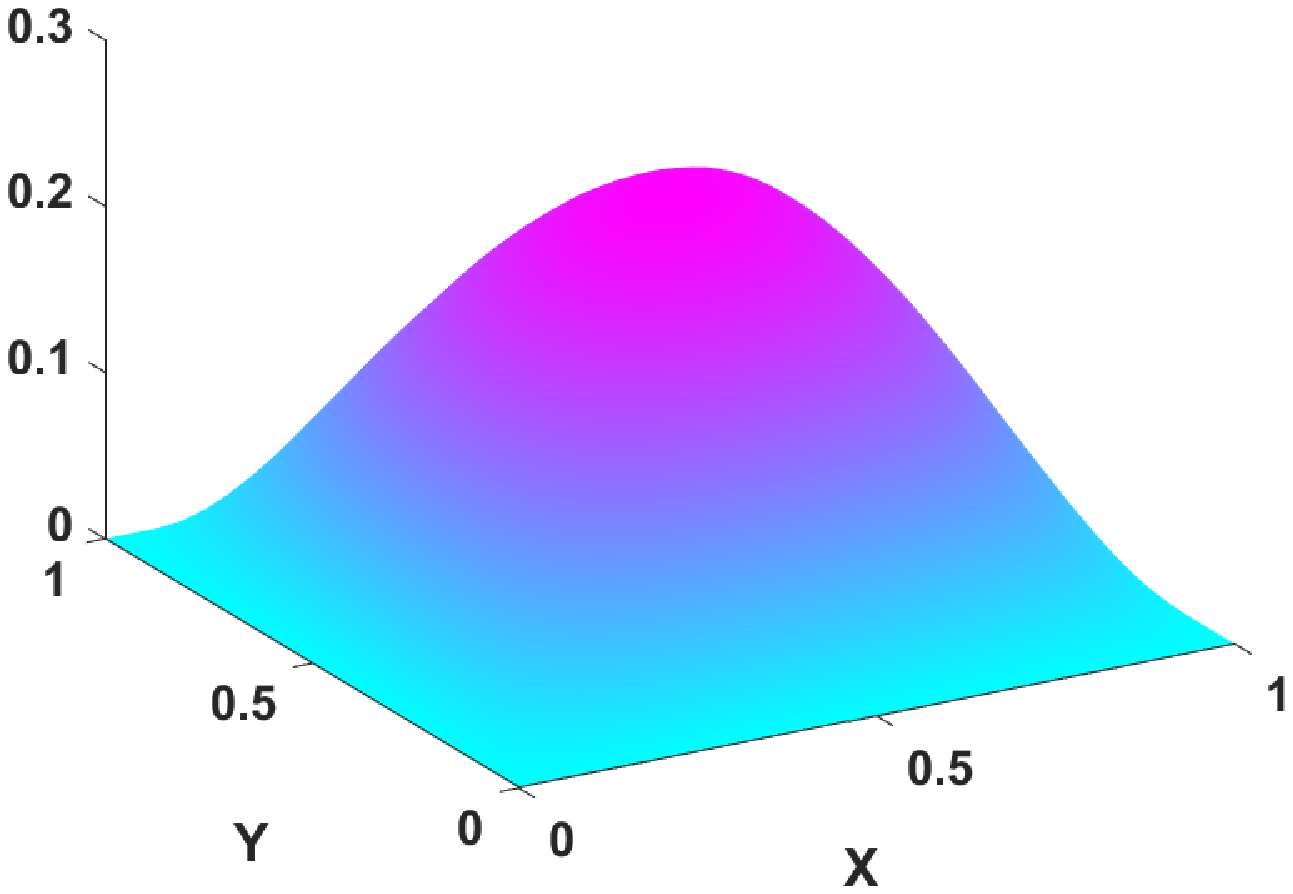}
\end{subfigure}
\begin{subfigure}{.24\textwidth}
\centering
\includegraphics[scale=0.25]{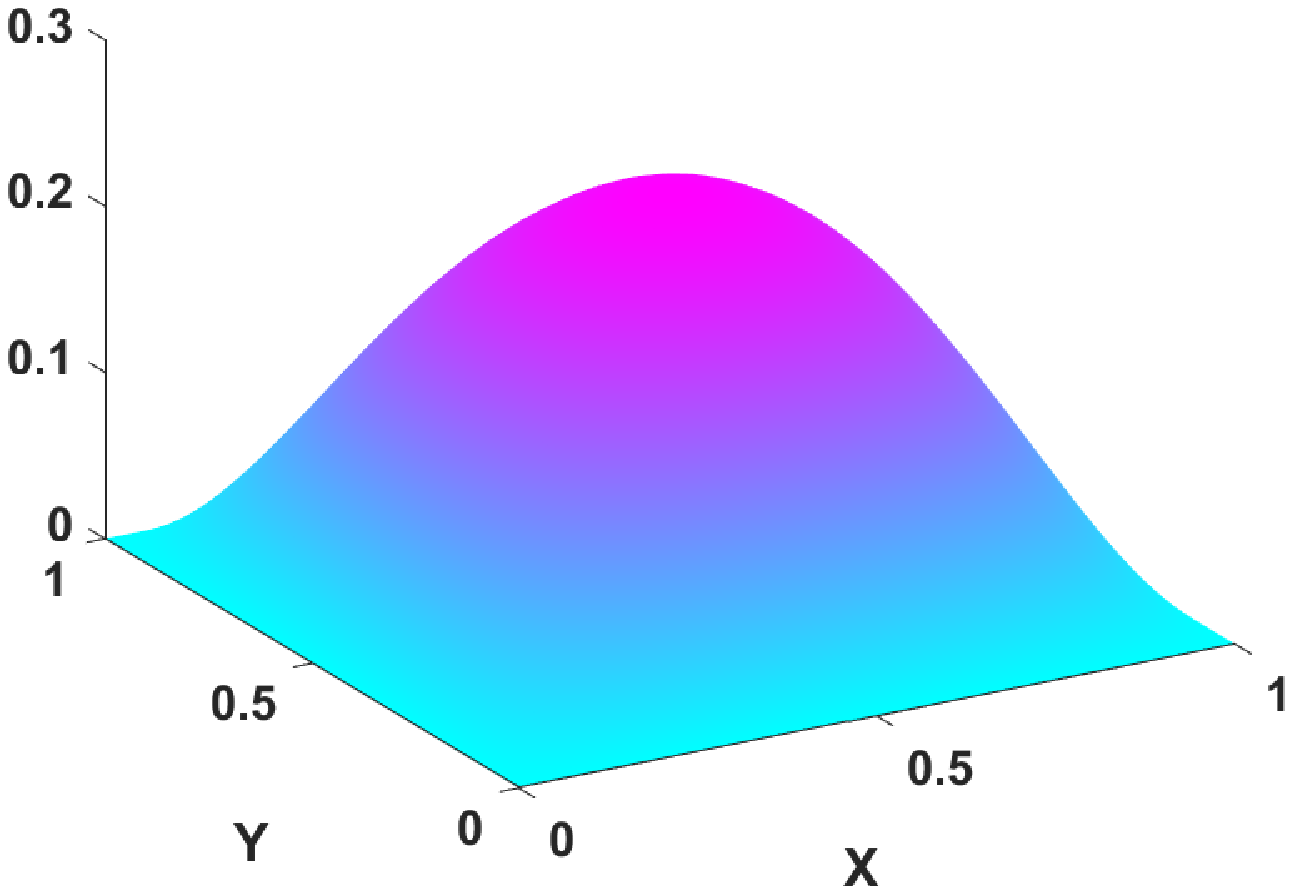}
x
\end{subfigure}
\newline
\raggedleft
\begin{subfigure}{.24\textwidth}
\centering
\includegraphics[scale=0.25]{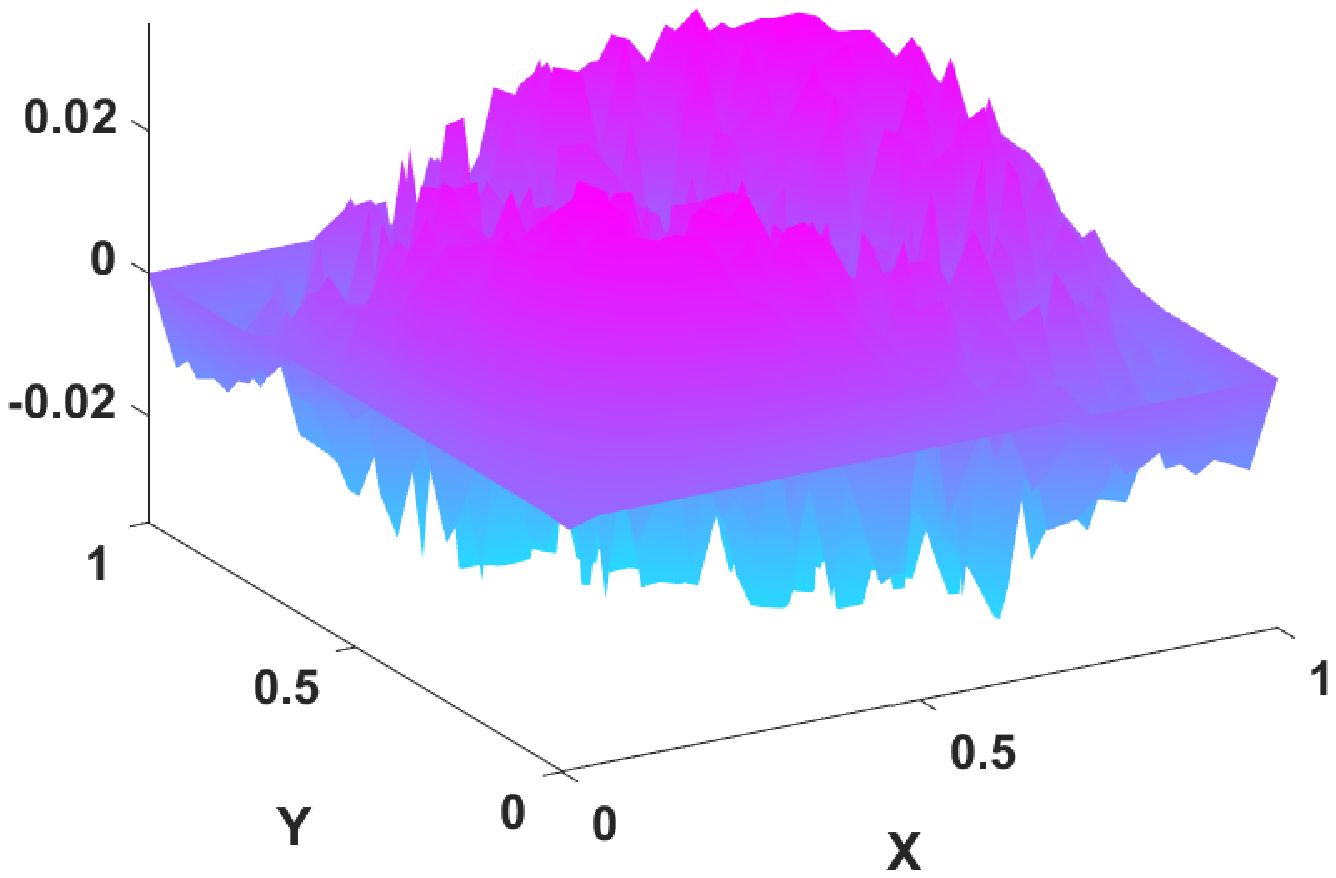}
\caption{$\delta =1e-2$.}
\end{subfigure}%
\begin{subfigure}{.24\textwidth}
\centering
\includegraphics[scale=0.25]{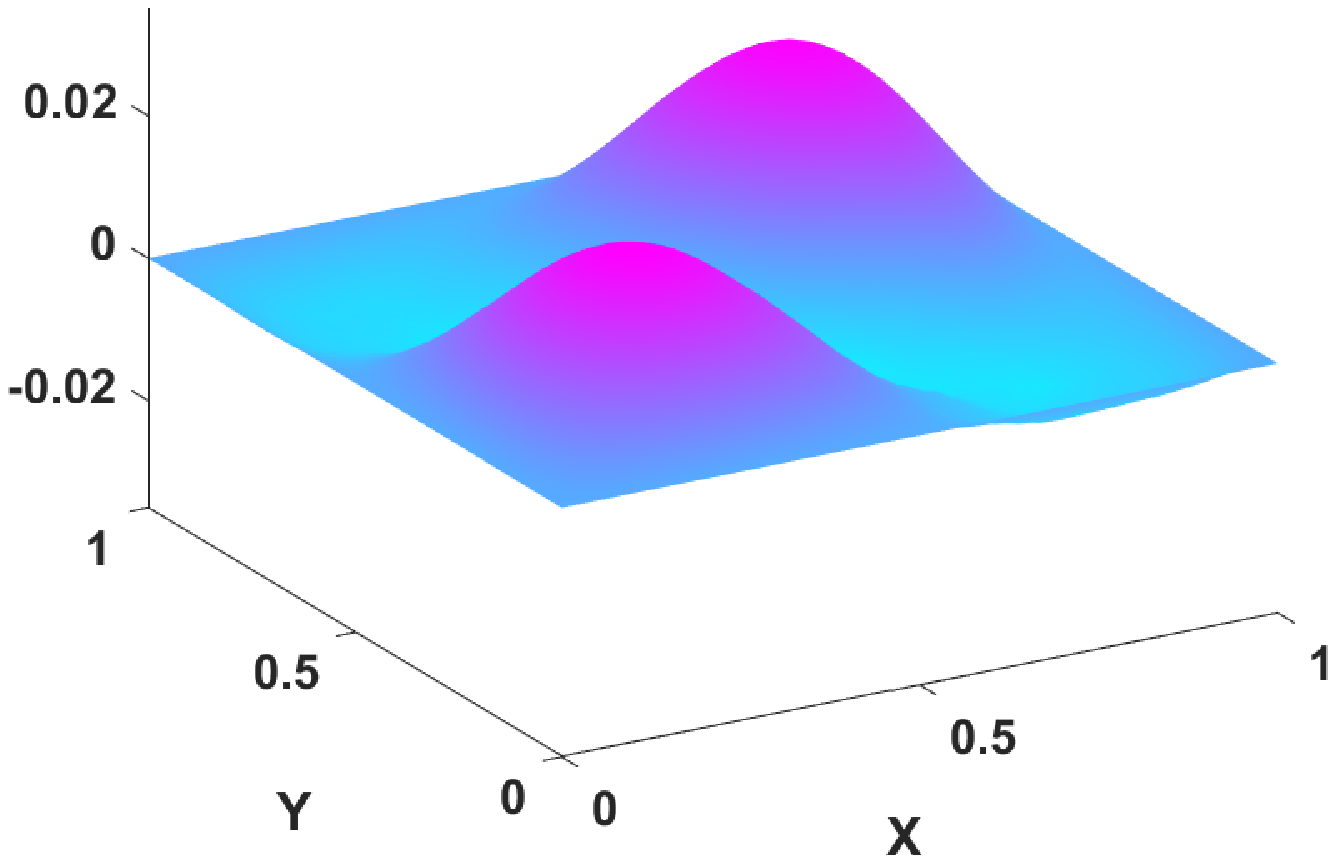}
\caption{$\delta = 5e-3$.}
\end{subfigure}
\begin{subfigure}{.24\textwidth}
\centering
\includegraphics[scale=0.25]{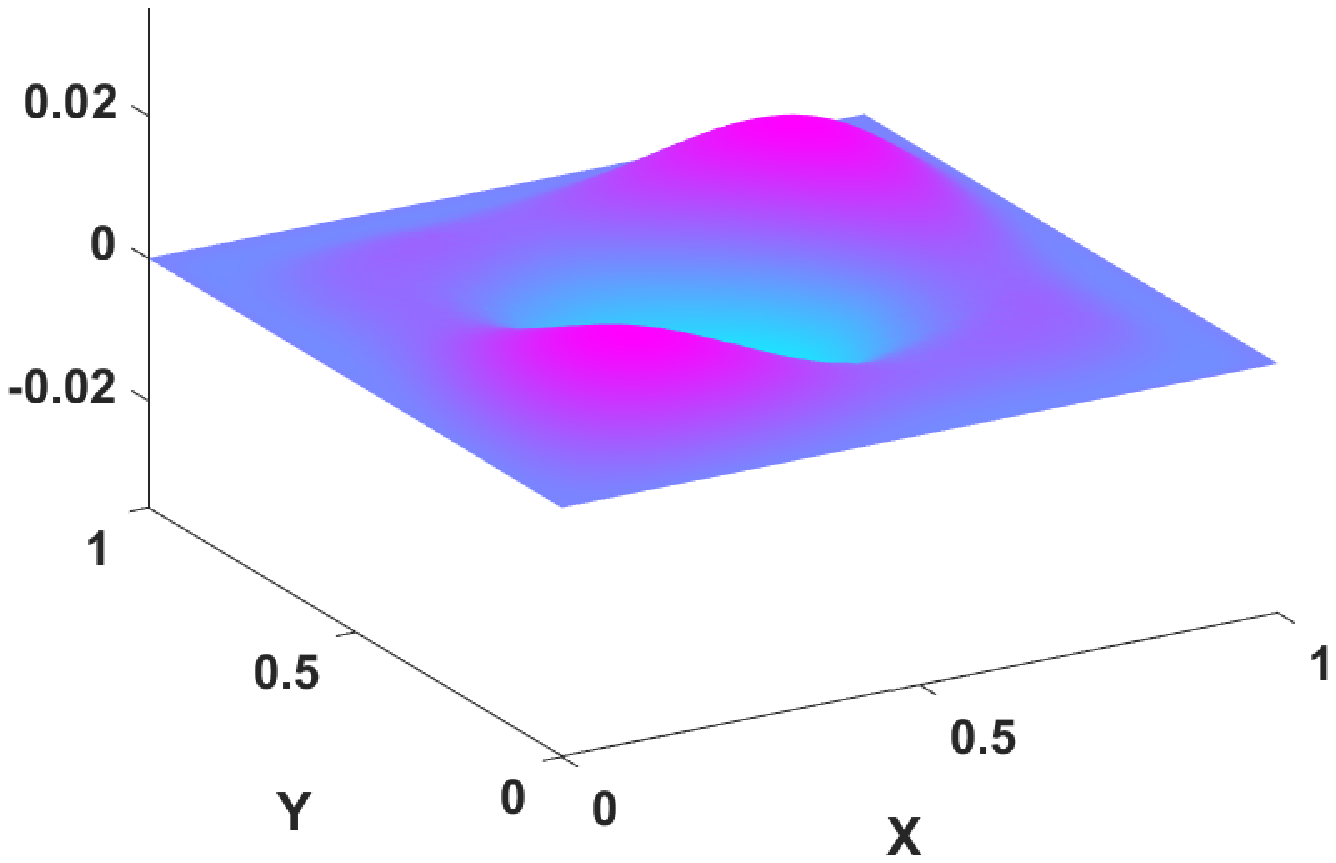}
\caption{$\delta= 2.5e-3$.}
\end{subfigure}
\caption{Example(c): Top left: Exact initial data $b$. The remain three columns are profiles of numerical reconstructions $b_{h,\tau}^\delta$ and their errors, with $h=\sqrt{\delta}/4$, $\tau = \sqrt{\delta}/20$, $\gamma=\sqrt{\delta}/4000$.}
\label{fig:2D:b}
\end{figure}

\section{Concluding remarks}
In this paper, we study the backward diffusion-wave problem,
involving a fractional derivative in time with order $\alpha\in(1,2)$.
From two terminal observations $u(T_1)$ and $u(T_2)$,
we simultaneously determine two initial data $u(0)$ and $u_t(0)$, as well as the solution $u(t)$ for all $t > 0$.
The existence, uniqueness and Lipschitz stability of the backward diffusion-wave problem are theoretically examined
under some mild conditions on $T_1$ and $T_2$.
Then, in case of noisy observations, we apply
quasi-boundary value method to regularize the "mildly" ill-posed problem,
and show the convergence of the regularized solution.
Moreover, in order to numerically solve the regularized problem, we proposed a fully discrete scheme
by using finite element method in space and convolution quadrature in time.
Sharp error bounds of the fully discrete scheme are established in both cases of smooth and nonsmooth data.
Numerical experiments fully support our theoretical findings.

Some interesting questions are still open. First of all, we are interested in the fractional evolution model with time-dependent coefficient, e.g.
\begin{equation} \label{eqn:fde-t}
\partial_t^\alpha u (x,t)+ \nabla \cdot(a(x,t)\nabla u) = f(x,t).
\end{equation}
The current analysis heavily relies on the decay properties of Mittag--Leffler functions, or equivalently the smoothing properties
of solution operators. This stratergy is not directly applicable to the model \eqref{eqn:fde-t}. The direct problem for subdiffusion ($\alpha\in(0,1)$)
and its numerical approximation have been studied in \cite{JinLiZhou:var} by using a perturbation argument.
However,  the backward problem is still unclear and requires some novel approaches.
Besides, we are interested in the backward problem with additional missing information. For example,
 the inverse source problems, determining source term $f(x)$ and fractional order $\alpha$ from terminal observation $u(T)$,
were studied in in \cite{Janno:2018, LiaoWei:2019}. The argument could be extend to the backward problem,
but the error analysis of numerical approximation seems more technical.

\appendix

\section{Proof of Lemma \ref{lem:fully-approx}}

\begin{proof}
{The estimate for $E_{\alpha,1}(-\lambda t_n^\alpha)-F_\tau^n(\lambda)$ follows from the same argument in the proof of \cite[Lemma 4.2]{ZhangZhou:2020}.
Then it suffices to establish a bound for  $t_n E_{\alpha,2}(-\lambda t^\alpha)-\bar F_\tau^n(\lambda)$,
we recall representations \eqref{eqn:FE-LAP} and \eqref{eqn:FE_ht} and derive
\begin{equation*}
\begin{aligned}
|t_nE_{\alpha,2}(-\lambda t_n^\alpha)-\bar{F}_\tau^n(\lambda)| &\le \left|\frac{1}{2\pi i} \int_{\contour\backslash{\Gamma_{\theta,\sigma}^\tau }} e^{zt_n} z^{\alpha-2}(z^\alpha+\lambda)^{-1}dz\right|\\
&\quad +\left|\frac{1}{2\pi i} \int_{\Gamma_{\theta,\sigma}^\tau } e^{zt_n}(z^{\alpha-2}(z^\alpha+\lambda)^{-1}- e^{-z\tau}\delta_\tau(e^{-z\tau})^{\alpha-2} (\delta_\tau(e^{-z\tau})^{\alpha}+\lambda)^{-1} dz\right|\\
&:= I_1+I_2.
\end{aligned}
\end{equation*}
With $\sigma=t_n^{-1}$, the bound for  $I_1$ follows from the direct computation
\begin{equation*}
\begin{aligned}
I_1&\le c\int_{\contour\backslash{\Gamma_{\theta,\sigma}^\tau }} |e^{zt_n}| | z^{\alpha-2}| |(z^\alpha+\lambda)^{-1}| |\d z|
\le c\int_{\pi/(\tau \sin\theta)}^\infty \frac{e^{\rho(\cos\theta)t_n} \rho^{\alpha-2}}{\rho^\alpha} \d\rho\\
&\le ct_n \int_{cn}^\infty e^{-c\rho} \rho^{-2} \d\rho \le ct_n n^{-1}
\end{aligned}
\end{equation*}
and
\begin{equation*}
\begin{aligned}
I_1
&\le c\int_{\pi/(\tau \sin\theta)}^\infty \frac{e^{\rho(\cos\theta)t_n} \rho^{\alpha-2}}{\lambda} \d \rho
\le ct_n (\lambda t_n^\alpha)^{-1} \int_{cn}^\infty e^{-c\rho} \rho^{\alpha-2} d\rho \\
&\le ct_n (\lambda t_n^\alpha)^{-1} n^{-1}\int_{cn}^\infty e^{-c\rho} \rho^{\alpha-1} d\rho
\le ct_n (\lambda t_n^\alpha)^{-1} n^{-1}.
\end{aligned}
\end{equation*}
As a result, we obtain
$
I_1 \le   \frac{c n^{-1}}{(1+\lambda t_n^\alpha)} t_n.
$
Next we turn to  the term $I_2$. According to Lemma \ref{lem:delta}, we have for all $z\in {\Gamma_{\theta,\sigma}^\tau }$,
\begin{equation*}
\begin{aligned}
\left| \frac{z^{\alpha-2}}{z^\alpha+\lambda}- \frac{e^{-z\tau}\delta_\tau(e^{-z\tau})^{\alpha-2}}{\delta_\tau(e^{-z\tau})^{\alpha}+\lambda}\right|
&\le c\tau \frac{|z|^{\alpha-1}}{|z^\alpha+\lambda|}
\end{aligned}
\end{equation*}
Therefore, with $\sigma = t_n^{-1}$, the term $I_2$ can be bounded as
\begin{equation*}
\begin{aligned}
I_2 &\le c\tau \int_{\Gamma_{\theta,\sigma}^\tau }|e^{zt_n}| \frac{ |z|^{\alpha-1}}{|z^\alpha+\lambda|} |dz|\le c\tau\lambda^{-1}(\int_\sigma^\infty e^{\rho\cos\theta t_n}\rho^{\alpha-1}d\rho +\sigma^{\alpha}\int_{-\theta}^\theta d\psi)
\le c\tau(\lambda t_n^{\alpha})^{-1}
\end{aligned}
\end{equation*}
and
\begin{equation*}
\begin{aligned}
I_2&\le c\tau\int_{\Gamma_{\theta,\sigma}^\tau } |e^{zt_n} | |z|^{-1} |dz|\le c\tau (\int_1^\infty e^{\rho\cos\theta}\rho^{-1}d\rho+\int_{-\theta}^\theta  d\psi)\le c\tau.
\end{aligned}
\end{equation*}
Then \eqref{eqn:es-01} follows immediately.

For the second estimate, we note that
\begin{equation*}
\begin{aligned}
t_nE_{\alpha,2}(-\lambda t_n^\alpha) &=   t_n-\frac{\lambda}{2\pi i}\int_\contour e^{zt_n}z^{-2}(z^\alpha+\lambda)^{-1}dz,\\
\bar{F}_\tau^n(\lambda) &= t_n-\frac{\lambda}{2\pi i }\int_{\Gamma_{\theta,\sigma}^\tau } e^{zt_n} e^{-z\tau} \delta_\tau(e^{-z\tau})^{-2}	(\delta_\tau(e^{-z\tau})^{\alpha}+\lambda)^{-1}dz,
\end{aligned}
\end{equation*}
with $n\ge 1$. Then we use the spliiting
\begin{equation*}
\begin{aligned}
\lambda^{-1}|t_nE_{\alpha,2}(-\lambda t_n^\alpha)-\bar{F}_\tau^n(\lambda)|&\le \left|\frac{1}{2\pi i} \int_{\contour\backslash{\Gamma_{\theta,\sigma}^\tau }} e^{zt_n} z^{-2}(z^\alpha+\lambda)^{-1} dz\right|\\
& + \left|\frac{1}{2\pi i}\int_{{\Gamma_{\theta,\sigma}^\tau }} e^{zt_n}[z^{-2}(z^\alpha+\lambda)^{-1} - e^{-z\tau} \delta_\tau(e^{-z\tau})^{-2}(\delta_\tau(e^{-z\tau})^{\alpha}+\lambda)^{-1}]dz\right|\\
&: = I_1+I_2.
\end{aligned}
\end{equation*}
According to Lemma \ref{lem:delta} we have for all $z\in\Gamma_{\theta,\sigma}^\tau$ ,
\begin{equation*}
\begin{aligned}
I_1
& \le c \int_{\contour\backslash{\Gamma_{\theta,\sigma}^\tau}}|e^{zt_n}| |z|^{-\alpha-2} |dz |\le c\int_{\pi/(\tau\sin\theta)} e^{\rho\cos\theta t_n} \rho^{-\alpha-2}d\rho \\
& \le ct_n^{\alpha+1} \int_{cn}^\infty e^{-c\rho} \rho^{-\alpha-2}d\rho \le ct_n^{\alpha+1} n^{-3} \int_0^\infty e^{-c\rho}\rho^{-\alpha+1 }d\rho \le ct_n^{\alpha-2} \tau^3.
\end{aligned}
\end{equation*}
And also we have
\begin{equation*}
|z^{-2}(z^\alpha+\lambda)^{-1} - e^{-z\tau} \delta_\tau(e^{-z\tau})^{-2}(\delta_\tau(e^{-z\tau})^{\alpha}+\lambda)^{-1}| \le c\tau |z|^{-\alpha-1},
\end{equation*}
and therefore with $\sigma = t_n^{-1}$, we have the bound for $n\geq 1$
\begin{equation*}
\begin{aligned}
I_2 &\le c\tau \int_{\Gamma_{\theta,\sigma}^\tau}|e^{zt_n}| |z|^{-\alpha-1} |dz| \le c\tau\left(\int_\sigma^\infty e^{-c\rho t_n} \rho^{-\alpha-1}d\rho +\sigma^{-\alpha} \int_{-\theta}^\theta d\psi\right)  \le c\tau t_n^\alpha .
\end{aligned}
\end{equation*}
This completes the proof of \eqref{eqn:es-02}.
}
\end{proof}

\bibliographystyle{abbrv}

\end{document}